\documentclass[11pt]{amsart}
\usepackage{tikz,amsthm,amsmath,amstext,amssymb,amscd,epsfig,euscript, mathrsfs,pspicture,multicol,graphpap,graphics,graphicx,times,enumerate,subfig,sidecap,wrapfig,color}
\usepackage{mathtools}
\usepackage{amsmath}
\usepackage{amssymb}
\usepackage{pgf}
\usepackage{stackrel}
\usepackage[colorlinks,citecolor=red]{hyperref}

\usepackage{tikz,amsthm,amsmath,amstext,amssymb,amscd,epsfig,euscript, mathrsfs, dsfont,pspicture,multicol, mathtools,graphpap,graphics,graphicx,times,enumerate,subfig,sidecap,wrapfig,color}

\usepackage{ hyperref}
\usepackage{enumerate}
\usepackage{esint}
\usepackage{verbatim} 

 \numberwithin{equation}{section}

\def\Q{{\mathbb{Q}}}

\def\cB{{\mathcal{B}}}

\def\cF{{\mathcal{F}}}

\def\cS{{\mathcal{S}}}
\def\WW{{\mathcal{W}}}
\def\EE{{\mathcal{E}}}
\def\RR{{\mathcal{R}}}
\def\NN{{\mathcal{N}}}
\def\CC{{\mathcal{C}}}
\def\cU{{\mathcal{U}}}

\newcommand{\ds}{\displaystyle }

\def\Z{{\mathbb{Z}}}

\def\cF{{\mathscr{F}}}

\def\cM{{\mathcal{M}}}

\def\ve{\varepsilon}

\renewcommand{\d}{{\partial}}
\def\lec{\lesssim}

\newcommand{\rf}[1]{{(\ref{#1})}}

\def\wt{\widetilde}

\def\Rn1{\mathbb{R}^{n+1}}
\def\rrn{\mathbb{R}^{n+1}}

\def\oom{\overline \Omega} 					

\DeclareMathOperator{\diam}{diam}
\def\Cap{\textup{Cap}} 					
\newcommand{\DD}{{\mathcal D}}

\newcounter{eps}

\newcommand{\sss}{{\mathsf {Stop}}}
\newcommand{\ttt}{{\mathsf {Top}}}

\newcommand{\tree}{{\rm Tree}}

\newcommand{\wh}[1]{{\widehat{#1}}}

\def\BMO{\mathop\mathrm{BMO}} 					
\def\Lip{\mathop\mathrm{Lip}} 						

\def\dist{\textup{dist}} 						
\def\supp{\mathop\mathrm{supp}}					
\def\loc{\mathop\mathrm{loc}}						
\DeclareMathOperator*{\essinf}{ess\,inf}				
\renewcommand{\div}{\mathop\mathrm{div }}			

\def\Xint#1{\mathchoice
{\XXint\displaystyle\textstyle{#1}}%
{\XXint\textstyle\scriptstyle{#1}}%
{\XXint\scriptstyle\scriptscriptstyle{#1}}%
{\XXint\scriptscriptstyle\scriptscriptstyle{#1}}%
\!\int}
\def\XXint#1#2#3{{\setbox0=\hbox{$#1{#2#3}{\int}$ }
\vcenter{\hbox{$#2#3$ }}\kern-.58\wd0}}

\def\avint{\Xint-}
\theoremstyle{plain}
\newtheorem{theorem}{Theorem}

\newtheorem{corollary}[theorem]{Corollary}

\newtheorem{lemma}[theorem]{Lemma}

\newtheorem{proposition}[theorem]{Proposition}

\theoremstyle{definition}

\newtheorem{remark}[theorem]{Remark}

\numberwithin{equation}{section}
\numberwithin{theorem}{section}

\newtheorem*{claim*}{Claim}

\def\HH{\mathcal{H}}
\newcommand{\vv}{\vspace{2mm}}
\newcommand{\vvv}{\vspace{4mm}}

\newcommand{\dv}{\mathop{\rm div}}
\def\R{\mathbb{R}}
\def\pom{{\partial\Omega}}
\def\vphi{\varphi}
\def\om{\Omega}

\def\hm{\omega}

  \newtheorem*{prob}{\bf Problem}


\begin{document}

\title[Regularity problem in rough domains]{The regularity problem for the Laplace equation\\ in rough domains}

\author[Mihalis Mourgoglou]{Mihalis Mourgoglou}
\address{Departamento de Matem\'aticas, Universidad del Pa\' is Vasco, UPV/EHU, Barrio Sarriena s/n 48940 Leioa, Spain and\\
IKERBASQUE, Basque Foundation for Science, Bilbao, Spain.}
\email{michail.mourgoglou@ehu.eus}
\author[Xavier Tolsa ]{Xavier Tolsa}
\address{ICREA, Barcelona\\
Dept. de Matemàtiques, Universitat Autònoma de Barcelona \\
and Centre de Recerca Matemàtica, Barcelona, Catalonia.}
\email{xtolsa@mat.uab.cat}
\thanks{M.M. was supported  by IKERBASQUE and partially supported by the grant PID2020-118986GB-I00 of the Ministerio de Econom\'ia y Competitividad (Spain), and by  IT-1247-19 (Basque Government). X.T. is supported by the European Research Council (ERC) under the European Union's Horizon 2020 research and innovation programme (grant agreement 101018680). Also partially supported by MICINN (Spain) under the grant PID2020-114167GB-I00, the María de Maeztu Program for units of excellence (CEX2020-001084-M), and 2021-SGR-00071 (Catalonia).
}



\newcommand{\mih}[1]{\marginpar{\color{red} \scriptsize \textbf{Mih:} #1}}
\newcommand{\xavi}[1]{\marginpar{\color{blue} \scriptsize \textbf{Xavi:} #1}}

\maketitle

\begin{abstract}
Let $\Omega \subset \mathbb{R}^{n+1}$, $n\geq 2$, be a bounded open and connected set  satisfying the corkscrew condition with uniformly $n$-rectifiable boundary. In this paper we study the connection between the solvability of $(D_{p'})$, the Dirichlet problem for the Laplacian with boundary data in $L^{p'}(\partial \Omega)$, and $(R_{p})$ (resp. $(\widetilde R_{p})$), the regularity problem for the Laplacian with boundary data in the Haj\l asz Sobolev space $W^{1,p}(\partial \Omega)$ (resp. $\widetilde W^{1,p}(\partial \Omega)$, the usual Sobolev space in terms of the tangential derivative), where $p \in (1,2+\ve)$ and $1/p+1/p'=1$. Our main result shows that $(D_{p'})$ is solvable if and only if so is $(R_{p})$. Under additional geometric assumptions (two-sided local John condition or weak Poincar\'e inequality on the boundary), we prove that $(D_{p'}) \Rightarrow (\widetilde R_{p})$. In particular, we deduce that in bounded chord-arc domains (resp. two-sided chord-arc domains) there exists $p_0 \in (1,2+\ve)$ so that $(R_{p_0})$ (resp. $(\widetilde R_{p_0})$) is solvable. We also extend the results to unbounded domains with compact boundary and show that in two-sided corkscrew domains with $n$-Ahlfors-David regular boundaries the single layer potential operator is invertible  from $L^p(\pom)$ to the inhomogeneous Sobolev space $ W^{1,p}(\pom)$.  Finally, we provide a counterexample of a chord-arc domain $\Omega_0 \subset \R^{n+1}$, $n \geq 3$, so that $(\wt R_p)$ is not solvable for any  $p \in [1, \infty)$. 
\end{abstract}


\tableofcontents

\section{Introduction}\label{sec:intro}

In this paper we study the equivalence between solvability of the {$L^{p'}$-Dirichlet} and solvability of the $L^{p}$-regularity problem  for the Laplace operator in corkscrew domains with uniformly rectifiable boundaries. 
In particular, our main result  solves an old problem posed by  Kenig in 1991 in \cite[Problem 3.2.2]{Ke}, which was reintroduced  by Toro  at the ICM   in 2010 \cite[Question 2.5]{To}, and can be stated as follows:

\begin{prob}
In a bounded chord-arc domain $\Omega \subset \R^{n+1}$, $n \geq 2$, does there exist $p>1$ such that the regularity problem for the Laplacian with  boundary data in the Sobolev space $W^{1,p}(\pom)$ is solvable?\footnote{The result is known to hold for $n=1$. See the discussion on p.\! 7.} If so, are the layer potentials invertible in appropriate $L^p$ spaces for such $p>1$?
\end{prob}

Up to now, the problem above was only solved for Lipschitz domains,  by means of $L^2$ Rellich type inequalities, which simultaneously addressed the Neumann problem. These inequalities are not available in more general domains, such as chord-arc domains (even assuming quantitative connectedness
conditions both for the domain and its boundary). So it is a major problem in the area how to solve both the regularity and the Neumann problems in chord-arc and more general domains. In the current paper, we solve the case of the regularity problem\footnote{The Neumann problem in chord-arc or more general domains is still open.}, even in greater geometric generality than just chord-arc domains. 
 The main ideas used in this paper are {\em not} based in the developments 
of the last years  that have allowed to make substantial progress in problems connected with the $L^p$ solvability of the Dirichlet problem
(see \cite{AHMMT}, \cite{AHM3TV}, \cite{HMM}, \cite{HLMN}, \cite{MT}, for example). Instead,  our approach relies on new ideas, such as the introduction of an ``almost harmonic extension'' of the boundary Dirichlet data, which is the correct analogue of the Varopoulos' extension for one “smoothness level” up (associated with boundary functions in the $W^{1,p}$ Sobolev space instead of $L^p$ or $\BMO$).  Its construction is based on a delicate corona decomposition of the domain in mutually disjoint bounded starlike Lipschitz domains along with a ``small" buffer region, and the solvability of the regularity problem for the Laplacian in $L^p$  in such domains (proved by Dahlberg and Kenig \cite{DaKe}  in 1987).  
It is worth noting that all duality methods that had previously appeared in the literature in connection to this problem fail in this general setting.  In fact,  for chord-arc domains,  David and Jerison \cite{DJe} had already proved that the  $L^{p'}$-Dirichlet problem for the Laplacian  is solvable  in 1990 but  even in  such domains the regularity problem remained open until now.

Let us remark that the Sobolev space  for which the problem was originally posed was the class of $L^p(\pom)$ functions whose tangential gradient is in $L^p(\pom)$ (denote it by $\wt W^{1,p}(\pom)$) and, in particular, Kenig referred to the paper of Semmes \cite{Se} for the study of such Sobolev spaces. It turns out that this is not the right space to solve the regularity problem  in so general domains. Indeed in Section \ref{sec:counterexample}, we show that there exists  a chord-arc domain $\Omega_0 \subset \R^{n+1}$, {$n \geq 3$,} such that for any $ p\in [1, \infty)$ there is a Lipschitz function $f$ on $\pom_0 $  so that the solution of the continuous Dirichlet  problem in $\Omega_0 $ with boundary data $f$  does not satisfy the regularity estimate \eqref{eq:main-est-reg2}, and thus, the regularity problem in terms of $\wt W^{1,p}(\pom_0 )$  is not solvable. Instead, we solve the problem for boundary data in the so-called Haj\l asz Sobolev space $W^{1,p}(\pom)$ (see Section \ref{secsobolev} for the precise definition of this space). We would like to highlight that, for example, in {\em two-sided} chord-arc domains, we have\begin{equation}\label{eq:comp-norm}
\|f\|_{\dot  W^{1,p}(\pom)} \approx \|\nabla_t f\|_{L^{p}(\pom)}, \,\,\quad \textup{for any}\,f\in \textup{Lip}(\pom),
\end{equation}
{where $\|\cdot\|_{\dot  W^{1,p}(\pom)}$ is the associated  Haj\l asz seminorm and $\nabla_t$ stands for the tangential gradient.}
 In fact, a more general result is true and in order for the $\lesssim$ part of \eqref{eq:comp-norm} to hold, it seems that one should impose scale invariant  connectivity conditions, {either from (certain) interior points in $\Omega$ and $\R^{n+1} \setminus \overline \Omega$  to boundary points (such as the two-sided local John condition), or a quantitative  connectivity condition of the boundary (such as a weak Poincar\'e inequality on $\pom$, which, in turn, implies quasi-convexity of $\pom$). 
 }

\vv

We introduce some definitions and notations.
Recall that a set $E\subset\R^{n+1}$ is called  $n$-AD-{\textit {regular}} (or just AD-regular or Ahlfors-David regular) if there exists some
constant $C_{0}>0$ such that
$$C_0^{-1}r^n\leq \HH^n(B(x,r)\cap E)\leq C_0\,r^n\quad \mbox{ for all $x\in E$ and $0<r\leq \diam(E)$.}$$
Following \cite{JeK2}, we say that $\Omega \subset \R^{n+1}$ satisfies the corkscrew condition, or that it is a corkscrew open set or domain if
there exists some $c>0$ such that for all
$x\in\pom$ and all $r\in(0, 2\diam(\Omega))$ there exists a ball $B\subset B(x,r)\cap\Omega$ such that
$r(B)\geq c\,r$. We say that $\Omega$ is two-sided corkscrew if both $\Omega$ and $\R^{n+1}\setminus
\overline\Omega$ satisfy the corkscrew condition.

\vv

Let $\Omega\subset\R^{n+1}$ be an open set and  set $\sigma:=\HH^n|_\pom$ to be its surface measure. For $\alpha>0$ and $x \in \pom$, we define the {\it cone with vertex $x$ and aperture $\alpha>0$} by
\begin{equation}\label{eqconealpha}
\gamma_\alpha(x)=\{ y \in \Omega: |x-y|<(1+\alpha)\dist(y, \pom)\}
\end{equation}
and the {\it non-tangential maximal function operator} of a measurable function $u:\Omega \to \R$ by
\begin{equation}\label{eqNalpha}
\NN_\alpha(u)(x):=\sup_{y \in \gamma_\alpha (x)} |u(y)|, \,\,x\in \pom.
\end{equation}
 If $\pom$ is AD-regular, it holds that  $\| \NN_\alpha(u)\|_{L^p(\sigma)} \approx_{\alpha,\beta} \| \NN_\beta(u)\|_{L^p(\sigma)}$ {for all $\alpha,\beta>0$} and so, from now on, we will only write $\NN$ dropping the dependence on the aperture. This is a well know result due to Fefferman and Stein  in $\rrn$ (see \cite[Lemma 1]{FS}). For the proof in the context of domains with AD-regular boundaries, see  \cite[Proposition 2.2]{HMT}. 
 
 \vv

We are now ready to state the definitions of the $L^p$-Dirichlet and the $L^p$-regularity problem in~$\Omega$:
\begin{itemize}
\item In a domain $\Omega  \subset \R^{n+1}$ with $n$-AD-regular boundary, we say that {\it the Dirichlet problem is solvable in $L^p$} for the Laplacian  (write $(D_{p})$ is solvable)  if there exists some constant $C_{D_p}>0$  such that, for any $f\in C_c(\pom)$,  the solution $u:\Omega\to\R$ of the continuous Dirichlet problem for the Laplacian in $\Omega$ with boundary data $f$\footnote{This is given as the harmonic function constructed by the Perron method,  say $H_f(x)$, which in domains with $n$-AD-regular boundary and for $f \in C_c(\pom)$, satisfies $\lim_{x \to \xi} H_f(x)=f(\xi)$ for every $\xi \in \pom$.  This is true in unbounded domains as well (see e.g.  \cite[Section 2.4]{AGMT}).  Nevertheless,  in bounded domains,  $H_f \in W^{1,2}(\om)$ (see e.g. \cite[Corollary 9.29]{HKM}).}  satisfies
\begin{equation}\label{eq:ntDirichlet}
\|\NN(u)\|_{L^p(\sigma)} \leq C_{D_p}\,\|f\|_{L^p(\sigma)}.
\end{equation}
\item Assume that $\Omega \subset \R^{n+1}$ is either a bounded domain or that it has unbounded boundary $\pom$.  If  $\pom$ is $n$-AD-regular,  then we say that {\it the regularity problem is solvable in $L^p$} for the Laplacian   (write $(R_{p})$ is solvable) if  there exists some constant $C_{R_p}>0$  such that, for any compactly supported  Lipschitz function 
$f:\pom\to\R$,  the solution $u:\Omega\to\R$ of the continuous
Dirichlet problem for the Laplacian in $\Omega$ with boundary data $f$ satisfies
\begin{equation}\label{eq:main-est-reg}
\|\NN(\nabla u)\|_{L^p(\sigma)} \leq C_{R_p}\|f\|_{\dot W^{1,p}(\sigma)}.
\end{equation}
Here (and throughout all the paper) we denote by $W^{1,p}(\sigma)$ and $\dot W^{1,p}(\sigma)$ the Haj\l asz  Sobolev spaces defined on the metric measure space $(\pom, \sigma)$. The definition of $(R_p)$ can be extended so that the boundary data are functions in $W^{1,p}(\sigma)$ that are continuous on $\pom$ (see Theorem \ref{thrm:continuous-regularity}).
\end{itemize}
As discussed earlier, one may define the regularity problem in domains with locally rectifiable boundaries using the tangential derivative with respect to the boundary at a boundary point, which for Lipschitz functions exists for $\sigma$-a.e.\ point of the boundary. 
 In that case the only difference in the definition of the regularity problem is that we ask for the estimate  
\begin{equation}\label{eq:main-est-reg2}
\|\NN(\nabla u)\|_{L^p(\sigma)} \leq \wt{C}_{R_p}\| \nabla_t f\|_{ L^{p}(\sigma)}
\end{equation}
and we  write that $(\wt R_p)$ is solvable for the Laplacian. This definition is customary in nicer domains such as Lipschitz domains. 

\vv

\begin{remark}\label{rem:epsilon0}
Recall that in bounded starlike Lipschitz domains, it was proved in \cite{DaKe} (see \cite{V} for $p=2$) that there exists $\ve \in (0,1)$ such that $(\wt R_p)$ is solvable for the Laplacian for any $1<p<2+\ve$ with $\ve$ and $\wt{C}_{R_p}$ depending only on dimension and the Lipschitz character of the domain. Therefore, one can find $\ve_0 \in (0,1)$ and $\wt{C}_{R_p}>0$ such that, in every starlike Lipschitz domain with the same  Lipschitz character, $(\wt R_p)$ is solvable for the Laplacian for any $1<p<2+\ve_0$  with the same constants $\wt{C}_{R_p}>0$.
\end{remark}

\vv

In the sequel, we will only work with the Laplacian and so, without any confusion,  we will  write that a boundary value problem is solvable without specifying the operator. We will also assume that  $n \geq 2$.

\vv

Next we wish to recall what means that a metric space $\Sigma$ supports Poincar\'e inequality. 
Given an interval $I\subset\R$, any continuous $\gamma:I \to \Sigma$ is called {\it path}. A path of finite length is called {\it rectifiable path}.
 Let $f: \Sigma \to \R$ be an arbitrary real-valued function. We say that a Borel measurable function $g: \Sigma \to \R$ is an {\it upper gradient} of $f$ if for all compact rectifiable paths $\gamma$ the following inequality holds:
\begin{equation}\label{eq:upper-grad}
|f(x)-f(y)| \leq \int_{\gamma} g \,d\HH^1,
\end{equation}
where $x, y\in\Sigma$ are the endpoints of the path.

We say that $(\Sigma,\sigma)$ supports a {\it weak} $(1,p)$-{\it Poincar\'e inequality} if there exist constants $C\geq 1$ and $\Lambda \geq 1$ so that for every ball $B$ centered at $\Sigma$ with radius $r(B) \in (0, \diam \Sigma)$ and  every  pair $(f,g)$, where $f \in L^1_{loc}(\sigma)$ and $g$ is an upper gradient of $f$, it holds
\begin{equation}\label{eq:Poincare}
 \avint_{B} \left| f(x) - m_{\sigma,B}(f)\right| \,d\sigma(x) \leq C r(B) \left( \avint_{\Lambda B} |g(x)|^p\,d\sigma(x) \right)^{1/p},
\end{equation}
where we denoted 
$m_{\sigma,B}(f)= \avint_B f(y) \,d\sigma(y).$
In the specific case that $\Sigma \subset \R^d$ is $n$-AD-regular and $\sigma=\HH^n|_\Sigma$,
 Azzam  \cite{Az}  {recently}  showed that if $\Sigma$ supports a weak $(1,n)$-Poincar\'e inequality, then $\Sigma$ is uniformly $n$-rectifiable. 

Let us now state our main result.
\begin{theorem}\label{teomain}
Let $\Omega\subset\R^{n+1}$ be a bounded corkscrew domain with $n$-AD-regular boundary. 
There exists $\ve_0>0$ (depending just on $n$) such that, given $p \in (1, 2+\ve_0)$, if $(D_{p'})$ is solvable (where $1/p+1/p'=1$), then $(R_p)$ is solvable. If, in addition, either $\pom$ admits a weak $(1,p)$-Poincar\'e inequality or $\Omega$ satisfies the two-sided local John condition, then $(\wt R_p)$ is solvable as well. 
\end{theorem}

The two-sided local John condition is a connectivity property of $\Omega$ introduced in \cite{HMT} which is defined in Section \ref{subsecwhitney}. Remark also that, after the completion of this paper, 
in a subsequent joint work with Bruno Poggi \cite{MPT}, we realized that the condition  that $p \in (1, 2+\ve_0)$ can be eliminated and require only that $p\in (1,\infty)$. To do so, we just have to replace the corona decomposition of $\Omega$ into Lipschitz subdomains from Section \ref{secorona} by a decomposition into sufficiently flat Lipschitz subdomains, arguing as in \cite{MPT}.

Remark that the part of Theorem \ref{teomain} dealing with the  $(\wt R_p)$ solvability follows from the implication $(D_{p'})$ $\Rightarrow$  $(R_p)$ and the following result.

\begin{lemma}\label{lemtec99}
Let $\Omega\subset\R^{n+1}$ be a bounded domain with uniformly $n$-rectifiable boundary. 
Suppose either that 
\begin{itemize}
\item[(a)]  $\pom$ admits a weak $(1,p)$-Poincaré inequality,  or
\item[(b)] $\Omega$ satisfies the two sided local John condition.\footnote{In a recent work \cite{TaTo} of O. Tapiola and the second named author, which was motivated by the results of the present manuscript and completed after writing this paper,  it is proven that if a domain satisfies the two-sided local John condition and has $n$-AD-regular boundary, then it is a two-sided chord-arc domain. See Section \ref{subsecwhitney} for the precise definition. Moreover, in the same work is shown that the boundaries of such domains support a weak $(1,p)$-Poincar\'e inequality for any $p\geq1$.}.
\end{itemize} 
For any Lipschitz function $f:\pom\to\R$, we have
\begin{equation}\label{eqcomparp99}
\|\nabla_{H,p} f\|_{L^p(\sigma)} \approx \|\nabla_t f\|_{L^p(\sigma)}.
\end{equation}
\end{lemma}

\vv

The second result of our paper 
 extends the solvability of $(R_p)$ to general boundary data in the  Haj\l asz Sobolev space $W^{1,p}(\pom)\equiv W^{1,p}(\sigma)$ and shows the well-posedness of the problem in this space.

\begin{theorem}\label{thm:1.3}
Let  $p\in (1,\infty)$ and let $\Omega\subset\R^{n+1}$ be a bounded corkscrew domain with $n$-AD-regular boundary. 
\begin{itemize}
\item[(a)] If $(R_p)$ is solvable, then for any $f \in  W^{1,p}(\pom)$ there exists a harmonic function $u$ in $\Omega$ such that $\|\NN(\nabla u)\|_{L^p(\sigma)} \leq C \|f\|_{\dot W^{1,p}(\pom)}$ and $u \to f$ non-tangentially $\sigma$-a.e. on $\pom$.

\item[(b)] Suppose that $\Omega$ satisfies also the weak local John condition.
Let $u:\Omega\to\R$ be a harmonic function which has a vanishing non-tangential limit for $\sigma$-a.e.\ $x\in\pom$ and such that $\|\NN(\nabla u)\|_{L^p(\sigma)}<\infty$. Then $u$ vanishes identically in $\Omega$.
\end{itemize}
\end{theorem}

\vv

Above, $\|\cdot\|_{\dot  W^{1,p}(\pom)}$ stands for  the  Haj\l asz seminorm in $W^{1,p}(\pom)$.
For the definition of the aforementioned weak local John condition, see Section \ref{subsecwhitney}

Let us highlight that in Theorem \ref{lemunbounded} we  extend Theorems \ref{teomain} and \ref{thm:1.3} to the case of unbounded domains with compact boundary (the statements are modified accordingly).  We also deal with the question of invertibility of the  single layer potential in our paper. Recall that
this operator,  which plays an important role in the study of the Dirichlet and regularity problems, is defined by
$$
\cS f(x)= \int_\pom \EE(x-y) f(y) \,d\sigma(y), \,\,\quad x \in\Omega,
$$
where $\EE$ stands for the fundamental solution of the Laplacian.
We prove invertibility of the single layer potential  from $L^p(\pom)$ to $ W^{1,p}(\pom)$ in the next theorem.

\begin{theorem}\label{thm:1.4} \label{teo-invert}
Let $p\in (1,\infty)$ and let $\Omega\subset\R^{n+1}$ be a bounded two-sided corkscrew domain with $n$-AD-regular boundary such that $\R^{n+1} \setminus \overline{\Omega}$ is connected.
Suppose either that $\Omega$
satisfies the two-sided 
local John condition or that $\pom$ supports a weak $(1,p)$-Poincar\'e inequality.
If $(D_{p'})$ is solvable  for  both
 $\Omega$ and  $\R^{n+1}\setminus \overline\Omega$, then
$\cS: L^p(\pom)\to  W^{1,p}(\pom)$ is bounded and invertible.
\end{theorem}

\vv

The last main result of the paper 
 goes in the converse direction. It shows that one can deduce  solvability of the Dirichlet problem
 from the regularity problem. The precise result is the following.

\begin{theorem}\label{propoconverse}
Let $\Omega\subset\R^{n+1}$ be a bounded domain with $n$-AD-regular boundary.    If $(R_p)$ is solvable  for some $p \in (1, \infty)$, then $(D_{p'})$ is solvable, where $1/p+1/p'=1$.
\end{theorem}

\vv

Some remarks are in order.
 \begin{itemize}
 \item[(i)] Recall that under the assumptions of  Theorem \ref{teomain},  solvability of  $(D_{p'})$ implies that $\pom$ is uniformly $n$-rectifiable (see \cite{HLMN} and \cite{MT}). 
 
 \item[(ii)] From Theorems \ref{teomain} and \ref{propoconverse}, for bounded corkscrew domains with $n$-AD-regular boundary and $1<p< 2+\ve_0$, we have
$(D_{p'}) \Leftrightarrow (R_p) $. It also holds $(\wt R_p)\Rightarrow (R_p)$, by the (easy) estimate \rf{eqfac71}. Under the additional assumption that $\Omega$ is two-sided local John or supports a $p$-Poincar\'e inequality, the converse implication $(R_p)\Rightarrow (\wt R_p)$ holds by Lemma \ref{lemtec99}.

 \item[(iii)]   Although our results are stated in bounded domains, the  constant $C_{R_p} $ (resp. $C_{D_p}$) we obtain in Theorem  \ref{teomain}  (resp. Theorem \ref{propoconverse}) is independent of the diameter of $\pom$. In fact, they only depend on $C_{D_p} $ (resp. $C_{R_p}$), the constants that appear in the definitions of AD-regularity and uniform rectifiability, the corkscrew  constants, {and other geometric parameters related to the two-sided local John condition or the supported weak Poincaré inequality when these conditions are required.}
 

\end{itemize}

 Combining  Theorems \ref{teomain} and  \ref{propoconverse} and
 the extrapolation  of solvability of $(D_{p'})$ (see Theorem \ref{teo9.2} and Remark
\ref{rem9.3} below),
we obtain the following extrapolation of solvability of the regularity problem:

\begin{corollary}\label{cor:extrapol}
Let $\Omega\subset\R^{n+1}$ be a bounded corkscrew domain with $n$-AD-regular boundary.
 Suppose that  $(R_p)$ 
 is solvable for some $p \in (1, 2+\ve_0)$. Then $(R_s)$  is solvable for all  $s \in (1,p+\ve_0']$, for some $\ve_0'>0$.
\end{corollary}

Theorems \ref{teomain} and  \ref{propoconverse} in concert  with \cite[Theorem 1.3]{AHMMT} give the following geometric implications:

\begin{corollary}\label{cor:geom}
Let $\Omega\subset\R^{n+1}$ be a bounded corkscrew domain with $n$-AD-regular boundary
and let $\ve_0$ be as in Theorem \ref{teomain}.
\begin{itemize}
\item[(a)] If $\Omega$ has IBPCAD\footnote{IBPCAD stands for interior big pieces of chord-arc domains (see Subsection \ref{subsecwhitney} for the exact definition)}, 
 then there exists {$p \in(1, 2+\ve_0)$} such that $(R_p)$ is solvable. 
\item[(b)]  If $\Omega$ has IBPCAD, 
then there exists {$p \in(1, 2+\ve_0)$} such that if $\pom$ admits a weak $(1,p)$-Poincar\'e inequality, then both $(R_p)$ and $(\tilde R_p)$ are solvable. 
\item[(c)] If 
$(R_p)$ is solvable for some $p \in (1, \infty)$, then $\Omega$ has IBPCAD. 

\end{itemize}
\end{corollary}


\vv

Solvability of $(D_{p'})$ and  $(R_{p})$\footnote{We will only discuss  results related to Dirichlet and regularity problems and not  Neumann.} has been a subject of considerable research the last 45 years. {In  $\R^2$, already in 1936, Lavrentiev \cite{Lav} proved that in simply connected chord-arc domains, the harmonic measure is in $A_\infty(\sigma)$ and so there exists $p>1$ such that $(D_{p'})$, $1/p+1/p'=1$, is solvable for the Laplacian. Then Jerison and Kenig \cite{JeK3} showed the duality $(D_{p'}) \Leftrightarrow (\wt R_p) $ in such domains and so $(\wt R_p) $ is also solvable for the Laplacian (remark that in planar  simply connected chord-arc domains the comparability \rf{eqcomparp99} holds and thus $(\wt R_p)\Leftrightarrow (R_p)$).}
 Moreover, Jerison \cite{Je} proved that for every $p>1$, there exists a chord-arc domain such that $(D_{p'})$ is not solvable and thus, neither is $(R_{p})$. For more details we refer to the discussion in \cite[pp. 115-116]{Ke}. 
 In higher dimensions, the study of $(D_{p'})$ for the Laplacian in Lipschitz domains was initiated by Dahlberg \cite{Da1, Da2} who  proved that it is solvable for $p \geq 2$. He also showed that in $C^1$ domains the range extends to $1<p<\infty$. Subsequently, a big breakthrough was made by Jerison and Kenig \cite{JeK1} who obtained  $(D_{2})$ for elliptic operators in divergence form with smooth and $L^\infty$ coefficients. The importance of their proof, even for the Laplace operator, is that their method was based on a new identity, the so-called Rellich identity. This is a clever integration by parts argument that allows to compare the tangential and the normal derivative of the solution on the boundary in the following sense:
\begin{equation}\label{eq:2-sided rellich}
\| \partial_\nu u \|_{L^2(\pom) }    \approx \| \nabla_t u \|_{L^2(\pom)}.
\end{equation}
{In chord-arc domains, David and Jerison \cite{DJe} extended the work of Lavrentiev and, relying on a geometric construction, showed that harmonic measure is $A_\infty$ and thus $(D_{p'})$ is solvable for some $p >1$. Recently the authors along with Azzam, Hofmann, and Martell \cite{AHMMT} gave a  geometric characterization of the corkscrew open sets that have AD-regular boundaries in which $(D_{p'})$ is solvable for some $p>1$. Those domains were identified to be the ones that have IBPCAD.}

A common method to solve boundary value problems (especially regularity and Neumann) is by proving invertibility of the appropriate  layer potential operators.  In $C^1$ domains this had already been established by Fabes, Jodeit, and Rivi\`ere \cite{FJR} in $L^p(\pom)$ and $W^{1,p}(\pom)$ for all  $p \in (1, \infty)$, using Fredholm theory, which resulted to the solution of $(D_p)$ and $(R_p)$ for all $p>1$. In Lipschitz domains, Fredholm theory is not applicable as Fabes, Jodeit, and Lewis \cite{FJL} showed that the relevant operators need not be compact. Therefore, to overcome this difficulty in Lipschitz domains, new ideas were required. Indeed, Verchota  \cite{V} realized how to use  \eqref{eq:2-sided rellich} in order to show (among others) that   the tangential gradient of the single layer potential is invertible in $L^2(\pom)$  and the double layer in $\wt W^{1,2}(\pom)$ and $L^2(\pom)$ solving $(R_{2})$ and $(D_2)$ in Lipschitz domains. In fact, he also showed solvability of $(R_{p})$ for any $1<p \leq 2$. Using a different method, his results were extrapolated  to the optimal range of exponents by Dahlberg and Kenig showing that $(R_{p})$ and $(D_{p'})$ are solvable in Lipschitz domains for $1<p<2+\ve$, for some $\ve$ small depending on dimension and the Lipschitz character of the domain\footnote{The main novelty of that paper is the solution of the Neumann problem rather than the regularity.}, while Kenig and Shen \cite{KeSh} extended them even further in the case of H\"older continuous periodic coefficients. In regular Semmes-Kenig-Toro domains (see \cite[Definition 4.8]{HMT}),  invertibility of  layer potentials was demonstrated by Hofmann, Mitrea, and Taylor in \cite{HMT}. It is worth noting here that for every $p<2$, there exists a bounded Lipschitz domain such $(D_p)$  is not solvable  \cite[pp. 153-154]{Ke1}.

Concerning the duality between $(R_p)$ and $(D_{p'})$, in Lipschitz domains for the Laplacian, Verchota established   $(R_p)  \Leftrightarrow (D_{p'})$ reducing matters to an $L^{p'}$ estimate on the boundary of the so-called harmonic conjugate system. For  real equations, Kenig and Pipher \cite{KP} showed  that $(R_p)$ implies  $ (D_{p'})$, $p \in (1, \infty)$, while Dindo\v{s} and Kirsch \cite{DiK} obtained the endpoint case where  the regularity data are in the Sobolev-Hardy space $H^{1,1}(\pom)$ and the Dirichlet data in $BMO(\pom)$.  For the equivalence between solvability of  $(R_p)$ and solvability of $(D_{p'}) $ for real symmetric constant coefficient systems,  we refer to  Kilty and Shen \cite{KiSh, Sh},   while,  in $\R^{n+1}_+$ and for  elliptic equations  with $t$-independent and $L^\infty$ complex coefficients under the assumption that the De Giorgi/Nash/Moser estimates hold,  we refer to  Hofmann, Kenig, Mayboroda, and Pipher \cite{HKMP} for the range  $p \in (1,2+\ve)$,  and to Auscher and the first author of this paper \cite{AMo} for the range  $ p \in (p_0, 2]$, for some $p_0<1$ determined by the  exponent in the assumed interior H\"older condition.  In the last paper,  the results in fact hold for systems, while extrapolation of solvability all the way to the endpoint was obtained as well.

{
\vv
The  strategy of Theorem \ref{teomain} (our main result) consists in constructing an ``almost harmonic extension'' $v$ of the boundary function $f \in \Lip(\pom)$ (the Dirichlet data) to $\Omega$, so that its distributional Laplacian $\Delta v$
satisfies an $L^p$-Carleson condition in $\Omega$, i.e., the Carleson functional defined in \eqref{eq:Carlesonfunctional} is in $L^p(\pom)$, and, moreover, roughly speaking, its normal derivative $\partial_\nu v$ in $\pom$
is controlled by the Haj\l asz tangential gradient of $f$. 
The construction of this almost orthogonal extension is performed by means of a corona type decomposition of $\Omega$ in terms of mutually disjoint interior Lipschitz subdomains together with some ``buffer'' regions. As far as we know this  decomposition is new and may be of independent interest\footnote{One may compare our construction with the one of Hofmann, Martell, and Mayboroda \cite{HMM}, which was done in terms of interior chord-arc domains that have bounded overlaps. Using Lipschitz graph domains, those arguments were improved by Bortz, Hoffman, Hofmann, Luna Garcia, and Nystr\"om in \cite{BHHLGN}.}. The two components  that allow us to show that the Laplacian of the extension $v$  satisfies an $L^p$-Carleson condition  are the following: a) an extension using the best affine approximation of boundary Lipschitz functions {\it \`a la} Dorronsoro in the buffer regions and the boundaries of the Lipschitz subdomains, and b) the application of the $L^p$ solvability of the regularity problem in the Lipschitz subdomains with boundary data the aforementioned extension. With the almost harmonic extension in hand,  we are able to apply a duality argument to show that
Dirichlet solvability in $L^{p'}$ implies the existence of an ``one-sided" Rellich type inequality that shows that
the normal derivative of the solution of the Dirichlet problem with boundary data $f$ is controlled by the Haj\l asz tangential gradient of $f$ in $L^p$ norm. Finally, we use arguments involving
layer potentials to obtain the desired non-tangential estimates that prove the $L^p$ solvability of the
regularity problem. We remark that we assume that the Dirichlet data $f$ are Lipschitz functions, however, there is no loss of generality in this assumption since Lipschitz functions are dense in the  Haj\l asz Sobolev space (see \cite{Hajlasz}) and the extension follows by  Theorems \ref{thrm:continuous-regularity}, \ref{thm:regularity-noncont}, and \ref{lemuniqueness}.

In the last step  described above we apply some techniques developed by Hofmann, Mitrea, and Taylor
\cite{HMT} which involve a tangential gradient defined by integration by parts (from now on, to be called the HMT tangential gradient). In our paper we show that the HMT tangential gradient coincides with the usual pointwise tangential gradient for all Lipschitz functions $f$, for almost all points with respect to surface measure up to a change of sign.
Concerning the problem $(\wt R_p)$ (the one in terms of the usual tangential gradient, instead of the Haj\l asz gradient),  it was shown in \cite{HMT} that, in domains satisfying the two-sided local John condition,  the norm in the Haj\l asz Sobolev space $W^{1,p}$ is comparable to the one defined using the HMT tangential gradient. This, in turn, equals the 
one in terms of pointwise tangential gradient, by our arguments. For domains satisfying a weak $(1,p)$-Poincaré 
inequality, we show that the Haj\l asz gradient is also comparable to the pointwise tangential gradient
in $L^p$ norm,  for any Lipschitz function. Those results are also new and may be of independent interest.

Our arguments for Theorems \ref{thm:1.3}, \ref{thm:1.4}, and
\ref{propoconverse} follow more classical arguments (although there are significant technical complications in the present framework). 
For example, in Theorem \ref{thm:1.4}
we deduce  invertibility of the single layer potential from $L^p(\sigma)$ to $ W^{1,p}(\pom)$
from the solvability $(\wt R_p)$ both in $\Omega$ and $\R^{n+1}\setminus \overline\Omega$, using arguments that are
inspired by the ones of Hofmann, Kenig, Mayboroda, and Pipher in \cite{HKMP}. Since $\R^{n+1}\setminus \overline\Omega$ is an unbounded domain with compact boundary, we use an appropriate version of the Kelvin transform  to deduce solbavility of $(R_p)$ in such domains from $(R_p)$ in bounded ones (see Theorem \ref{lemunbounded} for more details).
In connection with Theorem \ref{propoconverse}, we show that the $L^p$ solvability of the  regularity problem $(R_p)$ is equivalent to a suitable reverse H\"older inequality with exponent $p$ for harmonic measure, which, in turn, is equivalent to the $L^{p'}$ solvability of the Dirichlet problem. Similar arguments are  well known for the case of Lipschitz domains and they extend to the domains we are considering, although with additional technical difficulties.

\vv

The plan of our paper is the following. Section \ref{secprelim} contains some preliminary results which
will be used along the paper. Sections \ref{secorona}-\ref{secteomain} are devoted to the proof of Theorem \ref{teomain}, while Section 
\ref{sec7*} deals with Theorem \ref{thm:1.3}. On the other hand, Theorems \ref{thm:1.4} and
and Theorem \ref{propoconverse} 
 are proved in Sections \ref{sec8*} and \ref{secconverse}, respectively. The final Section \ref{sec:counterexample} deals with the construction of a chord-arc domain with connected boundary which for  the regularity problem $(\wt R_p)$ is not solvable for any $p\in [1,\infty)$. In the case that the boundary of $\Omega$ is not connected, it is immediate to check that $(\wt R_p)$ is not solvable: just write $\pom= E\cup F$, with $E$ and $F$ closed and disjoint, and consider a function $f$ which equals $1$ in $E$ and vanishes  in $F$. Then $\nabla_t f$ vanishes, but the gradient of the solution of the Dirichlet
problem is not identically $0$ because $f$ is not constant, and thus \eqref{eq:main-est-reg2} cannot hold. The counterexample in Section \ref{sec:counterexample} has  connected boundary and thus it is more interesting.
}


\vv

\subsection*{Acknowledgements}
The first named author warmly thanks Jonas Azzam for several fruitful conversations and also for  his support and encouragement. 

We would also like to thank Svitlana Mayboroda for pushing  us to  remove an additional assumption we had in a previous version of our manuscript (namely, that the measure theoretic boundary coincides with the topological boundary apart from a set of measure zero).

We also wish to thank the referees of this paper for their careful reading and their suggestions, which have contributed to improve the readability of the paper. The arguments in Appendix \ref{appendix} are inspired by the suggestion of one of referees.

\vv


\section{Preliminiaries} \label{secprelim}
\vv

We will write $a\lesssim b$ if there is $C>0$ so that $a\leq Cb$ and $a\lesssim_{t} b$ if the constant $C$ depends on the parameter $t$. We write $a\approx b$ to mean $a\lesssim b\lesssim a$ and define $a\approx_{t}b$ similarly. 
 
\vv

\subsection{Rectifiability, measures and dyadic lattices}\label{subsec:dyadic}

 A set $E\subset \R^{n+1}$ is called $n$-{\textit {rectifiable}} if there are Lipschitz maps
$f_i:\R^n\to\R^{n+1}$, $i=1,2,\ldots$, such that 
$
\HH^n\Bigl(E\setminus\bigcup_i f_i(\R^n)\Bigr) = 0,
$
where $\HH^n$ stands for the $n$-dimensional Hausdorff measure. We will assume $\HH^n$ to be normalized so that it coincides with $n$-dimensional Lebesgue measure in
$\R^n$. Sometimes we will denote the $(n+1)$-dimensional Lebesgue measures in $\R^{n+1}$ by $m$, and integration with respect to $dx$ or $dy$ also means integration with respect to Lebesgue measure.

All measures in this paper are assumed to be Radon measures.
A measure $\mu$ in $\R^{n+1}$ is called 
 $n$-AD-{\textit {regular}} (or just AD-regular or Ahlfors-David regular) if there exists some
constant $C_{0}>0$ such that
$$C_0^{-1}r^n\leq \mu(B(x,r))\leq C_0\,r^n\quad \mbox{ for all $x\in
\supp\mu$ and $0<r\leq \diam(\supp\mu)$.}$$

The measure $\mu$ is  uniformly  $n$-rectifiable if it is 
$n$-AD-regular and
there exist constants $\theta, M >0$ such that for all $x \in \supp\mu$ and all $0<r\leq \diam(\supp\mu)$ 
there is a Lipschitz mapping $g$ from the ball $B_n(0,r)$ in $\R^{n}$ to $\R^{n+1}$ with $\text{Lip}(g) \leq M$ such that
$
\mu (B(x,r)\cap g(B_{n}(0,r)))\geq \theta r^{n}.$

A set $E\subset \R^{n+1}$ is $n$-AD-regular if $\HH^n|_E$ is $n$-AD-regular.
 Also, 
$E$ is uniformly $n$-rectifiable if $\HH^n|_E$ is uniformly $n$-rectifiable.

The notion of uniform rectifiability should be considered a quantitative version of rectifiability. It was introduced in the pioneering works \cite{DS1} and \cite{DS2} of David and Semmes, who were seeking a good geometric framework under which all singular integrals with odd and sufficiently smooth kernels are bounded in $L^2$.

Given an $n$-AD-regular measure $\mu$ in $\R^{n+1}$, we consider 
the dyadic lattice of ``cubes'' built by David and Semmes in \cite[Chapter 3 of Part I]{DS2}. The properties satisfied by $\DD_\mu$ are the following. 
Assume first, for simplicity, that $\diam(\supp\mu)=\infty$). Then for each $j\in\Z$ there exists a family $\DD_{\mu,j}$ of Borel subsets of $\supp\mu$ (the dyadic cubes of the $j$-th generation) such that:
\begin{itemize}
\item[$(a)$] each $\DD_{\mu,j}$ is a partition of $\supp\mu$, i.e.\ $\supp\mu=\bigcup_{Q\in \DD_{\mu,j}} Q$ and $Q\cap Q'=\varnothing$ whenever $Q,Q'\in\DD_{\mu,j}$ and
$Q\neq Q'$;
\item[$(b)$] if $Q\in\DD_{\mu,j}$ and $Q'\in\DD_{\mu,k}$ with $k\leq j$, then either $Q\subset Q'$ or $Q\cap Q'=\varnothing$;
\item[$(c)$] for all $j\in\Z$ and $Q\in\DD_{\mu,j}$, we have $2^{-j}\lesssim\diam(Q)\leq2^{-j}$ and $\mu(Q)\approx 2^{-jn}$;
\item[$(d)$] there exists $C>0$ such that, for all $j\in\Z$, $Q\in\DD_{\mu,j}$, and $0<\tau<1$,
\begin{equation}\label{small boundary condition}
\begin{split}
\mu\big(\{x\in Q:\, &\dist(x,\supp\mu\setminus Q)\leq\tau2^{-j}\}\big)\\&+\mu\big(\{x\in \supp\mu\setminus Q:\, \dist(x,Q)\leq\tau2^{-j}\}\big)\leq C\tau^{1/C}2^{-jn}.
\end{split}
\end{equation}
This property is usually called the {\em small boundaries condition}.
From (\ref{small boundary condition}), it follows that there is a point $x_Q\in Q$ (the center of $Q$) such that $\dist(x_Q,\supp\mu\setminus Q)\gtrsim 2^{-j}$ (see \cite[Lemma 3.5 of Part I]{DS2}).
\end{itemize}
We set $\DD_\mu:=\bigcup_{j\in\Z}\DD_{\mu,j}$. 

In case that $\diam(\supp\mu)<\infty$, the families $\DD_{\mu,j}$ are only defined for $j\geq j_0$, with
$2^{-j_0}\approx \diam(\supp\mu)$, and the same properties above hold for $\DD_\mu:=\bigcup_{j\geq j_0}\DD_{\mu,j}$.

Given a cube $Q\in\DD_{\mu,j}$, we say that its side length is $2^{-j}$, and we denote it by $\ell(Q)$. Notice that $\diam(Q)\leq\ell(Q)$. 
We also denote 
\begin{equation}\label{defbq}
B(Q):=B(x_Q,c_1\ell(Q)),\qquad B_Q = B(x_Q,\ell(Q)),
\end{equation}
where $c_1>0$ is some fix constant so that $B(Q)\cap\supp\mu\subset Q$, for all $Q\in\DD_\mu$. Clearly, we have $Q\subset B_Q$.

For $\lambda>1$, we write
$$\lambda Q = \bigl\{x\in \supp\mu:\, \dist(x,Q)\leq (\lambda-1)\,\ell(Q)\bigr\}.$$


The side length of a ``true cube'' $P\subset\R^{n+1}$ is also denoted by $\ell(P)$. On the other hand, given a ball $B\subset\R^{n+1}$, its radius is denoted by $r(B)$. For $\lambda>0$, the ball $\lambda B$ is the ball concentric with $B$ with radius $\lambda\,r(B)$.

Given $E\subset\R^{n+1}$, a ball $B$, and a hyperplane $L$, we denote
$$b\beta_{E}(B,L) =  \sup_{y\in E\cap B} \frac{\dist(y,L)}{r(B)} + 
\sup_{y\in L\cap B}\!\! \frac{\dist(x,E)}{r(B)} .$$
We set
$$b\beta_{E}(B) = \inf_L b\beta_{E}(B,L),$$
where the infimum is taken over all hyperplanes $L\subset\R^{n+1}$. For a $B=B(x,r)$, we also write
$b\beta_{E}(x,r,L)=b\beta_{E}(B,L)$ and $ b\beta_{E}(x,r)=b\beta_{E}(B).$

For $p\geq1$, a measure $\mu$, a ball $B$, and a hyperplane $L$, we set
$$\beta_{\mu,p}(B,L) = \left(\frac1{r(B)^n}\int_B \left(\frac{\dist(x,L)}{r(B)}\right)^p\,d\mu(x)\right)^{1/p}.$$
We define
$\beta_{\mu,p}(B) = \inf_L \beta_{\mu,p}(B,L),$
where the infimum is taken over all hyperplanes $L$.
For $B=B(x,r)$, we also write
$\beta_{\mu,p}(x,r,L) = \beta_{\mu,p}(B,L)$ and $\beta_{\mu,p}(x,r) = \beta_{\mu,p}(B).$
For a given cube $Q\in\DD_\mu$, we define:
\begin{align*}
\begin{array}{ll}
\beta_{\mu,p}(Q,L)  = \beta_{\mu,p}(B_Q,L)
,&\quad \beta_{\mu,p}(\lambda Q,L)= \beta_{\mu,p}(\lambda B_Q,L),\\
\quad \,\beta_{\mu,p}(Q) = \beta_{\mu,p}(B_Q),
& \quad \quad\,\beta_{\mu,p}(\lambda Q)= \beta_{\mu,p}(\lambda B_Q).
\end{array}
\end{align*}
Also, we define similarly
\begin{align*}
\begin{array}{ll}
b\beta_{\mu}(Q,L)  = b\beta_{\supp\mu}(B_Q,L)
,&\quad b\beta_{\mu}(\lambda Q,L)= b\beta_{\supp\mu}(\lambda B_Q,L),\\
\quad \,b\beta_{\mu}(Q) = b\beta_{\supp\mu}(B_Q),
& \quad \quad\,b\beta_{\mu}(\lambda Q)= b\beta_{\supp\mu}(\lambda B_Q).
\end{array}
\end{align*}

The coefficients $b\beta_E$ and $\beta_{\mu,p}$ above measure the goodness of the approximation of $E$ and
$\supp\mu$, respectively, in a ball $B$ by a hyperplane. They play an important role in the theory of
uniform $n$-rectifiability. See \cite{DS1}.

\vv
\subsection{Types of domains and the Whitney decomposition}\label{subsecwhitney}

Recall that a domain is a connected open set.
In the whole paper, $\Omega$ will be an open set in $\R^{n+1}$ (quite often a domain), with $n\geq 1$.
Very often we will denote the $n$-Hausdorff measure on $\pom$ by $\sigma$.

Given two points $x,x' \in \Omega$, and a pair of numbers $M,N\geq1$, 
an $(M,N)$-{\it Harnack chain connecting $x$ to $x'$},  is a chain of
open balls
$B_1,\dots,B_N \subset \Omega$, 
with $x\in B_1,\, x'\in B_N,$ $B_k\cap B_{k+1}\neq \varnothing$
and $M^{-1}\diam (B_k) \leq \dist (B_k,\partial\Omega)\leq M\diam (B_k).$
We say that $\Omega$ satisfies the {\it Harnack chain condition}
if there is a uniform constant $M$ such that for any two points $x,x'\in\Omega$,
there is an $(M,N)$-Harnack chain connecting them, with $N$ depending only on $M$ and the ratio
$|x-x'|/\left(\min\big(\delta_{\Omega}(x),\delta_{\Omega}(x')\big)\right)$, where $\delta_{\Omega}(x):=\dist(x, \pom)$.  

Following \cite{JeK2}, we say that a
domain $\Omega\subset \mathbb{R}^{n+1}$ is NTA ({\it Non-tangentially accessible})  if it satisfies the
Harnack chain condition, and if both $\Omega$ and
$\Omega_{\textup{ext}}:= \R^{n+1}\setminus \overline{\Omega}$ satisfy the corkscrew condition. We also say that a connected open set $\Omega \subset \R^{n+1}$
 is a CAD ({\it chord-arc domain}), if it is NTA, and if $\pom$ is $n$-AD-regular. Remark that  NTA domains satisfy the local John condition. Additionally, if a domain $\Omega$ and its exterior $\R^{n+1} \setminus \overline\Omega$  are CAD, then we say that $\Omega$ is a two-sided CAD.

The open set $\Omega  \subset \R^{n+1}$ is said  to satisfy the {\it  local John condition} if there is $\theta\in(0,1)$
such that the following holds: For all $x\in\pom$ and $r\in(0,2\diam(\Omega))$ 
there is $y\in B(x,r)\cap \Omega$ such that $B(y,\theta r)\subset \Omega$ with the property that for all
$z\in B(x,r)\cap\pom$ one can find a rectifiable {path} $\gamma_z:[0,1]\to\overline \Omega$ with length
at most $\theta^{-1}|x-y|$ such that 
\begin{equation}\label{eqwlj11}
\gamma_z(0)=z,\qquad \gamma_z(1)=y,\qquad \dist(\gamma_z(t),\pom) \geq \theta\,|\gamma_z(t)-z|
\quad\mbox{for all $t\in [0,1].$}
\end{equation}
If both $\Omega$ and $\R^{n+1}\setminus\Omega$ satisfy the local John condition, we say that $\Omega$ satisfies the {\it two-sided local John condition}.

Another interesting connectivity condition is the semiuniformity. We say that $\Omega  \subset \R^{n+1}$ is {\em semiuniform} if 
there  is $\theta\in(0,1)$ such that
every $x\in\pom$ and $y\in \Omega$ can be connected by a rectifiable {path} $\gamma:[0,1]\to\overline \Omega$ with length
at most $\theta^{-1}|x-y|$ such that 
$\gamma(0)=x$, $\gamma(1)=y$, and
$\min\{\ell(x, z), \ell(y, z)\} \leq \theta^{-1} \dist(z, \pom)$ for all $z \in\gamma$, where $\ell(a, b)$ denotes the length of the sub-arc in $\gamma$ between $a$ and $b$. 

An open set $\Omega  \subset \R^{n+1}$ is said  to satisfy the {\it  weak local John condition} if there are $\lambda,\theta\in(0,1)$ and $\Lambda\geq2$
such that the following holds: 
For every $x\in\Omega$
there is a Borel subset $F\subset B(x,\Lambda \dist(x,\pom))\cap \partial \Omega)$ with
$\sigma(F)\geq \lambda\,\sigma(B(x,\Lambda \dist(x,\pom))\cap \partial \Omega)$
such that
every $z\in F$ 
one can find a rectifiable {path} $\gamma_z:[0,1]\to\overline \Omega$ with length
at most $\theta^{-1}|x-z|$ such that \rf{eqwlj11} holds.

Remark that, for a domain $\Omega$,
$$ \text{NTA}\quad \Rightarrow \quad \text{uniform} \quad\Rightarrow \quad\text{semiuniform} \quad\Rightarrow 
\quad \left\{\begin{array}{l} \text{local John,} \\ \\ \text{weak local John.}
\end{array}\right.$$
It is easy to check that there are domains that satisfy the weak local John condition but not the local John condition. However, as far as we know, it is an open problem if the local John condition implies  the weak local John condition (assuming $\pom$ to be AD regular, say).

Following \cite{AHMMT}, we say that a connected open set $\Omega\subset \mathbb{R}^{n+1}$
has {\it interior big pieces of chord-arc domains} (IBPCAD) if there exist positive constants $\eta$ and $C$, 
and $N\geq 2$, such that for every $x\in \Omega$, with $\delta_{\Omega}(x)<\diam(\pom)$,
there is a  chord-arc domain $\Omega_x\subset \Omega$ such that:
\begin{itemize}
\item $x\in \Omega_x$.
\item $\dist(x,\pom_x) \geq \eta \delta_{\Omega}(x)$.
\item $\diam(\Omega_x) \leq C\delta_{\Omega}(x)$.
\item 
 The set $\Delta^N_x :=\pom\cap B(x,N\delta_{\Omega}(x)$ satisfies $\sigma(\pom_x\cap \Delta^N_x) \geq \,\eta\, \sigma(\Delta^N_x) \,\approx_N\, \eta\,\delta_{\Omega}(x)^n$.
\item The chord-arc constants of the domains $\Omega_x$ are uniform in $x$.
\end{itemize}
 {Domains with $n$-AD-regular boundaries that satisfy IBPCAD are uniformly $n$-rectifiable.}
In fact, as shown in \cite{AHMMT}, if $\pom$ is AD-regular ant $\Omega$ has interior corkscrews, then the following
are equivalent:
\begin{itemize}
\item $\Omega$ satisfies the  IBPCAD  condition.
\item $\pom$ is uniformly rectifiable and $\Omega$ satisfies the weak local John condition.
\item $(D_p)$ is solvable for $\Omega$ for some $p\in (1,\infty)$.
\end{itemize}

\vv

\vv
We 
consider the following Whitney decomposition of $\Omega$ (assuming $\Omega\neq \R^{n+1}$): we have a family $\WW(\Omega)$ of dyadic cubes in $\R^n$ with disjoint interiors such that
$\bigcup_{P\in\WW(\Omega)} P = \Omega,$
and moreover there are
 some constants $\Lambda>20$ and $D_0\geq1$ such the following holds for every $P \in\WW(\Omega)$:
\begin{itemize}
\item[(i)] $10P \subset \Omega$;
\item[(ii)] $\Lambda P \cap \partial\Omega \neq \varnothing$;
\item[(iii)] there are at most $D_0$ cubes $P'\in\WW(\Omega)$
such that $10P \cap 10P' \neq \varnothing$. Further, for such cubes $P'$, we have $\frac12\ell(P')\leq \ell(P)\leq 2\ell(P')$.
\end{itemize}
From the properties (i) and (ii) it is clear that $\dist(P,\partial\Omega)\approx\ell(P)$. We assume that
the Whitney cubes are small enough so that
\begin{equation}\label{eqeq29}
\diam(P)< \frac1{20}\,\dist(P,\partial\Omega).
\end{equation}
The arguments to construct a Whitney decomposition satisfying the properties above are
standard.

Suppose that $\pom$ is $n$-AD-regular and consider the dyadic lattice $\DD_\sigma$ defined above.
Then, for each Whitney $P\in \WW(\Omega)$ there is some cube $Q\in\DD_\sigma$ such that $\ell(Q)=\ell(P)$
and $\dist(P,Q)\approx \ell(Q)$, with the implicit constant depending on the parameters of $\DD_\sigma$ and on the Whitney decomposition. We denote this by $Q=b(P)$ (``b'' stands for ``boundary''). Conversely,
given $Q\in\DD_\sigma$, we let 
\begin{equation}\label{eqwq00}
w(Q) = \bigcup_{P\in\WW(\Omega):Q=b(P)} P.
\end{equation}
It is immediate to check that $w(Q)$ is made up at most of a {uniformly} bounded number of cubes $P$, but it may happen
that $w(Q)=\varnothing$.


\vv

\subsection{Sobolev spaces and the Poincaré inequality}\label{secsobolev}

If $\om \subset \rrn$ is open, 
we define the {\it (homogeneous)  Sobolev space} $\dot W^{1,2}(\om)$ to be the space of functions $u \in L^1_{\loc}(\om)$ such that $\nabla u \in L^2(\om)$ equipped with the seminorm $\| u \|_{\dot W^{1,2}(\om)} := \| \nabla u\|_{L^2(\om)}$. 
On the other hand, $W^{1,2}(\om)$ stands for the usual inhomogeneous Sobolev space equipped with the norm 
$\| u \|_{ W^{1,2}(\om)} := \|u\|_{L^2(\om)} +\| \nabla u\|_{L^2(\om)}$.
 Moreover,  if $2^*:= \frac{2(n+1)}{n-1}$, we define the {\it (inhomogeneous) Sobolev space} $Y^{1,2}(\om):=L^{2^*}(\om) \cap \dot W^{1,2}(\om)$ with norm $\| u \|_{Y^{1,2}(\om)}:=\| u \|_{L^{2^*}(\om)}+\| \nabla u\|_{L^2(\om)}$. We also define the Sobolev  space $Y^{1,2}_0(\om):= \overline{C^\infty_c(\om)}^{\| \cdot \|_{Y^{1,2}(\om)}}$.  It is clear that, by H\"older's inequality,  for any ball $B$ centered at the boundary, $W^{1,2}(B) \subset Y^{1.2}(B)$.

 Let $\Sigma$  be a metric space equipped with a doubling
measure $\sigma$ on $\Sigma$, which means that there is a uniform constant $C_\sigma\geq1$ such that $\sigma(B(x,2r))\leq C_\sigma\, \sigma(B(x,r))$, for all $x\in \Sigma$ and $ r>0$. We will now  define  the {\it Haj\l{}asz's Sobolev spaces}  $\dot{W}^{1,p}(\Sigma)$ and  ${W}^{1,p}(\Sigma)$, which were introduced in \cite{Hajlasz}. For  more information on those spaces and, in general, Sobolev spaces in metric measure  spaces, the reader may consult \cite{Heinonen}.

For a Borel function $f:\Sigma\to\R$, we say that a non-negative Borel function $g:\Sigma \to \R$ is a {\it Haj\l asz upper gradient of  $f$} if  
\begin{equation}\label{eq:H-Sobolev}
 |f(x)-f(y)| \leq |x-y| \,(g(x)+g(y))\quad \mbox{ for $\sigma$-a.e. $x, y \in \Sigma$.} 
\end{equation}
We denote the collection of all the Haj\l asz upper gradients of $f$ by $D(f)$.

For $p\geq1$, we denote by $\dot{W}^{1,p}(\Sigma)$ the space of Borel functions $f$ which have 
a Haj\l asz upper gradient in $L^p(\sigma)$, and we let $W^{1,p}(\Sigma)$ be the space of functions $f\in L^p(\sigma)$ which have a Haj\l asz upper gradient in $L^p(\sigma)$, i.e.,  $W^{1,p}(\Sigma)=  \dot W^{1,p}(\Sigma) \cap L^p(\sigma)$.
We  define the semi-norm (as it annihilates constants)
\begin{equation}\label{eqseminorm}
 \| f \|_{ \dot W^{1.p}(\Sigma)} = \inf_{g \in D(f)} \| g\|_{L^p(\Sigma)}
 \end{equation}
 and the scale-invariant norm
 \begin{equation}\label{eqnorm}
 \| f\|_{W^{1,p}(\Sigma)} = \diam(\Sigma)^{-1} \|f\|_{L^p(\Sigma)} +   \inf_{g \in D(f)} \| g\|_{L^p(\Sigma)}.
 \end{equation}
 Remark that, 
for any a metric space $\Sigma$, in the case $p\in (1,\infty)$, from the uniform convexity of $L^p(\sigma)$, one easily deduces
that the infimum in the definition of the norm $\|\cdot\|_{W^{1,p}(\Sigma)}$ and $\|\cdot\|_{\dot W^{1,p}(\Sigma)}$ in \eqref{eqseminorm} and \eqref{eqnorm} is attained and is unique. We denote by $\nabla_{H,p} f$ the function $g$ which attains the infimum.

\begin{proposition}\label{propopoin}
For every ball $B$ centered at $\Sigma$, every Borel function $f:\Sigma\to\R$ and every Haj\l asz upper
gradient $g$ of $f$,
$$\avint_{B} \left|f - m_{\sigma,B}(f)\right| \,d\sigma \leq C\, r(B) \,\avint_{B} g\,d\sigma.$$
\end{proposition}

\begin{proof}
The proof is almost immediate:
\begin{align*}
\avint_{B} \left|f - m_{\sigma,B}(f)\right| \,d\sigma & \leq \avint_{B}\avint_{B} |f(x) - f(y)|\,d\sigma(x)\,d\sigma(y) \\
& \leq 2\,r(B)\,\avint_{B}\avint_{B} (g(x) + g(y))\,d\sigma(x)\,d\sigma(y) \leq 4\,r(B)\,\avint_{B} g\,d\sigma.
\end{align*}
\end{proof}
\vv

Observe now also that if $f:\Sigma\to\R$ is Lipschitz, then
\begin{equation}\label{eqlip12}
\|\nabla_{H,p}f\|_{L^\infty(\sigma)}\leq {{\rm Lip}(f)/2}.
\end{equation}
This follows easily from the fact that if $g$ is a Haj\l asz upper gradient for $f$, then 
$\min(g,{{\rm Lip}(f)/2})$ is also a Haj\l asz upper gradient.

\vv

\vv

\subsection{Finite perimeter, measure theoretic boundary, and tangential gradients} \label{subsec:finite perimeter}

An open set $\Omega\subset\R^{n+1}$ has {\it finite perimeter} if {the distributional gradient $\nabla \chi_\Omega$ of $\chi_\Omega$}  is a locally finite $\R^{n+1}$-valued measure.  From results of De Giorgi and Moser it follows that $\nabla \chi_\Omega=
-\nu_{\om}\,\HH^n_{\pom^*}$, where $\partial^*\Omega\subset\pom$ is the reduced boundary of $\Omega$ and
$|\nu_{\om}(x)|=1$ $\HH^n$-a.e.\ in $\partial^*\Omega$.  By  \cite[Theorem 5.15]{EG},  $\partial^*\Omega$ can be written, up to  a set of $\HH^n$-measure zero,  as a countable union of  compact sets $\{K_j\}_{j=1}^\infty$ where $K_j \subset S_j$ for a  $C^1$ hypersurface $S_j$  and $\nu_\om|_{S_j}$ is  normal to $S_j$.  Moreover, the following Green's formula is satisfied: for every $\vphi \in C^\infty_c(\R^{n+1};\R^{n+1})$, 
\begin{equation}\label{eq:Green-finiteperim}
\int_\om \dv \varphi(x)\,dx = \int_{\partial^*\om} \nu_{\om}(\xi) \cdot \vphi(\xi)\,\HH^n(\xi).
\end{equation}
More generally, given a unit vector $\nu_\om$ and $x \in \pom$, we define the (closed) half-spaces 
$$
H^\pm_{\nu_\om}(x)= \{ y \in \pom: \nu_\om \cdot (y-x) \geq 0\}.
$$
Then, for $x \in \partial^*\om$, it holds
\begin{equation}\label{eq:normal-halfspace}
\lim_{r \to 0} r^{-(n+1)} m(B(x,r)\cap \Omega^\pm \cap H^\pm_{\nu_\om}(x))>0,
\end{equation}
where $\om^+=\om$ and $\om^-=\R^{n+1} \setminus \om$ (see for instance \cite[p. 230]{EG}).  A unit vector $\nu_\om$  satisfying \eqref{eq:normal-halfspace} 
is called the {\it measure theoretic outer unit normal} to $\Omega$ at $x$ and we denote by $\partial_0 \om$ all the points of $\pom$ for which \eqref{eq:normal-halfspace} holds.  It is clear that $\partial^* \om \subset \partial_0 \om$.

The {\it measure theoretic boundary} $\partial_*\Omega$ consists of the points $x\in\pom$ such that
$$\limsup_{r\to0}\frac{m(B(x,r)\cap\Omega)}{r^{n+1}}>0\quad \mbox{ and }\quad
\limsup_{r\to0}\frac{m(B(x,r)\setminus \overline\Omega)}{r^{n+1}}>0.$$
 When $\Omega$ has finite perimeter, it holds that $\partial_*\Omega\subset \partial_0 \om \subset \partial^*\Omega\subset \pom$ and $\HH^n(\partial^*\Omega\setminus\partial_*\Omega)=0$. A good reference for those results is either the book of Evans and Gariepy \cite{EG} or the book of Maggi \cite{Maggi}.

It was proved by Federer (see \cite[p. 314]{Fed} and the references therein) that if $\om \subset \R^{n+1}$ is a bounded set such that $\HH^n(\pom)<\infty$ and $u \in C(\overline \om)\cap  \dot W^{1,1}(\om)$, then for any integer $j \in [1, n+1]$, it holds that
\begin{equation}\label{eq:Green-Federer}
\int_{\om} \partial_j u(x)\,dx=\int_{\partial_0 \om} e_j \cdot\nu_\om(\xi) \,u(\xi) \,d\sigma(\xi),
\end{equation}
where $e_j$ is the $j$-th standard basis vector of $\R^{n+1}$ and $\nu_\om$ is the measure theoretic outer unit normal to $\Omega$.
\vv

Given an $n$-rectifiable set $E\subset \R^{n+1}$, consider a point $x\in E$ such that the approximate tangent space $T_x E$ exists (this is just the $n$-dimensional vector space parallel to the approximate 
tangent to $E$ in $x$). We say that $f:\R^{n+1}\to\R$ is {\it tangentially differentiable} with respect to
$E$ at $x$ if the restriction of $f$ to $T_x E$ is differentiable at $x$, and we denote 
its {gradient} by $\nabla_{t,E}f(x)$, or $\nabla_{t}f(x)$ if there is no confusion about $E$.

If $f:E\to\R$ is Lipschitz, then we can consider an arbitrary Lipschitz extension of $f$ to $\R^{n+1}$, which we denote by $\wt f:\R^{n+1}\to\R$. Then we define $\nabla_{t,E}f(x) := \nabla_{t,E}\wt f(x)$ and, by  \cite[Theorem 11.4]{Maggi}, $\nabla_{t,E}f(x)$ exists
   for $\HH^n$-a.e. $x\in E$. Furthermore,
 the definition does not depend on the particular extension $\wt f$, for $\HH^n$-a.e.\ $x\in E$.
This follows easily from  \cite[Lemma 11.5]{Maggi}.
\vv

\begin{lemma}\label{lem2.2}
Let $E\subset\R^{n+1}$ be uniformly $n$-rectifiable and $f:\R^{n+1}\to\R$ be Lipschitz.
Then 
\begin{equation}\label{eqidtan}
|\nabla_{t,E}f(x)| =\limsup_{E\ni y\to x} \frac{|f(y) - f(x)|}{|y-x|}
\approx \limsup_{r\to0} \avint_{B(x,r)\cap E} \frac{|f(y) - f(x)|}{|y-x|}\,d\HH^n(y)
\end{equation}
 for $\HH^n$-a.e.\ $x\in E$.
\end{lemma}

\begin{proof}
Let $x\in E$ be a point of tangential differentiability of $f$ such that there exists an approximate tangent plane $L_x$ to $E$ at $x$ and 
$\lim_{r\to0} b\beta_E(x,r)=0$. We claim that \eqref{eqidtan} holds for $x$. 
To this end, notice that, by {the} definition of   tangential derivative,
\begin{equation}\label{eqtan74}
|\nabla_{t,E}f(x)| = \limsup_{\Pi_x \ni z\to x} \frac{|f(z) - f(x)|}{|z-x|},
\end{equation}
{where}  $\Pi_x$ is the orthogonal projection on $L_x$. Then we have
\begin{align*}
\limsup_{E\ni y\to x} \frac{|f(y) - f(x)|}{|y-x|} & \leq 
\limsup_{E\ni y\to x} \frac{|f(\Pi_x(y)) - f(x)|}{|y-x|} + \limsup_{E\ni y\to x} \frac{|f(\Pi_x(y)) - f(y)|}{|y-x|}\\
& \leq |\nabla_{t,E}f(x)| + {\rm Lip}(f) \,\limsup_{E\ni y\to x}\frac{|\Pi_x(y) - y|}{|y-x|}.
\end{align*}
From the condition that $\lim_{r\to0} b\beta_E(x,r)=0$ and the fact that $L_x$ is an approximate tangent hyperplane, if follows easily that $b\beta_E(x,r,L_x)\to 0$ {as} $r\to0$ and so the last limsup above vanishes. Thus,
$$\limsup_{E\ni y\to x} \frac{|f(y) - f(x)|}{|y-x|}  \leq |\nabla_{t,E}f(x)|.$$

The converse estimate is similar. For each $y\in L_x$, let $p(y)$ be the closest point from
$E$ to $y$. Then, we have $|y-p(y)|\lesssim b\beta(x,2|x-y|,L_x) |x-y|$, and thus
by \eqref{eqtan74}, arguing as {above,} 
\begin{align*}
|\nabla_{t,E}f(x)| & \leq  \limsup_{\Pi_x \ni y\to x} \frac{|f(p(y)) - f(x)|}{|y-x|} + 
\limsup_{\Pi_x \ni y\to x} \frac{|f(p(y)) - f(y)|}{|y-x|}.
\end{align*}
The second term satisfies
\begin{align*}
\limsup_{\Pi_x \ni y\to x} \frac{|f(p(y)) - f(y)|}{|y-x|} &\leq {\rm Lip}(f) 
 \limsup_{\Pi_x \ni y\to x} \frac{|p(y) - y|}{|y-x|} \\ &\lesssim  {\rm Lip}(f) \,{ \limsup_{\Pi_x \ni y\to x} } \,b\beta(x,2|x-y|,L_x) =0,
 \end{align*}
which shows that
$|\nabla_{t,E}f(x)|\leq \limsup_{E\ni z\to x} \frac{|f(z) - f(x)|}{|z-x|},$
so that the first identity in \eqref{eqidtan} holds.

Regarding the comparability on the right side of \eqref{eqidtan}, we use similar arguments.
First we choose an orthonormal basis $v_1,\ldots, v_n$ of $L_x-x$, and for each $i$ we let $B_i$ be a ball centered at $v_i$ with radius $1/2n$, 
so that for all points $z_i\in B_i\cap L_x-x$ and any vector $u\in L_x-x$,
$$\Big|\sum_{i=1}^n|z_i\cdot u| - \sum_{i=1}^n |v_i\cdot u|\Big| \leq n \,\frac1{2n}\,|u|
=\frac12\Big(\sum_i |v_i\cdot u|^2 \Big)^{1/2}
 \leq 
\frac12 \sum_{i=1}^n|v_i\cdot u|.
$$
In particular, this gives that
$$\sum_{i=1}^n|z_i\cdot u| \approx \sum_{i=1}^n |v_i\cdot u| \quad \mbox{for all $z_i\in B_i$, $i=1,\ldots, n$.}$$
 In this way, if we let $B_{i,r}$ be balls with radius $c(n)r$ centered at $x+r\,v_i$, for all
the points $z_{i,r}\in B_{i,r}\cap L_x$ we have 
\begin{equation}\label{eqak56}
|\nabla_{t,E} f(x)|\approx  \sum_{i=1}^n \Big|\frac{z_{i,r} -x}{|z_{i,r} - x|}\cdot \nabla_{t,E} f(x)\Big|
\approx \sum_{i=1}^n \frac{|f(z_{i,r}) - f(x)|}{|z_{i,r} - x|} + o(r).
\end{equation}

Assuming $b\beta_E(x,r,L_x)$ small enough, by the AD-regularity of $E$, we have
$\HH^n(\frac12 B_{i,r}\cap E)\approx r^n$ for each $i$. Consider arbitrary points $y_{i,r}\in \frac12B_{i,r}\cap E$. Since 
$$|f(y_{i,r}) - f(\Pi_x(y_{i,r})|\leq {\rm Lip}(f) \,|y_{i,r} - \Pi_x(y_{i,r})| \lesssim{\rm Lip}(f)\,
b\beta_E(x,r,L_x)\,r,$$
by \eqref{eqak56} we get
$$\sum_{i=1}^n \frac{|f(y_{i,r}) - f(x)|}{|y_{i,r} - x|} = \sum_{i=1}^n \frac{|f(\Pi_x(y_{i,r})) - f(x)|}{|y_{i,r} - x|} + o(r) \approx |\nabla_{t,E} f(x)| + o(r),$$
taking also into account that the points $\Pi_x(y_{i,r})$ belong to $L_x\cap B_{i,r}$ {and $|y_{i,r} - x| \approx r$}, 
and applying \eqref{eqak56} to the points $z_{i,r}=\Pi_x(y_{i,r})$, assuming again $b\beta_E(x,r,L_x)$ small enough.
Averaging over all $y_{i,r}\in \frac12B_{i,r}\cap E,$ and using that $\HH^n(\frac12 B_{i,r}\cap E)\approx r^n$, we conclude
$$|\nabla_{t,E} f(x)|\approx 
\avint_{B(x,2r)\cap E} 
\frac{|f(y) - f(x)|}{|y - x|} \,d\HH^n+ o(r).$$
\end{proof}

\vv

In \cite{HMT}, Hofmann, Mitrea, and Taylor have introduced some tangential derivatives and another tangential gradient which
are  well suited for arguments involving integration by parts and layer potentials under the assumption that $\HH^n(\pom\setminus \partial^*\Omega)=0$. We will adapt their definition to more general domains by simply restricting most things to the reduced boundary ${\partial^* \Omega}$.

Let $\Omega\subset\R^{n+1}$ be an open set with finite perimeter. For a $C_c^1$ function $\vphi:\R^{n+1}\to \R$ and $1\leq j,k\leq n+1$, one  defines the tangential derivatives on ${\partial^* \Omega}$ by
\begin{equation}\label{eqparts099}
 \partial_{t,j,k} \vphi := \nu_j\,(\partial_k \vphi)|_{\partial^* \Omega} - \nu_k\,(\partial_j \vphi)|_{\partial^* \Omega},
\end{equation}
where $\nu_i$, $i=1,\ldots,n+1$ are the components of the outer normal $\nu$. Remark that, by integration by parts, if and $\vphi,\psi$ are $C^1$ in a neighborhood of $\pom$, the same arguments as in \cite[p.  2676]{HMT} (using \eqref{eq:Green-finiteperim}) show that
\begin{equation}\label{eqparts09}
\int_{\partial^* \Omega} \partial_{t,j,k}\psi\,\vphi\,d\HH^n = \int_{\partial^* \Omega} \psi\,\partial_{t,k,j} \vphi\,d\HH^n.
\end{equation}

We define the Sobolev type space $L^p_1(\HH^n|_{\partial^* \Omega})$ {(see \cite[display (3.6.3)]{HMT})} as the subspace of functions in $L^p(\HH^n|_{\partial^* \Omega})$
for which there exists some constant $C(f)$ such that
\begin{equation}\label{eqlp1HMT}
\sum_{1\leq j,k\leq n+1} \left|\int_{\partial^* \Omega} f\, \partial_{t,k,j} \vphi\,d\HH^n\right|\leq C(f)\,\|
\vphi\|_{L^{p'}(\HH^n|_{\partial^* \Omega})}
\end{equation}
for all $\vphi\in C_c^\infty(\R^{n+1})$. By the Riesz representation theorem, for
each $f\in L^p_1(\HH^n|_{\partial^* \Omega})$ and each $j,k=1,\ldots,n+1$, there exists some function $h_{j,k}\in L^p(\HH^n|_{\partial^* \Omega})$ such that 
$$
\int_{\partial^* \Omega}  h_{j,k}\,\vphi\,d\HH^n= \int_{\partial^* \Omega}  f\, \partial_{t,k,j} \vphi\,d\HH^n$$
and we set 
$\partial_{t,j,k} f:=h_{j,k}$, so that this is coherent with \eqref{eqparts09}. It is easy to check that Lipschitz functions with compact support are contained in $L^p_1(\partial^* \Omega)$. 

We introduce two important integral operators whose kernels are associated with $\EE$, the fundamental solution for the Laplacian. The first one is the {\it single layer potential}
\begin{equation}\label{def-single}
\cS f(x)= \int_{\partial \Omega} \EE(x-y) f(y) \,d\sigma(y), \,\,\,x \in\Omega
\end{equation}
and the second one is the {\it double layer potential}
\begin{equation}\label{def-double}
\DD f(x)= \int_{\partial^*\om} \nu(y) \cdot \nabla_y \EE(x-y) f(y) \,d\sigma(y), \,\,\,x \in\Omega,
\end{equation}
where $\nu$ stands for the measure theoretic outer unit normal of $\Omega$. As $\nabla \cS$ is up to a multiplicative constant the Riesz transform,  if $\pom$ is uniformly $n$-rectifiable, by the $L^p(\sigma)$ boundedness of the Riesz 
transform \cite{David-surfaces} and Cotlar's inequality, we have
\begin{equation}\label{eq:ntboundssingle}
\|\NN(\nabla\cS f)\|_{L^p(\sigma)} \lesssim \|f\|_{L^p(\sigma)}.
\end{equation}
See \cite[Proposition 3.20]{HMT} for the precise details, for example.
Moreover, by the discussion above, if we argue as in the proof of  \cite[Proposition 3.37]{HMT},  it holds that
\begin{equation}\label{eq:ntboundsgraddouble}
\|\NN(\partial_j\DD f)\|_{L^p(\sigma)}^p \lesssim \sum_{k=1}^{n+1} \|\partial_{t, k, j} f\|_{L^p(\HH^n|_{\partial^* \Omega})}^p.
\end{equation}

\vv

Below we will need the following technical result regarding the tangential gradients.

\begin{lemma}\label{lemHMT-tan}
Let $\Omega\subset\R^{n+1}$ be a bounded corkscrew domain with uniformly $n$-rectifiable boundary. 
If $f$ is a function which is $C^1$ in a neighborhood of $\pom$,  then we have that for every $j, k \in \{1,2, \dots, n+1\}$,
 \begin{equation}\label{eqclau30}
\partial_{t,j,k} f(x) =\nu_j \,(\nabla_{t} f)_k(x) - \nu_k\,(\nabla_{t} f)_j(x)\quad\mbox{ for $\HH^n|_{\partial^* \om}$-a.e.\ $x\in\partial^* \om$,}
\end{equation}
where $(\nabla_{t} f)_k$ stands for the $k$-th component of $\nabla_t f \equiv \nabla_{t,\pom} f$.
\end{lemma}

\begin{proof}
As shown in \cite{HMT}, for a function $f$ which is $C^1$ in a neighborhood of $\pom$, we have the pointwise identity
$$\partial_{t,j,k} f = \nu_j\,(\partial_k f)|_\pom - \nu_k\,(\partial_j f)|_\pom
= (\nu_j\,e_k- \nu_k\,e_j) \cdot \nabla f \quad\mbox{ $\sigma$-a.e.\ in $\pom$},$$
where $e_i$, $i=1,\ldots,n+1$, is the standard orthonormal basis of $\R^{n+1}$.
Since
\begin{equation}\label{eqorto1}
(\nu_j\,e_k- \nu_k\,e_j) \cdot \nu = \nu_j\,\nu_k - \nu_k\,\nu_j=0,
\end{equation}
we have 
$$\partial_{t,j,k} f=(\nu_j\,e_k- \nu_k\,e_j) \cdot \nabla f = (\nu_j\,e_k- \nu_k\,e_j) \cdot \nabla_t f,$$
which implies \eqref{eqclau30}.
\end{proof}

\vv


\section{The corona decomposition in terms of Lipschitz subdomains}\label{secorona}

 In this section we assume that $\Omega\subset\R^{n+1}$ is a bounded open set with uniformly rectifiable
boundary satisfying the corkscrew condition. Notice that if $\Omega$ 
satisfies the assumptions in Theorem \ref{teomain} ($n$-AD-regular boundary, corkscrew condition, and solvability of $(D_{p'})$), it follows that $\partial\Omega$ is uniformly rectifiable, by combining the results from \cite{Hofmann-Le}
and either \cite{HLMN} or \cite{MT}.

\vv

\subsection{The approximating Lipschitz graph}\label{subs:approxLipgraph}

In this subsection we describe how to associate an approximating Lipschitz graph to a cube $R\in \DD_\sigma$, assuming $b\beta_{\sigma}(k_1R)$ to be small enough for some big constant $k_1>2$ (recall that the coefficients $b\beta_\sigma$ are defined at the end of Section \ref{subsec:dyadic}).
We will follow 
the arguments in \cite[Chapters 7, 8, 12, 13, 14]{DS1} quite closely.
The first step consists in defining suitable stopping cubes.

Given $x\in\R^{n+1}$, we write $x= (x',x_{n+1})$.
For a given cube $R\in\DD_\sigma$, we denote by $L_R$ a best approximating hyperplane for
 $b\beta_\sigma(k_1R)$.
We also assume, without loss of generality, that 
$
L_R \,\,\textup{is the horizontal  hyperplane}\,\, \{x_{n+1}=0\}.
$ 
We denote by ${C}(R)$ the cylinder
$$\big\{x\in \R^{n+1}: |x'-(x_R)'|\leq 2^{-1/2}r(B(R)),\,
{|x_{n+1}-(x_R)_{n+1}| }\leq 2^{-1/2}r(B(R))\big\}.$$
Observe that $C(R)\subset B(R)$.

We fix $0<\ve\ll\delta\ll1$ to be chosen later (depending on the cokscrew condition and the uniform rectifiability constants), $k_1>2$, and we denote by $\cB$ or $\cB(\ve)$ the family of cubes $Q\in
\DD_\sigma$ such that 
$b\beta_\sigma(k_1Q) > \ve$.
We consider $R\in\DD_\sigma$ such that $b\beta_\sigma(k_1R)\leq \ve$. 
We let $\sss(R)$ be the family of maximal cubes $Q\in\DD_\sigma(R)$ such that at least one of the following holds:
\begin{itemize}
\item[(a)] {$Q\cap C(R) = \varnothing$.}
\item[(b)] $Q\in\cB(\ve)$, i.e., $b\beta_\sigma(k_1Q) > \ve$.\item[(c)] $\angle(L_Q,L_R)> \delta$, where $L_Q$, $L_R$ are best approximating hyperplanes for $\beta_{\sigma,\infty}(k_1Q)$ and $\beta_{\sigma,\infty}(k_1R)$, respectively, and {$\angle(L_Q,L_R)$ denotes the angle between $L_Q$ and $L_R$.}
\end{itemize}
We denote by $\tree(R)$ the family of cubes in $\DD_\sigma(R)$ which are not strictly contained
in any cube from $\sss(R)$. We also consider the function
$$d_R(x) = \inf_{Q\in\tree(R)} \big(\dist(x,Q) + \diam(Q)\big).$$
Notice that $d_R$ is $1$-Lipschitz.
Assuming $k_1$ big enough (but independent of $\ve$ and $\delta$) and arguing as in the proof of  \cite[Proposition 8.2]{DS1}, the following holds:

\begin{lemma}\label{lemgraf}
Denote by $\Pi_R$ the orthogonal projection on $L_R$.
There is a Lipschitz function $A:L_R \to L_R^\bot$ with slope at most $C\delta$ such that
$$\dist(x,(\Pi_R(x),A(\Pi_R(x)))) \leq C_1\ve\,d_R(x)\quad \mbox{ for all $x\in k_1R$.}$$
\end{lemma}

Remark that in this lemma, and in the whole subsection, we assume that {$R$ is} as 
above, so that, in particular, $b\beta_\sigma(k_1R)\leq \ve$.

For all $y\in L_R$ we denote
$D_R(y)= \inf_{x\in\Pi_R^{-1}(y)}d_R(x).$
It is immediate to check that $D_R$ is also a $1$-Lipschitz function. Further, as shown in \cite[Lemma
8.21]{DS1}, there is some fixed constant $C_2$ such that
\begin{equation}\label{eqDR}
C_2^{-1}d_R(x) \leq D_R(\Pi_R(x)) \leq d_R(x)\quad \mbox{ for all $x\in {3B(R)}$.}
\end{equation}

We denote by $Z(R)$ the set of points $x\in R$ such that {$d_R(x)=0$.}
The following lemma is an immediate consequence of {the} results obtained in  \cite[Chapters 7, 12-14]{DS1}. 

\begin{lemma}\label{lempack1}
There are some constants $C_3(\ve,\delta)>0$ and $k_1\geq 2$ such that
\begin{align}\label{eqlempack1}
\sigma(R) \approx \sigma(C(R))& \leq 2\,\sigma(Z(R))\\ &\quad + 2\sum_{Q\in\sss(R)\cap\cB(\ve)} \!\sigma(Q) + C_3 \sum_{Q\in\tree(R)} \!\beta_{\sigma,1}(k_1Q)^2\,\sigma(Q).\nonumber
\end{align}
\end{lemma}

The fact that $\sigma(R) \approx \sigma(C(R))$ is an immediate consequence of the AD-regularity of $\sigma$.
The lemma above is not stated explicitly in \cite{DS1}. However, this follows easily from the results in \cite{DS1}. Indeed,
denote
$$\cF_1= \Big\{R\in\DD_\sigma:\sigma\Big(\bigcup_{Q\in \sss(R)\cap {\rm{(c)}}} Q\Big)\geq \sigma(C(R))/2\Big\},$$
where $Q\in \sss(R)\cap \rm{(c)}$ means that $Q$ satisfies the condition (c) in the above definition of $\sss(R)$.
Notice that $\cF_1$ is very similar to the analogous set $\cF_1$ defined in \cite[p.39]{DS1}. A (harmless) difference is that we wrote
$\sigma(C(R))/2$ in the definition above, instead of $\sigma(R)/2$ as in \cite{DS1}. Assuming $\ve>0$ small enough (depending on $\delta$) in the definition of 
$\cB(\ve)$,
in equation (12.2) from \cite{DS1} (proved along the Chapters 12-14) it is shown that 
there exists some $k>1$ (independent of $\ve$ and $\delta$) such that if $R\in \cF_1$, then 
$$\iint_X  \beta_{\sigma,1}(x,kt)^2\,\frac{d\sigma(x)\,dt}t \gtrsim_\delta \sigma(R),$$
where
$$X= \big\{(x,t)\in \supp\sigma \times (0,+\infty): x\in kR,\,k^{-1}d_R(x)\leq t\leq k\ell(R)\big\}.$$
It is easy to check that, choosing $k_1>k$ large enough,
$$\iint_X  \beta_{\sigma,1}(x,t)^2\,\frac{d\sigma(x)\,dt}t \lesssim_k\sum_{Q\in \tree(R)}\beta_{\sigma,1}(k_1Q)^2\,\sigma(Q).$$
Hence, \rf{eqlempack1} holds when $R\in\cF_1$.

In the case $R\not\in\cF_1$, by the definition of $\sss(R)$ we have
\begin{align*}
\sigma(R) \approx \sigma(C(R))  &\leq \sigma(Z(R)\cap C(R)) \\ &\quad + \sum_{Q\in\sss(R)\cap\cB(\ve)}\! \sigma(Q\cap C(R))
+\! \sum_{Q\in\sss(R)\cap \rm{(c)}} \!\sigma(Q\cap C(R))
.
\end{align*}
Since the last sum does not exceed $\sigma(C(R))/2$, we deduce that
$$\frac12\,\sigma(C(R)) \leq \sigma(Z(R)\cap C(R)) + \!\sum_{Q\in\sss(R)\cap\cB(\ve)}\! \sigma(Q\cap C(R)),$$
and so \rf{eqlempack1} also holds.

\vv


\subsection{The starlike Lipschitz subdomains $\Omega_R^\pm$} \label{subs:star-Lip}

 Abusing notation,  below we write 
 $$
 D_R(x')=D_R(x),\quad\textup{for}\,\,x=(x',x_{n+1}).
 $$

\begin{lemma}\label{lem333}
Let $$U_R = \{x\in C(R): x_{n+1}> A(x') + C_1C_2\ve D_R(x')\},$$
 $$V_R = \{x\in C(R): x_{n+1}< A(x') - C_1C_2\ve D_R(x')\},$$
 and
 $$W_R = \{x\in C(R): A(x') - C_1C_2\ve D_R(x') \leq x_{n+1}\leq A(x') + C_1C_2\ve D_R(x')\}.$$
  Then $\pom\cap C(R)\subset W_R$. Also,
$U_R$ is either contained in $\Omega$ or in  $\R^{n+1}\setminus
\overline\Omega$, and the same happens with $V_R$. Further, at least one of the sets $U_R$, $V_R$ is contained in $\Omega$.
\end{lemma}

Remark that it may happen that $U_R$ and $V_R$ are both contained in $\Omega$, or that one set is contained in $\Omega$ and the other in $\R^{n+1}\setminus
\overline\Omega$.

\begin{proof}
Let us see that $\pom\cap C(R) \subset W_R$. Indeed, we have
$\partial\Omega\cap C(R)\subset \partial\Omega\cap B(R) \subset R$, by the definition of $B(R)$.
Then, by Lemma \ref{lemgraf} and \eqref{eqDR}, for all $x\in \partial\Omega\cap C(R)$ we have
$|x- (x',A(x'))|\leq C_1\ve\,d_R(x) \leq C_1C_2\ve\,D_R(x),$
which is equivalent to saying that $x\in W_R$.

Next we claim that if $U_R\cap \Omega\neq \varnothing$, then $U_R\subset \Omega$.
This follows from connectivity, taking into account that if $x\in U_R\cap \Omega$ and $r=\dist(x,\partial U_R)$, then $B(x,r)\subset \Omega$. {Otherwise, there exists some point $y\in B(x,r)\setminus
\overline \Omega$, and thus there exists some $z\in\pom$ which belongs to the segment $\overline{xy}$.} This would contradict the fact that $\pom\subset W_R$.
The same argument works replacing $U_R$ and/or $\Omega$ by $V_R$ and/or $\R^{n+1}\setminus\overline{\Omega}$, and thus we deduce that any of the sets $U_R$, $V_R$ is contained either in $\Omega$ or in $\R^{n+1}\setminus\overline{\Omega}$. 

Finally, from the 
corkscrew condition we can find  a point $y\in {B(x_R,r(B(R)))}\cap \Omega$ with $\dist(y,\pom)\gtrsim r(B(R))$. {So if $\ve,\delta$ are small enough we deduce that $y\in (U_R \cup V_R)\cap \Omega$ because $b\beta_\sigma(k_1R)\leq \ve$ and both $\pom\cap C(R)$ and the graphs of $A$ in $C(R)$ are contained in
a $C\delta\ell(R)$-neighborhood of the hyperplane $\{x_{n+1}=0\}$.}
  Then by the discussion in the previous paragraph, we infer that either 
$U_R\subset\Omega$ or $V_R\subset\Omega$. 
\end{proof}
\vv

Suppose that $U_R\subset \Omega$.
We denote by $\Gamma_R^+$ the Lipschitz graph of the function 
$C(R)\cap L_R\ni x'\mapsto A(x') + \delta\,D_R(x')$. Notice that this is a Lipschitz function with slope at most 
$C\delta< 1$ (assuming $\delta$ small enough). So $\Gamma_R^+$ intersects neither the top nor the bottom faces of $C(R)$, assuming $\ve$ small enough too.
Then we define
$$\Omega_R^+ =\big\{x=(x',x_{n+1}) \in {\rm Int}(C(R)): x_{n+1}> A(x') + \delta \,D_R(x')\big\}.$$
Observe that $\Omega_R^+$ is a starlike Lipschitz domain (with uniform Lipschitz character) and that $\Omega_R^+\subset U_R$, assuming that $C_1C_2\ve\ll\delta$.

In case that $V_R\subset \Omega$, we define $\Gamma_R^-$ and $\Omega_R^-$ analogously, replacing the above function $A(x') + \delta\,D_R(x')$ by $A(x') - \delta\,D_R(x')$. It is clear that $\Omega_R^-$ is also a starlike Lipschitz domain.
If $V_R\subset \R^{n+1}\setminus \overline\Omega$, then we set $\Omega_R^-=\varnothing$.
In any case, we define
$$\Omega_R=\Omega_R^+ \cup\Omega_R^-,$$
so that $\Omega_R$ is the disjoint union of at most two starlike Lipschitz subdomains.
From Lemma \ref{lem333} and the assumption that $C_1C_2\ve\ll\delta$,
it is immediate to check that
\begin{equation}\label{eqsep99}
 \dist(x,\partial\Omega)\geq \frac\delta2\,D_R(x)\quad\mbox{ for all $x\in\Omega_R$.}
\end{equation}

\vv

For a given $a>1$, we say that two cubes $Q,Q'$ are $a$-close if 
$$ a^{-1}\ell(Q)\leq\ell(Q')\leq a\ell(Q) \; \text{ and }\;\dist(Q,Q')\leq a(\ell(Q)+\ell(Q')).
$$
We say that $Q\in\DD_\sigma$ is $a$-close to $\tree(R)$ if there exists some $Q'\in\tree(R)$ such that
$Q$ and $Q'$ are $a$-close.
For $1<a^*<a^{**}$ to be fixed below, {we} define the augmented trees
$$\tree^*(R) = \{Q\in\DD_\sigma: \text{$Q$ is $a^*$-close to $\tree(R)$}\},$$
$$\tree^{**}(R) = \{Q\in\DD_\sigma: \text{$Q$ is $a^{**}$-close to $\tree(R)$}\}.$$
Obviously, $\tree(R)\subset\tree^*(R)\subset\tree^{**}(R)$. Notice also that the families of cubes from $\tree^*(R)$ or $\tree^{**}(R)$ may not be trees.
\vv

{
We consider now the decomposition of $\Omega$ into the family of Whitney cubes $\WW(\Omega)$ described in
Subsection \ref{subsecwhitney}.}

\begin{lemma}\label{lemcontingtree}
Assuming $a^*>1$ to be big enough, we have
$$\overline\Omega_R\cap\Omega\subset \bigcup_{Q\in\tree^*(R)} w(Q).$$
\end{lemma}

{Recall that $w(Q)$ is the Whitney region associated with $Q$ (see \eqref{eqwq00}).}
Notice that in case that $\Omega_R^-\neq\varnothing$, it may happen that $w(Q)$ is the union of some
Whitney cubes contained in $\Omega^+_R$ and others in $\Omega^-_R$, for example.

\begin{proof}
Let $P\in\WW(\Omega)$ be such that $P\cap \overline\Omega_R\neq \varnothing$ and let {$Q = b(P) \in \DD_\sigma$ is such that $P\subset w(Q)$.} It suffices to show that $Q\in\tree^*(R)$ if $a^*$ is taken big enough.
To this end, we have to show that there exists some $Q'\in\tree(R)$ which is $a^*$-close to $Q$.

Notice first that $\ell(P)\leq C_5\,\ell(R)$ for some fixed constant $C_5$, because $P$ intersects $\overline\Omega_R$ and thus $\overline{B(R)}$. Let $x\in P\cap \overline \Omega_R$. Then by \eqref{eqsep99} we have
$$\frac\delta2\,D_R(x)\leq \dist(x,\partial\Omega) \approx \ell(P)=\ell(Q),$$
for $Q$ as above. Thus,
$$d_R(x_Q)\approx D_R(x_Q) \leq { D_R(x)} + C\,\ell(Q) \lesssim \delta^{-1}\ell(Q).$$
From the definition of $d_R$ we infer that there exists some cube $Q'\in\tree(R)$ such that
$$\ell(Q') + \dist(Q,Q') \leq C\,\delta^{-1}\ell(Q).$$

In case that $\ell(Q')\geq C_5^{-1}\ell(R)$, we let $Q''=Q'$. Otherwise, we let $Q''$ be an ancestor of $Q'$ belonging to $\tree(R)$ and satisfying 
$$C_5^{-1}\ell(Q)\leq \ell(Q'')< 2C_5^{-1}\ell(Q).$$
The above condition $\ell(Q)=\ell(P)\leq C_5\,\ell(R)$ ensures the existence of $Q''$. Then, in any case, it easily follows  that $Q'$ is $a^*$-close to $Q$, for $a^*$ big enough depending on $\delta$.
\end{proof}
\vv
{In the rest of the lemmas in this subsection, we assume, without loss of generality, that $\Omega_R^+\subset U_R
\subset\Omega$.}

\begin{lemma}
If $Q\in\tree(R)$, then
$$\dist(w(Q),\Omega^+_R)\leq C\,\ell(Q).$$
Also, if $\Omega_R^-\neq\varnothing$, 
$$\dist(w(Q),\Omega^-_R)\leq C\,\ell(Q).$$
\end{lemma}

\begin{proof} We will prove the first statement. The second one follows by the same arguments.
It is clear that $\dist(w(Q),\Omega_R^+)\leq C\,\ell(R)$, 
and so the statement above holds if $\ell(Q)\gtrsim\ell(R)$.

So we may assume that $\ell(Q)\leq c_1\,\ell(R)$ for some small $c_1$ to be fixed below. 
By construction, the parent $\wh Q$ of $Q$ satisfies $\wh Q\cap C(R)\neq\varnothing$.
Thus there exists some point $x\in C(R)$ such that $|x-x_Q|\lesssim \ell(Q)$.
Clearly, it holds
$\dist(x,\pom)\lesssim \ell(Q).$
On the other hand, by interchanging $x$ with $(x',x_{n+1}+2\ell(Q))$ if necessary, we may assume that 
\begin{equation}\label{eqnew77}
\dist(x,\pom\cup\Gamma_R^+)\geq \ell(Q).
\end{equation}
By the definition of $d_R$ and $D_R$,
$$D_R(x_Q)\leq d_R(x_Q) \leq \ell(Q).$$
Hence,
$$D_R(x) \leq D_R(x_Q) + C\,\ell(Q)\lesssim\ell(Q) \leq \dist(x,\pom).$$
Assuming $\delta$ small enough, we deduce that
$$10\delta\,D_R(x)\leq \dist(x,\pom) \leq C\ell(Q)\leq c_1 C\ell(R).$$
By the definition of {$\Omega_R^+$, this implies that $x\in\Omega_R^+$} if $c_1$ is small enough. Indeed, since $x\in C(R)$, {by \rf{eqnew77} and the last estimate, $A(x') + \delta\,D_R(x')< x_{n+1}\ll_{c_1} \ell(R)$. }
\end{proof}
\vv

We denote 
$$\partial\tree^{**}(R) = \{Q\in\tree^{**}(R): w(Q)\not \subset\Omega_R\}.$$

\vv
\begin{lemma}\label{lemdtree**}
For all $S\in\DD_\sigma$, we have
$$\sum_{Q\in\partial\tree^{**}(R):Q\subset S} \sigma(Q) \lesssim_{a^{**}}\sigma(S).$$
\end{lemma}

\begin{proof}
We will prove the following:
\begin{claim*}
For each $Q\in\partial\tree^{**}(R)$ there exists some cube $P=P(Q)\in\WW(\Omega)$ such that
$$P\cap \pom_R\neq \varnothing,\qquad \ell(P)\approx_{a^{**},\delta} \ell(Q), \qquad \dist(P,Q) \lesssim_{a^{**},\delta}\ell(Q).$$
\end{claim*}

The lemma follows easily from this claim. Indeed, using that $\Omega_R$ is either a Lipschitz domain or a union of two Lipschitz domains, that
\begin{equation}\label{eqclaim1}
\HH^n( 2P \cap \pom_R)\approx \ell(P)^n\approx_{a^{**},\delta} \sigma(Q),
\end{equation} and 
that $P(Q)\subset B(x_S,C(a^{**})\ell(S))$
for each $Q\subset S$,
we obtain
\begin{align*}
\sum_{Q\in\partial\tree^{**}(R):Q\subset S} \sigma(Q) 
& \lesssim_{a^{**},\delta} \!\! \sum_{Q\in\partial\tree^{**}(R):Q\subset S} \HH^n( 2P(Q) \cap \pom_R)\\
& \leq_{a^{**},\delta} 
 \sum_{\substack{P'\in \WW(\Omega):\\P'\subset B(x_S,C(a^{**})\ell(S))}}\,\sum_{\substack{Q\in\partial\tree^{**}(R):\\ P'=P(Q)}}\HH^n( 2P' \cap \pom_R).
\end{align*}
From the properties \rf{eqclaim1} it easily follows that, for every $P'\in \WW(\Omega)$,
$$\#\{Q\in \partial\tree^{**}(R): P'=P(Q)\}\leq C(a^{**},\delta).$$
Taking also into account the  finite superposition of the cubes $2P'$, we deduce
\begin{align*}
\sum_{Q\in\partial\tree^{**}(R):Q\subset S} \sigma(Q) 
& \lesssim_{a^{**},\delta} \!\! \sum_{\substack{P'\in \WW(\Omega):\\P'\subset B(x_S,C(a^{**})\ell(S))}} \HH^n(2P'\cap \pom_R)\\
&  \lesssim_{a^{**}} \HH^n\big(\pom_R\cap B(x_S,C'(a^{**})\ell(S))\big) \lesssim
\ell(S)^n.
\end{align*}
\vv

To prove the claim we distinguish several cases:

\subsubsection*{ Case 1.}
 In case that $\ell(Q)\geq c(a^{**},\delta)\ell(R)$ (with
$c(a^{**},\delta)$ to be chosen below), we let $P(Q)$ be any Whitney cube that intersects the top face
of $C(R)$. It is immediate to check that this choice satisfies the properties described in \eqref{eqclaim1}.

\subsubsection*{ Case 2.}
Suppose now that $\ell(Q)\leq c(a^{**},\delta)\ell(R)$ and that $\dist(Q,\partial C(R))\geq C_6\,\ell(Q)$ for some
big $C_6(a^{**})>1$ to be chosen below. Let us see that this implies that $x_Q\in C(R)$. Indeed, from the definition of $\tree^{**}(R)$ there exists some $S\in\tree(R)$ such that $Q$ and $S$ are $a^{**}$-close. Since $S\cap C(R)\neq\varnothing$,
there exists some $\wt x_S\in S\cap C(R)$. If $x_Q\not\in C(R)$, by continuity
the segment $\overline{x_Q \,\wt x_S}$ intersects $\partial C(R)$ at some point $z$. 
So we have
\begin{align*}
\dist(x_Q,\partial C(R))& \leq |x_Q-z| \leq |x_Q- x_{\wt S}|\\ & \leq A (\ell(Q)+\ell(S)) + \diam(Q) \leq C(a^{**})\ell(Q),
\end{align*}
which contradicts the assumption above if $C_6(a^{**})$ is big enough.
 In particular, notice that the conditions that $\dist(Q,\partial C(R))\geq C_6\,\ell(Q)$ and $x_Q\in C(R)$
imply that $w(Q)\subset C(R)$ if $C_6$ is taken big enough.

If $w(Q)\cap\Omega_R\neq\varnothing$, then we take a Whitney cube $P$ with $\ell(P)=\ell(Q)$ contained in $w(Q)$ that intersects $\Omega_R$. Otherwise, $w(Q)\subset \Omega\setminus \Omega_R$, and from the fact that $\ell(Q)\leq c(a^{**},\delta)\ell(R)$ we infer that $w(Q)$ lies below the Lipschitz graph 
$\Gamma_R^+$ that defines the bottom of $\pom_R$, and above the graph $\Gamma_R^-$ in case that $\Omega_R^-\neq\varnothing$.
Then we take $x\in w(Q)$ and $x^+ =\Pi^{-1}_R(x)\cap \Gamma_R^+$, and also $x^- =\Pi_R^{-1}(x)\cap \Gamma_R^-$ 
in case that $\Omega_R^-\neq\varnothing$.

When  $\Omega_R^-=\varnothing$, we let $P$ be the Whitney that contains $x^+$. Since
 $V_R\subset\R^{n+1}\setminus \overline\Omega$, 
there exists $y = \Pi_R^{-1}(x) \cap
\pom \cap C(R)$. Then we deduce
\begin{equation}\label{eqcase1}
\ell(P)\approx\dist(x^+,\pom) \approx |x^+- y|\geq \HH^1(w(Q)\cap \Pi_R^{-1}(x^+)) \geq \ell(Q).
\end{equation}
Using again that there exists some $S\in\tree(R)$ such that $Q$ and $S$ are $a^{**}$-close we get
\begin{align}\label{eqcase11}
\ell(P)& \approx\dist(x^+,\pom)\lesssim D_R(x^+) = D_R(x) \leq D_R(x_Q) + C\,\ell(Q)\\
& \leq D_R(x_S) + |x_Q - x_S| +C\,\ell(Q) \lesssim C(a^{**},\delta)\ell(Q).\nonumber
\end{align}
Further,
\begin{equation}\label{eqcase12}
\dist(P,Q) \leq |x_Q-x^+| \leq |x_Q-x| +|x- x^+| \lesssim \ell(Q),
\end{equation}
and so $P$ satisfies the properties in the claim.

If $\Omega_R^-\neq\varnothing$ (i.e., $V_R\subset\Omega$), we let $P$ be the largest Whitney cube that intersects $\{x^+,x^-\}$.
From the fact that $b\beta_\sigma(k_1Q)\lesssim\ve$ and the stopping condition (c) we easily infer that there
exists some point $y\in  \Pi_R^{-1}(x^+) \cap C(R)$ such that $\dist(y,\pom)\lesssim\ve\ell(Q)$.
Then it follows
\begin{equation}\label{eqcase2}
\ell(P) \gtrsim |x^+ - x^-|\geq \HH^1(w(Q)\cap \Pi_R^{-1}(x^+)) \geq \ell(Q).
\end{equation}
Also the estimates \eqref{eqcase11} and \eqref{eqcase12} are still valid, replacing $x^+$ by $x^-$ if $x^-\in P$. So again $P$ satisfies the required properties.

\subsubsection*{ Case 3.}
Suppose that $\ell(Q)\leq c(a^{**},\delta)\ell(R)$ and that $\dist(Q,\partial C(R))< C_6\,\ell(Q)$ for
$C_6(a^{**})>1$ as above.
From the smallness of $\ell(Q)$ (remark that the constant $c(a^{**},\delta)$ is chosen after fixing $C_6$)
we infer that 
$\dist(Q,\partial_l C(R)) = \dist(Q,\partial C(R)),$
where $\partial_l C(R)$ is the lateral part of $\partial C(R)$. So there exists some point $z\in
\partial_l C(R)$ such that $|x_Q-z|\lesssim_{a^{**}} \ell(Q)$. We also denote $z^+= \Pi^{-1}(z) \cap\Gamma_R^+$.
As above, we take $S\in\tree(R)$ such that $Q$ and $S$ are $a^{**}$-close, so that by the $1$-Lipschitzness of $D_R$ we have
$$\dist(z^+, \pom)\approx D_R(z^+)=
D_R(z) \leq D_R(x_S) + |z-x_Q| + |x_Q - x_S|\lesssim_{a^{**}} \ell(Q).$$
From this fact we infer that 
there exists some point $y\in\Pi^{-1}(z)\cap \pom_R$ such that $|z-y|\lesssim_{a^{**}}\ell(Q)$ and $\dist(y,\pom)\approx_{a^{**}}\ell(Q)$.
For example, we can take
$y= (z',z_{n+1}+ C(a^{**})\ell(Q))$ for a suitable $C(a^{**})>1$. This point satisfies
$$|x_Q -y|\leq |x_Q - z| + |z-y| \lesssim_{a^{**}} \ell(Q),$$ 
and so letting $P$ be the Whitney cube that contains $y$ we are done.
\end{proof}

\vv


\subsection{The corona decomposition of $\Omega$}\label{subs:cordecOmega}

We will now perform a corona decomposition of $\Omega$ using the Lipschitz subdomains $\Omega_R$ constructed above.
We define inductively a family $\ttt\subset \DD_\sigma$ as follows. First we let $R_0\in\DD_\sigma$ be a cube {such that 
$b\beta_\sigma(k_1R_0)\leq\ve$ having maximal side length.} Assuming $R_0,R_1,\ldots, R_i$ to be defined, we let $R_{i+1}\in\DD_\sigma$ be a cube from 
$$\DD_\sigma \setminus \bigcup_{0\leq k\le i} \tree^{**}(R_k)$$
such that $b\beta_\sigma(k_1R_{i+1})\leq \ve$ with maximal side length.
We set
$$\ttt=\{R_i\}_{i\geq0}.$$

For each $R\in\ttt$ we consider the subdomain $\Omega_R$ constructed in the previous subsection. 
\vv

\begin{lemma}\label{lem3.7**}
The sets $\overline\Omega_R\cap\Omega$, with $R\in\ttt$, are pairwise disjoint, assuming that the {constant $a^{**}$ is big enough (possibly depending on $a^*$). }
\end{lemma}

Remark that the constants {$a^*$ and $a^{**}$} depend on $\delta$. However, this dependence
is harmless for our purposes.

In particular, from this lemma it follows that, if we denote $H=\Omega \setminus\bigcup_{R\in\ttt} \Omega_R$, then
the union 
\begin{equation}\label{eqpart93}
\Omega = \bigcup_{R\in\ttt} \Omega_R \cup H.
\end{equation}
is a partition of $\Omega$ into 
disjoint sets.

\begin{proof}[Proof of Lemma \ref{lem3.7**}]
Suppose that $R,R'\in\ttt$ are such that {$\overline\Omega_R\cap\overline\Omega_{R'}\cap\Omega\neq\varnothing$.}
Suppose also that $R=R_i$, $R'=R_j$, with $j>i$, so that in particular $\ell(R')\leq \ell(R)$.
From Lemma \ref{lemcontingtree} we infer that there exist cubes $Q\in\tree^*(R)$ and $Q'\in\tree^*(R')$ such that $w(Q)\cap w(Q')\neq\varnothing$. Clearly, this implies that $\ell(Q)\approx \ell(Q')$, and from the definition of
$\tree^*(R)$ and $\tree^*(R')$ we deduce that there are two cubes $S\in\tree(R)$, $S'\in\tree(R')$
such that $\dist(S,S')\lesssim_a\ell(S)\approx_a\ell(S')$. Let $\wt S$ be the ancestor of $S$ with 
$\ell(\wt S)=\ell(R')$ {(or take $\wt S=S$ if $\ell(S)>\ell(R')$). Clearly, $\wt S\in\tree(R)$} and $\dist(\wt S,R')\lesssim_{a^*}\ell(R')$.
So $R'\in\tree^{**}(R)$ {if $a^{**}=a^{**}(a^*)$} is chosen big enough, which contradicts the construction of $\ttt$.
\end{proof}

\vv

\begin{lemma}\label{lempack2}
The family $\ttt$ satisfies the packing condition
$$\sum_{R\in\ttt:R\subset S} \sigma(R)\lesssim_{\ve,\delta} \sigma(S)\quad \mbox{ for all $S\in\DD_\sigma$}.$$
\end{lemma}

\begin{proof}
By Lemma \ref{lempack1} we have
\begin{align}\label{eqlala77}
\sum_{R\in\ttt:R\subset S} \sigma(R) & \lesssim
\sum_{R\in\ttt:R\subset S} \sigma(Z(R)) 
+  \sum_{R\in\ttt}\,\sum_{Q\in\sss(R)\cap\cB(\ve)} \sigma(Q)\\
& \quad
+ \sum_{R\in\ttt:R\subset S}\sum_{Q\in\tree(R)} \beta_{\sigma,1}(k_1Q)^2\,\sigma(Q).\nonumber
\end{align}
By construction, the sets $Z(R)$ are disjoint, and thus the first sum does not exceed $\sigma(S)$.
The second term does not exceed
$$\sum_{Q\in\DD_\sigma(S)\cap\cB(\ve)}\sigma(Q)\lesssim_\ve\sigma(S),$$
by {the} uniform rectifiability of $\pom$.
Concerning the last term in \eqref{eqlala77}, the families $\tree(R)$, with $R\in\ttt$, are also disjoint by construction. Therefore, {again by the} uniform rectifiability of $\pom$,
$$\sum_{R\in\ttt:R\subset S}\sum_{Q\in\tree(R)} \beta_{\sigma,1}(k_1Q)^2\,\sigma(Q) \leq
\sum_{Q\subset S} \beta_{\sigma,1}(k_1Q)^2\,\sigma(Q)\lesssim_{\ve,\delta} \sigma(S).$$
\end{proof}

\vv

\begin{lemma}\label{lem:HCarleson}
There is a subfamily $\HH\subset \DD_\sigma$ such that
\begin{equation}\label{eqHH*}
H\subset \bigcup_{Q\in\HH} w(Q)
\end{equation}
which satisfies the packing condition
\begin{equation}\label{eqHH*2}
\sum_{Q\in \HH:Q\subset S} \sigma(Q)\lesssim \sigma(S)\quad \mbox{ for all $S\in\DD_\sigma$},
\end{equation}
with the implicit constant depending on $\ve,\delta,a^{**}$.
\end{lemma}

\begin{proof}
By construction, 
$$\DD_\sigma \subset \cB\cup \bigcup_{R\in\ttt} \tree^{**}(R),$$
and thus
$$\Omega \subset \bigcup_{Q\in\cB} w(Q) \cup \bigcup_{R\in\ttt}\,\bigcup_{Q\in\partial\tree^{**}(R)}w(Q) \cup \bigcup_{R\in\ttt} \Omega_R.$$
So \eqref{eqHH*} holds if we define
$$\HH := \cB \cup \bigcup_{R\in\ttt} \partial\tree^{**}(R).$$

It remains to prove the packing condition \eqref{eqHH*2}. From the uniform rectifiability of $\pom$ 
we know that the family $\cB$ satisfies a Carleson packing condition, and so it suffices to show that
the same holds for
$\bigcup_{R\in\ttt} \partial\tree^{**}(R)$. This is an immediate consequence of Lemmas \ref{lemdtree**} 
and \ref{lempack2}. Indeed, for any $S\in\DD_\sigma$, let
$$T_0=\{R\in\ttt:\tree^{**}(R)\cap \DD_\sigma(S)\neq \varnothing\}$$
and
$$T_1=\{R\in T_0:\ell(R)\leq \ell(S)\},\qquad T_2=\{R\in T_0:\ell(R)> \ell(S)\},$$
so that
\begin{align*}
\sum_{R\in\ttt}\,\sum_{Q\in\partial\tree^{**}(R)\cap\DD_\sigma(S)} \sigma(Q) &\leq 
\sum_{R\in T_1}\, \sum_{Q\in\partial\tree^{**}(R)} \sigma(Q) \\
&\quad +\sum_{R\in T_2}\,\sum_{Q\in\partial\tree^{**}(R)\cap\DD_\sigma(S)} \sigma(Q). 
\end{align*}
Since all the cubes from $\partial\tree^{**}(R)$ are contained in $C(a^{**}) R$, it follows that the cubes from
$T_1$ are contained in $C'(a^{**})S$, and thus
$$\sum_{R\in T_1}\, \sum_{Q\in\partial\tree^{**}(R)} \sigma(Q) \lesssim_{a^{**}}\sum_{R\in T_1}\sigma(R)
\lesssim_{\ve,\delta}\sigma(S).$$
Also, it is immediate to check that the number of cubes from $T_2$ is uniformly bounded by some constant
depending on $a^{**}$. Therefore, 
$$\sum_{R\in T_2}\,\sum_{Q\in\partial\tree^{**}(R)\cap\DD_\sigma(S)} \sigma(Q)\lesssim_{a^{**}}
\sum_{R\in T_2}\sigma(S)\lesssim_{a^{**}}\sigma(S).$$
\end{proof}

\vv

Below we will need the following auxiliary result too.

\begin{lemma}\label{lemsurfacetot}
Given any $\alpha>0$, for $\sigma$-a.e.\ $\xi\in\pom$ there exists 
$r_{\xi,\alpha}>0$ such that 
\begin{equation}\label{eqinch0}
\gamma_\alpha(\xi)\cap B(\xi,r_{\xi,\alpha}) \subset \bigcup_{R\in\ttt}\Omega_R.
\end{equation}
Further, $\sigma$-a.e.\ $\xi\in\pom$ belongs to $\bigcup_{R\in\ttt}\partial\Omega_R^{\pm}$.
\end{lemma}

\begin{proof}
To prove the first statement, notice that if  $\gamma_\alpha(\xi)\cap  H\neq\varnothing$, then there exists
some $Q\in\HH$  such that $ \gamma_\alpha(\xi)\cap w(Q)\neq \varnothing$, which
implies that $\dist(\xi,Q)\lesssim_\alpha \ell(Q)$, and so $\xi \in C(\alpha)Q$, for some $C(\alpha)\geq1$.
From the Carleson condition on $\HH$, we know that $\sum_{Q\in\HH} \sigma(C(\alpha)Q)\lesssim_\alpha \sigma(\pom)$, and so  by 
Chebychev's inequality,  
$$
\sigma\Big( \Big\{\xi \in \pom: \sum_{Q \in \HH} \chi_{C(\alpha)Q}(\xi) >N  \Big\}  \Big)\leq N^{-1} \int \sum_{Q \in \HH} \chi_{C(\alpha)Q}\,d \sigma \lec_\beta N^{-1} \sigma(\pom).
$$
By taking $N \to \infty$, we deduce that for $\sigma$-a.e.\ $\xi\in\pom$,  there is a finite number of cubes $Q\in\HH$ such that $\xi\in C(\alpha)Q$. 
So, by the definition of $H$, for such points $\xi$ there exists some $r_{\xi,\alpha}>0$
such that $\gamma_\alpha(\xi)\cap B(\xi,r_{\xi,\alpha}) \cap H=\varnothing$, which yields \rf{eqinch0}.

The inclusion \rf{eqinch0} implies that $\sigma$-a.e.\ $\xi\in\pom$ belongs to $\bigcup_{R\in\ttt}\partial\Omega_R^{\pm}$.
\end{proof}

\vv


\section{The almost harmonic extension of Lipschitz functions on the boundary}\label{sec:extension}

In this section we will assume that $\Omega$ is a domain in $\R^{n+1}$ satisfying the properties in Theorem \ref{teomain}. That is, $\Omega$ is a bounded corkscrew domain with $n$-AD-regular boundary 
and  $(D_{p'})$ is solvable 
for some  $p \in (1, 2+\ve_0)$. {As remarked at the beginning of Section \ref{secorona}, this implies that $\partial\Omega$ is uniformly rectifiable.

Let $f:\partial\Omega\to \R$ be a Lipschitz function, so that in particular $f\in W^{1,p}(\pom)$. In this section we will construct
an ``almost harmonic extension'' of $f$ to $\Omega$. To this end, first we need to define another auxiliary extension, following 
some ideas from \cite{AMV}.

Given a ball $B\subset\R^{n+1}$ centered in $\pom$ and an affine map $A:\R^{n+1}\to\R$, we consider the coefficient
$$\gamma_{f}(B) :=\inf_A\left( |\nabla A| + \avint_B \frac{|f-A|}{r(B)}\,d\sigma\right),$$
where the infimums are taken over all affine maps $A:\R^{n+1}\to\R$. We denote by $A_B$ an affine map that minimizes $\gamma_{f}(B)$.

Recall that, by Proposition \ref{propopoin}, the following Poincaré inequality holds, for every ball 
 $B$ centered in $\pom$:
\begin{equation}\label{eq:veryweakPoinc}
\avint_B |f- m_{B, \sigma}f|\,d\sigma \lesssim r(B)\,m_{B,\sigma}(\nabla_{H,p} f),
\end{equation}
where $\nabla_{H,p} f$ denotes the optimal Haj\l asz upper $p$-integrable gradient for $f$.
\vv

\begin{lemma}
The following properties hold:
\begin{itemize}
\item[(a)]
For every ball $B$ centered in $\pom$ with $r(B)\leq \diam(\Omega)$,
\begin{equation}\label{eqprop00}
\gamma_{f}(B)\lesssim m_{B, \sigma}(|\nabla_{H,p} f|)
\end{equation}
and
\begin{equation}\label{eqprop11}
|\nabla A_B|\lesssim m_{B,\sigma}(\nabla_{H,p} f).
\end{equation}

\item[(b)] If $B,B'$ are balls centered in $\pom$ such that $B\subset B'$ with $r(B)\approx r(B')\leq \diam(\Omega)$, then
\begin{equation}\label{eqprop22}
|A_B(x) - A_{B'}(x)| \lesssim m_{B',\sigma}(\nabla_{H,p} f)\,\big(r(B) + \dist(x,B)\big).
\end{equation}
\end{itemize}
The implicit constants in the above estimates depend only on the implicit constant in the Poincaré inequality \eqref{eq:veryweakPoinc} and on the 
AD-regularity of $\pom$.
\end{lemma}

The estimates in this lemma follow easily from some similar (but sharper and more general) calculations from
\cite{AMV} for some related coefficients. However, for the reader's convenience we provide the detailed arguments here.

\begin{proof}
To prove \rf{eqprop00}, we just use the affine constant function $A:=m_{\sigma,B}f$ as a competitor in the definition of $\gamma_f$ and we apply 
\eqref{eq:veryweakPoinc} to get 
$$\gamma_{f}(B)\leq 0 + \avint_B \frac{|f-m_{\sigma,B}f|}{r(B)}\,d\sigma \lesssim m_{B,\sigma}(\nabla_{H,p} f).$$
The estimate \rf{eqprop11} is now an immediate consequence of the definition of $\gamma_f(B)$ and \rf{eqprop00}:
$$|\nabla A_B| \leq \gamma_f(B)\lesssim m_{B,\sigma}(\nabla_{H,p} f).$$

Finally we deal with (b). For $B$ and $B'$ as in (b), we have
\begin{align*}
\avint_B \frac{|A_B(x) - A_{B'}(x)|}{r(B)}\,d\sigma(x) & \lesssim \avint_B \frac{|f(x) -A_B(x)|}{r(B)}\,d\sigma(x)\\& \quad + 
\avint_{B'} \frac{|f(x) -A_{B'}(x)|}{r(B')}\,d\sigma(x)\\
& \lesssim m_{B,\sigma}(\nabla_{H,p} f) + m_{B',\sigma}(\nabla_{H,p} f) \lesssim 
m_{B',\sigma}(\nabla_{H,p} f).
\end{align*}
Thus, there exists some $x_0\in\supp\sigma \cap B$ such that
$$|A_B(x_0) - A_{B'}(x_0)| \lesssim r(B)\,
m_{B',\sigma}(\nabla_{H,p} f).$$
Consequently, for any $x\in\R^{n+1}$, using also the property \rf{eqprop11},
\begin{align*}
|A_B(x) - A_{B'}(x)| & = \big|A_B(x_0) - A_{B'}(x_0) + \nabla (A_B - A_{B'})\cdot(x-x_0)\big|\\
& \leq C\,r(B)\,
m_{B',\sigma}(\nabla_{H,p} f) + \big(|\nabla A_B| + |\nabla A_{B'}|\big)\, |x-x_0| \\
& \lesssim 
m_{B',\sigma}(\nabla_{H,p} f)\,(r(B) + |x-x_0|),
\end{align*}
which gives (b).
\end{proof}
\vv

Next, for each Whitney cube $P\in \WW(\Omega)$ we consider 
a $C^\infty$ bump function $\vphi_P$ supported on $1.1P$ such that the functions $\vphi_P$, $P\in\WW(\Omega)$,
form a partition of unity of $\chi_\Omega$. That is,
$$\sum_{P\in\WW(\Omega)}\vphi_P = \chi_\Omega.$$
We define the extension $\wt f:\overline \Omega\to\R$ of $f$ as follows:
$$\wt f|_\pom =f,\qquad\quad\wt f|_\Omega = \sum_{P\in\WW(\Omega)} \vphi_P\,A_{2B_{b(P)}}.$$
{It is clear that $\wt f$ is smooth in $\Omega$.}
Recall that $b(P)$ is the (unique) boundary cube from $\DD_\sigma$ associated with $P$ and that
$B_{b(P)}$ is a ball concentric with $b(P)$ that contains $b(P)$. See Section \ref{secprelim}.
Abusing notation we will also write
$$A_{2B_{b(P)}}\equiv A_{b(P)} \equiv A_P.$$

\vv
\begin{lemma}\label{lem:Lip-ext}  
Given a Lipschitz function $f:\pom\to\R$, the extension $\wt f$ is Lipschitz {in $\overline \Omega$}, with ${\rm Lip}(\wt f)\lesssim {\rm Lip}(f)$.
\end{lemma}

\begin{proof}
We will show that $\wt f$ is Lipschitz in $\Omega$ and continuous in $\pom$, so that it is globally Lipschitz in $\overline\Omega$. 

Given $x,y\in\Omega$, let $P_1,P_2\in\WW(\Omega)$ be such that $x\in P_1$, $y\in P_2$.
Assume first the $2P_1\cap 2P_2\neq \varnothing$, and let $P_0\in
\WW(\Omega)$ be a cube with minimal side length such that 
$$2B_{b(P)}\subset 2B_{b(P_0)}
\quad\mbox{ for any $P\in\WW(\Omega)$ such that $1.1P\cap (P_1\cup P_2)\neq \varnothing$}
.$$
 In this case $\ell(P_1)\approx\ell(P_2)\approx\ell(P_0)$, and we write
\begin{align}
\wt f(x) - \wt f(y) &= 
\sum_{P\in\WW(\Omega)} \vphi_P(x)\, A_P(x)-  \sum_{P\in\WW(\Omega)} \vphi_P(y)\, A_P(y) \label{eqal823}\\
& = \!\sum_{P\in\WW(\Omega)} \!\!\vphi_P(x)\,\big(A_P(x) - A_P(y)\big)\nonumber
\\ &\quad +
\sum_{P\in\WW(\Omega)} \!\!\big(\vphi_P(x) - \vphi_P(y)\big)\, (A_P(y) - A_{P_0}(y)).\nonumber
\end{align}
To bound the first sum on the right hand side we use \rf{eqprop11}:
\begin{align*}
\sum_{P\in\WW(\Omega)} \vphi_P(x)\,\big|A_P(x) - A_P(y)\big| & \leq 
\sum_{P\in\WW(\Omega)} \vphi_P(x)\,\big|\nabla A_P|\,|x-y| \\
& \lesssim \sum_{P\in\WW(\Omega)} \vphi_P(x) \,m_{2B_{b(P)},\sigma}(\nabla_{H,p} f)
\,|x-y| \\ &\lesssim {\rm Lip}(f) \,|x-y|,
\end{align*}
where in the last inequality we used that, since $f$ is Lipschitz in $\Omega$, then 
$\nabla_{H,p}f\leq {\rm Lip}(f)$, by \rf{eqlip12}.
 To deal with the last sum on the right hand side of \rf{eqal823} we may assume that the cubes $P$ appearing in the sum are such that either $1.1P\cap P_1\neq\varnothing$ or
$1.1P\cap P_2\neq\varnothing$, since otherwise the associated summand vanishes. We denote by $I_0$ the family of such cubes.
So the cubes from $I_0$ are such that $B_{b(P)}\subset 2B_{b(P_0)}$ and they satisfy $\ell(P)\approx\ell(P_0)$. 
Then by
\rf{eqprop22},
$$|A_P(y) - A_{P_0}(y)| \lesssim 
m_{2B_{b(P_0)},\sigma}(\nabla_{H,p} f)
\,\ell(P)\leq{\rm Lip}(f) \,\ell(P).$$
Thus, {since $\vphi_P$ is smooth and $\|\nabla \vphi_P\|_\infty \lesssim \ell(P)^{-1}$, we have that}
\begin{align*}
\sum_{P\in\WW(\Omega)} \big|&\vphi_P(x) - \vphi_P(y)\big|\, \big|A_P(y) - A_{P_0}(y)\big| \\ &
\lesssim \sum_{P\in I_0} \frac{|x-y|}{\ell(P)}\,\ell(P_0)\,m_{2B_{b(P_0)},\sigma}(\nabla_{H,p} f)\lesssim {\rm Lip}(f) \,|x-y|,
\end{align*}
{where we used that the number of cubes $ P  \in I_0$ is at most a uniform  dimensional constant.} Hence,
\begin{equation}\label{eqlip99}
|\wt f(x) - \wt f(y)|\lesssim {\rm Lip}(f) \,|x-y|.
\end{equation}

{Next we turn our attention to the case   $2P_1\cap2P_2 = \varnothing$ and write}
\begin{align}\label{eqal824}
\wt f(x) - \wt f(y) & = 
\sum_{P\in\WW(\Omega)} \vphi_P(x)\, \big(A_P(x) - A_{P_1}(x)\big) \\&\quad - \sum_{P\in\WW(\Omega)} \vphi_P(y)\, \big(A_P(y) - A_{P_2}(y)\big) + \big(A_{P_1}(x) - A_{P_2}(y)\big)\nonumber\\
& \quad=: S_1+S_2+S_3.\nonumber
\end{align}
{If we use again \rf{eqprop22} and the fact that $|x-y|\gtrsim\ell(P_1)$, we infer that
$$|S_1| \lesssim \sum_{P\in\WW(\Omega)} \vphi_P(x)\, 
m_{2B_{b(P)},\sigma}(\nabla_{H,p} f) \,\ell(P_1) \lesssim {\rm Lip}(f)\,|x-y|.$$
and, by an analogous estimate, $|S_2|  \lesssim  {\rm Lip}(f)\,|x-y|$.
Finally,}
\begin{align*}
\big|A_{P_1}(x) - A_{P_2}(y)\big| & \leq \big|A_{P_1}(x) - m_{2B_{b(P_1)}, \sigma}(A_{P_1})\big| +
\big|m_{2B_{b(P_1)}, \sigma}(A_{P_1} - f)\big|\\
& \; + \big|m_{2B_{b(P_1)}, \sigma}(f) - m_{2B_{b(P_2)}, \sigma}(f)\big|
+ \big|m_{2B_{b(P_2)}, \sigma}(f- A_{P_2})\big|\\
&\; + \big|m_{2B_{b(P_2)}, \sigma}(A_{P_2})- A_{P_2}(y) \big|. 
\end{align*}
By \rf{eqprop00} and \rf{eqprop11},
\begin{multline*}
\big|A_{P_1}(x) - m_{2B_{b(P_1)}, \sigma}(A_{P_1})\big| +
\big|m_{2B_{b(P_1)}, \sigma}(A_{P_1} - f)\big| \\ \lesssim |\nabla A_{P_1}|\,\ell(P_1) + m_{2B_{b(P_1)}, \sigma}(\nabla_{H,p} f)\,\ell(P_1) \lesssim {\rm Lip}(f)\,|x-y|.
\end{multline*}
The same estimate holds replacing $x$ by $y$ and $P_1$ by $P_2$.
Therefore,
\begin{align*}
\big|A_{P_1}(x) - A_{P_2}(y)\big| & \lesssim {\rm Lip}(f)\,|x-y| + \big|m_{2B_{b(P_1)}, \sigma}(f) - m_{2B_{b(P_2)}, \sigma}(f)\big|\\
&  \lesssim {\rm Lip}(f)\,|x-y| + {\rm Lip}(f)\,\big(|x-y| +\ell(P_1) + \ell(P_2)\big) \\ &\approx {\rm Lip}(f)\,|x-y|.
\end{align*}
This shows that \rf{eqlip99} holds also in this case and thus $\wt f$ is Lipschitz in $\Omega$.

Finally it remains to show that $\wt f|_\Omega$ extends continuously to $\pom$, and that the extension coincides with $f$.
For $x\in\pom$ and $y\in \Omega$ such
that $|x-y|\leq \delta$, we write
$$|\wt f(x) - \wt f(y)| = \Big|f(x)- \sum_{P\in\WW(\Omega)}\vphi_P(y)\,A_P(y)\Big| \leq
\sum_{P\in\WW(\Omega)}\vphi_P(y)\,|f(x) - A_P(y)|.$$
Notice that any cube $P\in\WW(\Omega)$ such that $y\in 1.1P$ is contained in $B(x,C\delta)$, for a suitable $C$ depending on the parameters of the Whitney decomposition. 
Then, for all $x$, $y$, and $P$ as in the last sum, using that $f$ is Lipschitz and that $\ell(P)\leq\delta$, together with \rf{eqprop00} and
\rf{eqprop11},
\begin{align*}
|f(x) - A_P(y)|& \leq |f(x) - m_{2B_{b(P)}, \sigma}f| + |
m_{2B_{b(P)}, \sigma}( f - A_P) |\\ &\quad + | m_{2B_{b(P)}, \sigma}(A_P)-A_P(y)\Big|\\
& \lesssim 
{\rm Lip}(f)\,\delta + 
m_{2B_{b(P)}, \sigma}(\nabla_{H,p} f)\,\delta\lesssim {\rm Lip}(f)\,\delta.
\end{align*}
Thus,
$|\wt f(x) - \wt f(y)|\lesssim {\rm Lip}(f)\,\delta.$
Clearly, the same estimate also holds in case that $x,y\in\pom$, and so $\lim_{\overline{\Omega}\ni y\to x} \wt f(y) =\wt f(x)$, as wished.

\end{proof}

\vv

Given the corona construction in terms of the family $\ttt$ from Section \ref{secorona},
for each $R\in\ttt$ we denote by $v_R$ the solution of the Dirichlet problem in $\Omega_R$ with boundary data $\wt f|_{\pom_R}$.
We define the function $v:\overline \Omega\to \R$ by
$$v = \left\{\begin{array}{ll} \wt f & \text{in $\overline\Omega \setminus \bigcup_{R\in\ttt}\Omega_R$,}\\& \\ v_R & \text{in each $\Omega_R$, with $R\in\ttt$,}
\end{array}
\right.
$$
and we call it  {\it almost harmonic extension of $f$}. Notice that $v$ is well defined because of the disjointness 
of the domains $\Omega_R$ proved in Lemma \ref{lem3.7**}.

\vv

\begin{lemma}\label{lem190}
Assume that $f$ is Lipschitz on $\pom$. Then the almost harmonic extension $v$ is continuous in $\overline\Omega$. Also,
$v\in \dot W^{1,2}(\Omega)$, with
\begin{align}\label{eqCC5bis}
\|\nabla v\|_{L^2(\Omega)}^2 \lesssim {\rm Lip}(f)^2\,m(\Omega).
\end{align}
\end{lemma}

\vv

\begin{remark}\label{remnab}
Let $\Omega_0\subset \R^{n+1}$ be a bounded Wiener regular domain.
Given $f\in {\rm Lip}(\pom_0)$, we can consider an extension of $f$ to the whole $\R^{n+1}$ with the same Lipschitz norm, which we continue to 
denote by $f$.
Let $u_f:\Omega_0\to\R$ be the solution of the {continuous} Dirichlet problem in $\Omega_0$ with boundary data $f$. {By  \cite[Corollary 9.29]{HKM}, it holds that $u_f$ is the unique solution of the variational Dirichlet problem, i.e., $\Delta u=0$ in $\Omega$ in the weak sense, $u_f \in W^{1,2}(\om_0)$, and $u_f-f \in W^{1,2}_0(\Omega_0)$.} Since $u_f$ minimizes the Dirichlet energy in $\Omega_0$, we deduce that
$$\|\nabla u_f\|_{L^2(\Omega_0)} \leq \|\nabla f\|_{L^2(\Omega_0)} \leq \|\nabla f\|_{L^\infty(\Omega_0)}
\,\,m(\Omega_0)^{1/2}\leq
{\rm Lip}(f)\,m(\Omega_0)^{1/2}.$$
Notice the analogy between the previous estimate and \rf{eqCC5bis}.
\end{remark}
\vv

\begin{proof}[Proof of Lemma \ref{lem190}]
First we will show that $v$ is continuous. Since the domains $\Omega_R$ are Lipschitz and thus Wiener regular, it is clear that $v$ is continuous in each set $\overline \Omega_R$.
Then taking into account that $v=\wt f$ in $H$ and that $\wt f$ is Lipschitz, it just remains to
show that $v$ is continuous {on} $\pom$.
To this end, we fix $x\in\pom$ and $\ve>0$. Let $\ttt_\ve$ be the family of cubes $R\in\ttt$ such that
$\ell(R)\geq\ve$. Since this is a finite family, there is some $\delta>0$ such that
\begin{equation}\label{eqapp90}
|v_R(y) - v_R(z)|\leq \ve \quad\mbox{ for all $y,z\in\overline\Omega_R$ such that $|y-z|\leq\delta$, $R\in\ttt_\ve$.}
\end{equation}

{Set $\delta':=\min(\delta,\ve)$ and for $y\in \overline \Omega\cap B(x,\delta')$,  let us estimate $|v(x) - v(y)|$.} If 
$y\not \in\bigcup_{R\in\ttt}\Omega_R$, then
\begin{equation}\label{eqapp910}
|v(x) - v(y)| = |\wt f(x) - \wt f(y)|\leq {\rm Lip}(\wt f)\,|x-y| \lesssim {\rm Lip}(f)\,\delta'
\leq {\rm Lip}(f)\,\ve.
\end{equation}
If $y\in\Omega_R$ for some $R\in\ttt_\ve$, then we consider some $z\in \overline \Omega\cap \pom_R\cap B(x,\delta')$ (it is easy to check that such $z$ exists) and,  since  $v(x) - v(z)=\wt f(x)-\wt f(z)$, we may apply Lemma \ref{lem:Lip-ext}  and  \rf{eqapp90}  to get that
\begin{equation}\label{eqapp91}
|v(x) - v(y)|\leq |v(x) - v(z)| + |v_R(z) - v_R(y)| \leq  {\rm Lip}(\wt f)\,|x-z| + \ve \lesssim 
({\rm Lip}(f) + 1) \ve.
\end{equation}
Suppose now that  $y\in\Omega_R$ for some $R\in\ttt\setminus\ttt_\ve$. In this case, we have
$\Omega_R\subset B(R) \subset B(x,C\ve)$ because $\dist(x,\Omega_R)\leq |x-y|\leq \ve$ and $\diam(\Omega_R)\lesssim\ve$ (since $\ell(R)<\ve$). Let $\xi_{\min}$ and $\xi_{\max}$ the points in $\pom_R$ where $v_R$  attains
its minimum and  maximum, respectively. Then we have
$$\wt f(\xi_{\min}) - \wt f(x)= v(\xi_{\min}) - v(x) \leq v(y) - v(x) \leq v(\xi_{\max}) - v(x)=
\wt f(\xi_{\max}) - \wt f(x).$$
Since $\xi_{\min},\xi_{\max}\in B(x,C\ve)$ and $\wt f$ is Lipschitz in $\overline\Omega\setminus \bigcup_{S\in\ttt}\Omega_S$, we deduce that
$$-{\rm Lip}(\wt f)\,\ve\lesssim v(y) - v(x) \lesssim {\rm Lip}(\wt f)\,\ve,$$
and so $|v(x)-v(y)|\lesssim {\rm Lip}(f)\,\ve$.
Together with \rf{eqapp910} and \rf{eqapp91} this shows that 
$$|v(x) -v(y)|\lesssim ( {\rm Lip}(f)+1)\,\ve\quad \mbox{ for all $y\in \Omega$ such that $|x-y|\leq\delta'$,}$$
 which concludes the proof that $v$ is continuous in $x$, and thus in the whole $\overline \Omega$.
 
It remains to estimate $\|\nabla v\|_{L^2(\Omega)}$. 
Since $v$ is Lipschitz on $\pom_R$, for any $R\in\ttt$, by Remark \ref{remnab} it follows that
$$\|\nabla v\|_{L^2(\Omega_R)}^2 \leq
{\rm Lip}(\wt f)\,m(\Omega_R)\lesssim {\rm Lip}( f)\,m(\Omega_R).$$
Taking into account that $v$ is continuous in $\Omega$ and Lipschitz in $\overline\Omega \setminus \bigcup_{R\in\ttt}\Omega_R$, we infer that $v\in W_{\textup{loc}}^{1,2}(\Omega)$, and moreover
\begin{align*}
\|\nabla v\|_{L^2(\Omega)}^2 & \lesssim {\rm Lip}(\wt f)^2\,m(H) + \sum_{R\in\ttt} \|\nabla v\|_{L^2(\Omega_R)}^2\\
&\leq {\rm Lip}(f)^2\,m(H) + \sum_{R\in\ttt} {\rm Lip}(\wt f)^2\,m(\Omega_R) \lesssim {\rm Lip}(f)^2\,m(\Omega),
\end{align*}
{concluding the proof of the lemma}.
\end{proof}
\vv

\begin{lemma}\label{lem45*}
In the distributional sense,
$$\Delta v = \sum_{R\in\ttt} (\partial_{\nu_R}\wt f - \partial_{\nu_R}v_R) \,\HH^n|_{\partial\Omega_R\setminus\partial\Omega} + \chi_H\,\Delta\wt f 
\quad \mbox{ in $\Omega$,}$$
where $\nu_R$ stands for the outer {unit normal to  $\pom_R$}.
\end{lemma}

\begin{proof}
 Let $\phi$ be a $C^\infty$ function supported on some ball $B\subset\subset \Omega$.
Denote by $I$ the family of cubes  $R\in\ttt$ such that $\overline{\Omega}_R\cap B\neq\varnothing$. Notice that
$I$ is a finite family. Then we write
$$\langle \Delta v,\phi\rangle = \int v \,\Delta\phi \,dm = 
\sum_{R\in I} \int_{B\cap\Omega_R} v_R \,\Delta\phi \,dm + \int_H v\,\Delta\phi\,dm.$$
Notice now that $B\cap\Omega_R$ and $B\cap H$ are piecewise {Lipschitz} domains and that $v_R$ is harmonic in $\Omega_R$.
Hence, by Green's formula, 
\begin{align*} 
\int v \,\Delta\phi \,dm & =
\sum_{R\in I} \int_{B\cap\partial\Omega_R} v_R \,\partial_{\nu_R}\phi \,d\HH^n - \sum_{R\in I} \int_{B\cap\partial\Omega_R} \partial_{\nu_R}v_R\,\phi \,d\HH^n\\
&\quad  + \int_H \Delta \wt f \,\phi \,dm+\int_{B\cap\partial H} \wt f \,\partial_{\nu_H}\phi \,d\HH^n - \int_{B\cap\partial H} \partial_{\nu_H}\wt f\,\phi \,d\HH^n,
\end{align*}
where $\nu_H$ is the {outer unit normal to  $B\cap \partial H$}.
Using the continuity of $v$ in $\Omega$, it is clear that
$$\sum_{R\in I} \int_{B\cap\partial\Omega_R} v_R \,\partial_{\nu_R}\phi \,d\HH^n=
-\int_{B\cap \partial H} \wt f \,\partial_{\nu_H}\phi \,d\HH^n,$$
{since  $\nu_R(x) = -\nu_H(x)$ $\HH^n$-a.e.\ $x \in \partial H$, $v_R=\wt f$ on $\pom_R$ for every $R \in \ttt$, and  $\phi$ is smooth.}
Hence,
$$\langle \Delta u,\phi\rangle = 
 - \sum_{R\in I} \int_{B\cap\partial\Omega_R} \partial_{\nu_R}v_R\,\phi \,d\HH^n- \int_{B\cap\partial H} \partial_{\nu_H}\wt f\,\phi \,d\HH^n + \int_H \Delta \wt f \,\phi \,dm,$$
which proves the lemma.
\end{proof}

\vv

The next lemma provides some auxiliary calculations which will be necessary later.

\begin{lemma}\label{lem46}
For each $P_0\in\WW(\Omega)$, the following hold, for a suitable fixed constant $C$:
\begin{itemize}
\item[(a)]\quad
$\ds |\nabla \wt f(x)| \lesssim m_{CB_{b(P_0)}, \sigma}(\nabla_{H,p} f)$ for all $x\in P_0$.\vv

\item[(b)]\quad
$|\Delta \wt f(x)|\lesssim m_{CB_{b(P_0)}, \sigma}(\nabla_{H,p} f)\,\ell(P_0)^{-1}$ for all $x\in P_0$.
\end{itemize}
\end{lemma}

\begin{proof}
For each $x\in P_0$, we have
\begin{align*}
\nabla \wt f(x) &= \sum_{P\in\WW(\Omega)}\nabla \vphi_P(x)\,A_P(x) +  \sum_{P\in\WW(\Omega)}\vphi_P(x)\,\nabla A_P\\
& = 
\sum_{P\in\WW(\Omega)}\nabla \vphi_P(x)\,(A_P(x)-A_{P_0}(x)) + \sum_{P\in\WW(\Omega)} \vphi_P(x)\, \nabla A_P.
\end{align*}
Using \rf{eqprop11} and \rf{eqprop22} and taking into account that the sums above only run over cubes $P$ such that $1.1P\cap P_0\neq\varnothing$ we get
\begin{align*}
|\nabla \wt f(x)| &\lesssim \frac1{\ell(P_0)}\,m_{CB_{b(P_0)}, \sigma}(\nabla_{H,p} f)\,\ell(P_0) + m_{CB_{b(P_0)}, \sigma}(\nabla_{H,p} f) \\ &\lesssim m_{CB_{b(P_0)}, \sigma}(\nabla_{H,p} f),
\end{align*}
which proves (a).

Regarding the second statement of the lemma, for $x\in P_0$ we write
\begin{align*}
\Delta \wt f(x) &= \sum_{P\in\WW(\Omega)}\Delta \vphi_P(x)\,A_P(x) + 2 \sum_{P\in\WW(\Omega)}\nabla \vphi_P(x)\cdot \nabla A_P\\
& = 
\sum_{P\in\WW(\Omega)}\Delta \vphi_P(x)\,(A_P(x)-A_{P_0}(x)) + 2 \sum_{P\in\WW(\Omega)}\nabla \vphi_P(x)\cdot \nabla A_P.
\end{align*}
Using again \rf{eqprop11} and \rf{eqprop22} and that the sums above only run over cubes such that $1.1P\cap P_0\neq\varnothing$, we deduce
\begin{align*}
|\Delta \wt f(x)| & \lesssim \frac1{\ell(P_0)^2}\,m_{CB_{b(P_0)}, \sigma}(\nabla_{H,p} f)\,\ell(P_0) + 
\frac1{\ell(P_0)}\,m_{CB_{b(P_0)}, \sigma}(\nabla_{H,p} f) \\ & \approx \frac1{\ell(P_0)}\,m_{CB_{b(P_0)}, \sigma}(\nabla_{H,p} f).
\end{align*}
\end{proof}

\vv

Given a signed measure $\eta$ in $\Omega$ 
and $x\in\partial\Omega$, we denote the {\it Carleson functional}
\begin{equation}\label{eq:Carlesonfunctional}
\CC(\eta)(x) = \sup_{r>0} \frac{|\eta|(B(x,r)\cap\Omega)}{r^n}.
\end{equation}
For $g\in L^1_{loc}(\Omega)$ we set $\CC(g)(x) = \CC(g\,dm)(x)$, where $dm$ is the Lebesgue measure in $\R^{n+1}$. That is,
$$\CC(g)(x) = \sup_{r>0} \frac1{r^n} \int_{B(x,r)} |g|\,dm.$$

\vv

\begin{lemma} \label{lemdeltap}
{Let $p\in (1,2+\ve_0)$, with $\ve_0>0$ as in Remark \ref{rem:epsilon0}.}
Suppose that $f:\partial\Omega\to\R$ is a Lipschitz function on
$\pom$ and let $v=v(f)$ the almost {harmonic} extension of $f$. Then
\begin{equation}\label{eqCC5}
\|\CC(\Delta v)\|_{L^p(\sigma)} \lesssim \|\nabla_{H,p} f\|_{L^p(\sigma)}.
\end{equation}
\end{lemma}

\begin{proof}
Recall that, by Lemma \ref{lem45*},
$$\Delta v = \eta + \chi_H\,\Delta\wt f\quad \mbox{ in $\Omega$,}$$
where
$$\eta = \sum_{R\in\ttt} (\partial_{\nu_R}\wt f - \partial_{\nu_R}v_R) \,\HH^n|_{\partial\Omega_R\setminus\partial\Omega}.$$
Then we split 
$$\|\CC(\Delta v)\|_{L^p(\sigma)} \leq \|\CC(\eta)\|_{L^p(\sigma)} + \|\CC(\chi_H\Delta\wt f)\|_{L^p(\sigma)} .$$

To estimate $\|\CC(\eta)\|_{L^p(\sigma)}$ we consider the measure
$\mu = \sum_{R\in\ttt}\HH^n|_{\partial\Omega_R}.$
From the fact that $\sigma$-a.e.\ $\xi\in\pom$ belongs to $\bigcup_{R\in\ttt}\partial\Omega_R^{\pm}$ (see Lemma \ref{lemsurfacetot}) it follows that $\partial\Omega\subset\supp\mu$ and $\sigma\leq \mu$ (i.e., $\sigma(A)\leq\mu(A)$ for any Borel set $A$).
Further, from the Carleson packing condition of the family $\ttt$ and the $n$-AD-regularity of each measure $\HH^n|_{\partial\Omega_R}$, it follows easily that $\mu$ is also $n$-AD-regular (with uniform constants). For any ball $B$ centered at $\supp\mu$, we have
\begin{align*}|\eta|(B) & =
\sum_{R\in\ttt} \int_{B\cap \partial\Omega_R\setminus\pom} |\partial_{\nu_R} (\wt f- v_R)|\,d\HH^n \\ &\leq
 \int_B \sum_{R\in\ttt}\big(|\partial_{\nu_R} v_R| + |\nabla \wt f|\big)\,\chi_{\pom_R\setminus\pom}\,d\mu. 
\end{align*}
Denoting by $\NN_R$ the non-tangential maximal operator for the domain $\Omega_R$ and taking
into account that $\mu$-a.e.\ $x\in\R^{n+1}$ belongs at most to one of the sets $\partial\Omega_R$, $R\in\ttt$, we get 
\begin{align*}
|\eta|(B)  & \leq  \int_B \sum_{R\in\ttt}\big(\NN_R(\nabla v_R) +|\nabla\wt f|\big)\,\chi_{\pom_R\setminus\pom}
\,d\mu\\
& \lesssim \int_B \Big(\sum_{R\in\ttt}\big(\NN_R(\nabla v_R)
+|\nabla\wt f|\,\chi_{\pom_R\setminus\pom}
\big)^p\Big)^{1/p}\,d\mu.
\end{align*}
Here we understand that $\NN_R(\nabla  v_R)$ vanishes away from $\pom_R$.
Hence, for $\mu$-a.e.\ $x\in\R^{n+1}$,
\begin{align*}
\CC(\eta)(x) & \lesssim \sup_{r>0} \frac1{\mu(B(x,r))} \int_{B(x,r)} \!\Big(\sum_{R\in\ttt}\!\big(\NN_R(\nabla v_R)
+|\nabla\wt f|\,\chi_{\pom_R\setminus\pom}
\big)^p\Big)^{1/p}d\mu \\ &=\cM_\mu F(x),
\end{align*}
where $\cM_\mu$ is the Hardy-Littlewood maximal operator with respect to $\mu$ and
$$F(x) = \Big(\sum_{R\in\ttt}\big(\NN_R(\nabla v_R)(x)
+|\nabla\wt f(x)|\,\chi_{\pom_R\setminus\pom}
\big)^p\Big)^{1/p}.$$
Therefore, using that $\sigma\leq\mu$ and the $L^p$-solvability of the regularity problem in Lipschitz domains, we obtain
\begin{align}\label{eqalk41}
\|\CC(\eta)\|_{L^p(\sigma)}^p & \lesssim 
\|\cM_\mu F\|_{L^p(\mu)}^p  \lesssim \|F\|_{L^p(\mu)}^p \nonumber\\
& = \sum_{R\in\ttt}\int_{\pom_R}\big(\NN_R(\nabla v_R) +
|\nabla\wt f|\,\chi_{\pom_R\setminus\pom}
\big)^p\,d\mu\nonumber\\
& \lesssim \sum_{R\in\ttt} \left(\int_{\pom_R}|\nabla_{t_R} \wt f|^p\,d\mu + \int_{\pom_R\setminus \pom}{|\nabla \wt f|^p}\,d\mu\right)\nonumber\\
& \lesssim \sum_{R\in\ttt} \int_{\pom_R\setminus \pom}{|\nabla\wt f|^p}\,d\mu + \sum_{R\in\ttt}\int_{\pom_R\cap\pom}|\nabla_{t_R} \wt f|^p\,d\mu,
\end{align}
where $\nabla_{t_R}$ stands for the tangential derivative in $\Omega_R$, so that
$$|\nabla_{t_R} f(x)| = \limsup_{\pom_R\ni y\to x}\frac{|f(y)-f(x)|}{|y-x|}.$$
Further, we have that, for $\sigma$-a.e. $x\in\pom\cap\pom_R$,
$|\nabla_{t_R} \wt f(x)| = |\nabla_{t}  f(x)|,$
where $\nabla_{t}$ stands for the tangential derivative in $\Omega$.
This follows from Rademacher's theorem about the a.e.\ differentiability of Lipschitz functions (see, for example, Theorem 11.4 and Lemma 11.5 from \cite{Maggi}).
Notice now that, by Lemma \ref{lem2.2} and 
since $\nabla_{H,p}f$ is a Haj\l asz upper gradient for $f$, for $\sigma$-a.e.\ $x\in\pom$ we have
\begin{align}\label{eqfac71}
|\nabla_{t}  f(x)| & \approx \limsup_{r\to0} \avint_{B(x,r)} \frac{|f(y) - f(x)|}{|y-x|}\,d\sigma(y)\\
& \leq \limsup_{r\to0} \avint_{B(x,r)}
\big(\nabla_{H,p}f(y) + \nabla_{H,p}f(x)\big)\,d\sigma(y)\notag\\
& \leq \cM_\sigma(\nabla_{H,p}f)(x) + \nabla_{H,p}f(x).\notag
\end{align}
So using also that the sets $\pom\cap \pom_R$ are
disjoint (with the possible exception of a set of zero surface measure),
we get
\begin{align}\label{eqal946}
 \sum_{R\in\ttt} \int_{\pom\cap\pom_R} &|\nabla_{t_R} \wt f|^p\,d\sigma
 =  \sum_{R\in\ttt} \int_{\pom\cap\pom_R} |\nabla_{t} f|^p\,d\sigma\\
 &\lesssim\int_\pom |\cM_\sigma(\nabla_{H,p}f)|^p\,d\sigma +  \int_\pom |\nabla_{H,p}f|^p\,d\sigma\lesssim \int_\pom |\nabla_{H,p}f|^p\,d\sigma.
 \notag
\end{align}

Regarding the first term of \rf{eqalk41}, for each $R$ we have, by Lemma \ref{lem46}(a), {
$$\int_{\pom_R\setminus \pom}|\nabla \wt f|^p\,d\mu =  \int_{\pom_R\setminus \pom}|\nabla \wt f|^p\,d \mathcal{H}^n|_{\pom_R}
\lesssim \sum_{\substack{P\in\WW(\Omega):\\P\cap \pom_R\neq\varnothing}} m_{CB_{b(P)}, \sigma}(\nabla_{H,p} f)\,\ell(P)^n.
$$}
So taking into account that the cubes $b(P)$ with $P\in\WW(\Omega)$ such that $P\cap \pom_R\neq\varnothing$ are contained in $\HH$ (where $\HH$ is the Carleson family of cubes defined in Lemma \ref{lem:HCarleson})
and using also \rf{eqal946},
we obtain 
\begin{align}\label{eqal947}
\|\CC(\eta)\|_{L^p(\sigma)}^p
 & \lesssim 
\sum_{Q\in\HH} m_{CB_Q, \sigma}(\nabla_{H,p} f)^p\,\ell(Q)^n + \int_\pom |\nabla_{H,p}f|^p\,d\sigma.
\end{align}
{
Since $\HH$ is  a Carleson family, by the lower AD-regularity of $\sigma$, Carleson's embedding theorem, and the $L^p(\sigma)$ boundedness of the Hardy-Littlewood maximal operator $\cM_{\sigma}$,
it follows that
\begin{align}\label{eqhcar23}
\sum_{Q\in\HH} m_{CB_Q, \sigma}(\nabla_{H,p} f)^p\,\ell(Q)^n &\lesssim \int_\pom \sup_{Q \ni x} m_{CB_Q, \sigma}(\nabla_{H,p} f)^p \,d\sigma(x)\\
& \lesssim \|\cM_{\sigma}(\nabla_{H,p} f)\|_{L^p(\sigma)}^p
\lesssim \|\nabla_{H,p} f\|_{L^p(\sigma)}^p.\notag
\end{align}
}
Together with \rf{eqal947}, this shows that
$$\|\CC(\eta)\|_{L^p(\sigma)}\lesssim\|\nabla_{H,p} f\|_{L^p(\sigma)}.$$


{It remains to estimate $ \|\CC(\chi_H\Delta\wt f)\|_{L^p(\sigma)}$. To this end,} notice that by Lemma \ref{lem46}(b)
for any ball $B$ centered at $\pom$,  we have
$$\int_{B\cap H} |\Delta \wt f|\,dm \leq 
\sum_{Q\in\HH} \int_{w(Q)\cap B} |\Delta \wt f|\,dm\lesssim
\sum_{Q\in\HH:Q\subset CB} 
m_{CB_Q, \sigma}(\nabla_{H,p} f)\,\sigma(Q). 
$$
From the fact that $\HH$ is a Carleson family and Carleson's embedding theorem, it follows that 
$$\int_{B\cap H} |\Delta \wt f|\,dm \lesssim
\int \sup_{Q:x\in Q\subset CB} 
m_{CB_Q, \sigma}(\nabla_{H,p} f)  \,d\sigma(x) \lesssim \int_{CB} \cM_{\sigma}(\nabla_{H,p} f)\,d\sigma.
$$
Therefore, $\CC(\chi_H\Delta\wt f)(x)\leq\cM_\sigma(\cM_{\sigma}(\nabla_{H,p} f))(x)$, and so
\begin{align*}
 \|\CC(\chi_H\Delta\wt f)\|_{L^p(\sigma)} & \lesssim 
\|\cM_\sigma(\cM_{\sigma}(\nabla_{H,p} f))\|_{L^p(\sigma)} \lesssim 
\|\cM_{\sigma}(\nabla_{H,p} f)\|_{L^p(\sigma)} \\ &\lesssim\|\nabla_{H,p} f\|_{L^p(\sigma)}.
\end{align*}
{Combining this and the previously obtained estimate for $\|\CC(\eta)\|_{L^p(\sigma)}$,} we conclude the proof of the lemma.
\end{proof}

\begin{remark}
Extensions of $L^p(\R^n)$ functions to the upper half-space $\R^{n+1}_+$ so that $\|\CC(\nabla v)\|_{L^p(\sigma)} \lesssim \|f\|_{L^p(\sigma)}$ and $\| \NN( v)\|_{L^p(\sigma)} \lesssim \|f\|_{L^p(\sigma)}$ were  considered by Hyt\"onen and Ros\'en in \cite{HR2} using the standard dyadic extension and then a clever combination of an $\ve$-approximability type argument along with an iteration process. This method, or rather its  extension in terms of best affine approximations from the Dorronsoro-type  coefficients like our $\wt f$, does not seem to apply to boundary functions in Sobolev spaces even in $\R^{n+1}_+$, let alone in rough domains. 
\end{remark}

\vv


\section{The one-sided Rellich inequality for the regularity problem}

We will need the following important technical result. We will abuse notation and write $|\Delta v| \,dm$ instead of $d |\Delta v|$, which would be the accurate notation since we have only shown that $|\Delta v|$ is a non-negative Radon measure.

\begin{lemma}\label{lematec}
{Let $\Omega\subset\R^{n+1}$ satisfy the assumptions in Theorem \ref{teomain}. Let also  $w:\overline\Omega\to\R$ be a harmonic function  in $\Omega$, which is continuous in $\overline \Omega$ and Lipschitz on $\pom$. If $f\in {\rm Lip}(\pom)$ and $v$ is the almost harmonic extension of $f$ to $\Omega$, then he have}
\begin{equation}\label{eqclau9}
\left|\int_{\Omega} \nabla v\cdot \nabla w\,dm\right| \lesssim 
\|\nabla_{H,p} f\|_{L^p(\sigma)}\,\|w\|_{L^{p'}(\sigma)}+ \int_\Omega |w|  |\Delta v|\,dm.
\end{equation}
\end{lemma}

\begin{proof}
We consider the family $I_k$ of dyadic cubes $T \subset\R^{n+1}$ with side length $2^{-k}$ such that
$3T\cap\pom\neq \varnothing$ and we let 
\begin{equation*}
\Omega_k= \Omega\setminus \bigcup_{T\in I_k} \overline{T}.
\end{equation*}
{By comparing $\pom_k$ to $\pom$, it is easy to check that $\pom_k$ is $n$-AD-regular uniformly on $k$} and that $\dist(x,\pom)\approx 2^{-k}$ for all
$x\in\pom_k$. 

By the assumptions above, Lemma \ref{lem190}, and Remark \ref{remnab}, $v,w\in W^{1,2}(\Omega)$ and thus
$$\int_{\Omega} |\nabla v\cdot \nabla w|\,dm <\infty.$$
Thus to prove \eqref{eqclau9} it suffices to show that
$$\left|\int_{\Omega_k} \nabla v\cdot \nabla w\,dm\right| \lesssim 
\|\nabla_{H,p} f\|_{L^p(\sigma)}\,\|w\|_{L^{p'}(\sigma)}+ \int_\Omega |w| \,|\Delta v|\,dm$$
uniformly on $k$. To this end, for $\ve>0$ let
$$v_\ve = v *\phi_\ve,$$
where $\{\phi_\ve\}_{\ve>0}$ is an approximation of the identity, where $\phi_\ve$ is smooth, radial, and supported on $B(0,\ve)$.  {So $v_\ve$ and $w$ are smooth functions in $ \Omega$ and thus, since $\Omega_k$ is of finite perimeter (as $\pom_k$ is AD-regular) and $\overline{\Omega}_k \subset \Omega$, we may apply} Green's formula
to get
\begin{align*}
\int_{\Omega_k} \nabla v_\ve\cdot \nabla w\,dm =\int_{\pom_k} \partial_{\nu_k} v_\ve\,w\,d\HH^n - \int_{\Omega_k} \Delta v_\ve\, w\,dm,
\end{align*}
where $\nu_k$ stands for the outer unit normal on $\pom_k$.
Since $\nabla v_\ve= \phi_\ve *\nabla v$ and $v\in W^{1,2}(\Omega)$, it is clear that
$$\lim_{\ve\to 0}\int_{\Omega_k} \nabla v_\ve\cdot \nabla w\,dm = 
\int_{\Omega_k} \nabla v\cdot \nabla w\,dm.$$
{Let now $\eta_k$ be a positive smooth function with compact support  so that $\chi_{\overline{\Omega}_k} \leq \eta_k \leq \chi_{\Omega_{k+1}}$. In light of $\Delta v_\ve = \phi_\ve * \Delta v$ and the smoothness of $w$ in $\Omega$, it follows that
$$\limsup_{\ve\to 0} \left|\int_{\Omega_k} \Delta v_\ve\, w\,dm\right| \leq \limsup_{\ve\to 0} \int  (\phi_\ve * |\Delta v|)\, |w \eta_k| \,dm \leq  \int_{\Omega} |w|\,|\Delta v|\, dm,$$
where the last inequality follows from the fact that $|\Delta v|$ is a non-negative Radon measure, $|w \eta_k| $ is a continuous function with compact support in $\Omega_{k+1}$, and $\eta_k \leq 1$.} So we have
$$\left|\int_{\Omega_k} \nabla v\cdot \nabla w\,dm\right| 
\leq \limsup_{\ve\to 0} \int_{\pom_k} \big|\partial_{\nu_k} v_\ve\,w\big|\,d\HH^n 
+\int_{\Omega} |w|\,|\Delta v|\,dm.$$
Therefore, by H\"older's inequality, 
to conclude the proof of the lemma it suffices to show that
\begin{equation}\label{eqfj5}
\limsup_{\ve\to 0} \int_{\pom_k} \big|\partial_{\nu_k} v_\ve\big|^p\,d\HH^n
\lesssim \|\nabla_{H,p} f\|_{L^p(\sigma)}^p.
\end{equation}

{Let us denote $\Gamma_k=\pom_k$  and  set
$$U_\ve = \bigcup_{R\in \ttt}\cU_\ve(\pom_R)$$
(recall that $\cU_\ve(A)$ stands for the $\ve$-neighborhood of $A$). To prove \eqref{eqfj5} we assume $\ve\ll 2^{-k}$.} Then we split
\begin{align*}
\int_{\Gamma_k}|\nabla v_\ve|^p\,d\HH^n  & = \sum_{R\in \ttt} \int_{\Gamma_k\cap \Omega_R\setminus U_\ve} |\nabla v_\ve|^p\,d\HH^n \\
&\quad + \int_{\Gamma_k\cap H\setminus U_\ve} |\nabla v_\ve|^p\,d\HH^n
+ \int_{\Gamma_k\cap U_\ve} |\nabla v_\ve|^p\,d\HH^n  =: I_{1} + I_2 + I_3.
\end{align*}
To estimate $I_{1}$, observe that, by the mean value property, since $v$ (and thus $\nabla v$) is harmonic in each $\Omega_R$ and $\phi_\ve$ is radial and supported in $B(0,\ve)$, then $\nabla v_\ve =
\phi_\ve *\nabla v
= \nabla v$ in $\Omega_R\setminus U_\ve$ {(to check this, just notice that for all $x\in\Omega_R\setminus U_\ve$, $\phi_\ve *\nabla v(x)$ can be written as a convex combination of averages $\avint_{B(x,r)} \nabla v\,dm$, with $0<r< \ve$).}  Thus,
$$I_{1} \leq \sum_{R\in\ttt} \int_{\Gamma_k\cap \Omega_R} |\nabla v|^p\,d\HH^n =:
\sum_{R\in\ttt} I_{1,R}.$$
To deal with $I_{1,R}$, we split $\Omega_R$ into a family $\WW_R$ of Whitney cubes, as in Subsection \ref{subsecwhitney} replacing $\Omega$ by $\Omega_R$. We also consider a dyadic lattice $\DD_R$ of ``cubes'' in $\pom_R$, so that the 
largest cube is $\pom_R$. To each $P\in \WW_R$ we assign a boundary cube $b_R(P)\in\DD_R$ such that
$\ell(P)=\ell(b_R(P))$ and $\dist(P,b_R(P))\approx\ell(P)$. We also consider a non-tangential
maximal operator $\NN_R$ for $\Omega_R$ associated with non-tangential cones with big enough aperture. Then we have
\begin{align*}
I_{1,R} &= \sum_{P\in\WW_R} \int_{\Gamma_k\cap P} |\nabla v|^p\,d\HH^n \\ &\lesssim 
\sum_{\substack{P\in\WW_R\\ P\cap\Gamma_k\neq\varnothing}} \inf_{x\in b_R(P)} \NN_R(\nabla v)(x)^p\,\ell(P)^n\approx \sum_{Q\in J_R} \inf_{x\in Q} \NN_R(\nabla v)(x)^p\,\ell(Q)^n,
\end{align*}
where $J_R$ is the family of cubes $Q\in\DD_R$ such that there exists some $P\in\WW_R$ with 
$Q=b_R(P)$, $P\cap\Gamma_k\neq\varnothing$. We claim that the family $J_R$ satisfies a Carleson
packing condition. Indeed, by the AD-regularity of $\Gamma_k$ (with uniform constants), for any $S\in \DD_R$, we have
$$\sum_{Q\in J_R:Q\subset S}\ell(Q)^n\lesssim \sum_{\substack{P\in\WW_R\\ P\cap\Gamma_k\neq\varnothing\\P\subset B(x_S,C_7\ell(S))}
} \HH^n(2P\cap\Gamma_k),
$$
for a suitable big constant $C_7$, where $x_s$ is the center of $S$. Then, by the bounded overlaps of the cubes $2P$ and the AD-regularity of $\Gamma_k$, we derive
$$\sum_{Q\in J_R:Q\subset S} \ell(Q)^n \lesssim \HH^n(B(x_S,C_8\ell(S))\cap\Gamma_k) \lesssim \ell(S)^n.$$
Then, by the AD-regularity of $\pom_R$, Carleson's embedding theorem, the solvability of $(\wt R_p)$ in bounded starlike Lipschitz domains (with uniform constants), Lemma \ref{lem46}(a), and \eqref{eqfac71}, we derive
\begin{align*}
I_{1,R} & \lesssim \int_{\pom_R} \NN_R(\nabla v)^p\,d\HH^n \lesssim \int_{\pom_R} |\nabla_{t_R} v|^p\,d\HH^n\\
& \lesssim \sum_{P\in\WW(\Omega):P\cap \pom_R\neq\varnothing}
m_{CB_{b(P)}, \sigma}(\nabla_{H,p} f)^p\,\ell(P)^n + \int_{\pom\cap\pom_R} |\nabla_{H,p}f|^p\,d\sigma.
\end{align*}
To conclude the estimate of $I_1$ we use that the cubes
$b(P)$ with $P\in\WW(\Omega)$ intersecting $\pom_R$ belong to $\HH$, which is a Carleson family of cubes from $\DD_\sigma$ by Lemma~\ref{lem:HCarleson}:
\begin{align}\label{eqI1*}
I_1& =\sum_{R\in\ttt} I_{1,R}\\ &\lesssim \sum_{R\in\ttt} \int_{\pom_R}\!\! |\nabla_{t_R} v|^p\,d\HH^n\lesssim \sum_{Q\in\HH}\!
m_{CB_Q, \sigma}(\nabla_{H,p} f)^p\,\ell(Q)^n + \int_{\pom} \!|\nabla_{H,p}f|^p\,d\sigma \nonumber\\ &\lesssim \|\nabla_{H,p} f\|_{L^p(\sigma)}^p,\nonumber
\end{align}
where we argued as in \eqref{eqhcar23} for the last estimate.

Next we turn our attention to the integral $I_2$. To this end, recall that $H$ is covered by a
family of cubes $w(Q)$, with $Q\in\HH$.
Further, for $x\in w(Q)\setminus U_\ve$, by Lemma \ref{lem46}(a) we have
$|\nabla v_\ve(x)|\lesssim m_{CB_Q}(\nabla_{H,p} f),$ and thus, by another application of  Carleson's theorem, we obtain
$$I_2\lesssim 
\sum_{Q\in\HH}
m_{CB_Q, \sigma}(\nabla_{H,p} f)^p\,\ell(Q)^n  \lesssim \|\nabla_{H,p} f\|_{L^p(\sigma)}^p.$$

To deal with $I_3$ we split
$$I_3 \lesssim \int_{\Gamma_k\cap U_\ve} |(\nabla v\,\chi_H)*\phi_\ve|^p\,d\HH^n +
\int_{\Gamma_k\cap U_\ve} |(\nabla v\,\chi_{\Omega\setminus H})*\phi_\ve|^p\,d\HH^n =: I_{3,1} + I_{3,2}.$$
We estimate $I_{3,1}$ very similarly to $I_2$. Recall that we have assumed that $\ve\ll 2^{-k}$, so that
$\ve$ is much smaller than  the side length of any Whitney cube $P$ of $\Omega$ such that $2P\cap \Gamma_k\neq\varnothing$.
So denoting {by $w^*(Q)$ the $(5\ell(Q))$-neighborhood of $w(Q)$,}  
we have
\begin{align*}
I_{3,1} & \leq \!\sum_{P\in\WW_R:P\cap\Gamma_k\neq\varnothing}
\int_{\Gamma_k\cap 2P} |(\nabla v\,\chi_H)*\phi_\ve|^p\,d\HH^n \\
& = \!\sum_{P\in\WW_R:P\cap\Gamma_k\neq\varnothing}
\int_{\Gamma_k\cap 2P} |(\nabla v\,\chi_{H\cap 2.1P})*\phi_\ve|^p\,d\HH^n\\
& \leq
\sum_{Q\in \HH}
\int_{ \Gamma_k\cap w(Q)} |(\nabla v\,\chi_{w^*(Q)\cap H})*\phi_\ve|^p\,d\HH^n \\ &\lesssim 
\sum_{Q\in \HH} m_{CB_Q}(\nabla_{H,p} f)^p\,\ell(Q)^n \lesssim \|\nabla_{H,p} f\|_{L^p(\sigma)}^p,
\end{align*}
where we took $Q=b(P)$ and we applied again Lemma \ref{lem46}(a) and \eqref{eqhcar23} to get the last inequality.

Finally we consider the integral $I_{3,2}$. From the smallness of $\ve$, we derive
\begin{align*}
I_{3,2} & = \sum_{R\in \ttt} \int_{\Gamma_k\cap\cU_\ve(\pom_R)}
 |(\nabla v\,\chi_{\Omega\setminus H})*\phi_\ve|^p\,d\HH^n \\
 & = \sum_{R\in \ttt} \int_{\Gamma_k\cap\cU_\ve(\pom_R)}
 |(\nabla v\,\chi_{\Omega_R})*\phi_\ve|^p\,d\HH^n =: \sum_{R\in \ttt} I_{3,2,R}. 
\end{align*}
To estimate $I_{3,2,R}$ we consider a family $K_R$ of balls $B$ centered at $\Gamma_k\cap\cU_\ve(\pom_R)$ with radius $2\ve$ having bounded overlaps that cover $\Gamma_k\cap\cU_\ve(\pom_R)$ 
(this family can be found using the Besicovitch covering theorem), so that
\begin{align}\label{eqal810}
I_{3,2,R} &\leq \sum_{B\in K_R} \int_{\Gamma_k\cap B}
 |(\nabla v\,\chi_{\Omega_R})*\phi_\ve|^p\,d\HH^n \\ & = \sum_{B\in K_R} \int_{\Gamma_k\cap B}
 |(\nabla v\,\chi_{\Omega_R\cap 2B})*\phi_\ve|^p\,d\HH^n\nonumber\\
 & 
 \lesssim \sum_{B\in K_R} \left(\frac1{r(B)^{n+1}}\int_{\Omega_R\cap 2B}|\nabla v|\,dm\right)^p\,r(B)^n
 \nonumber\\
 &\lesssim \sum_{B\in K_R} \frac1{r(B)}\int_{\Omega_R\cap 2B}|\nabla v|^p\,dm  \lesssim \frac1\ve\int_{\cU_{5\ve}(\pom_R )\cap\Omega_R}|\nabla v|^p\,dm,
\nonumber
\end{align}
where we applied H\"older's inequality in the penultimate inequality and then used that the 
balls $2B$, with $B\in K_R$, are contained in $\cU_{5\ve}(\pom_R )$ and also have bounded overlaps (because they all have the same radius).
To estimate the last integral we consider the Whitney decomposition of $\Omega_R$ into the family of cubes $\WW_R$ 
described above. Let $\WW_{R,\ve}$ be the subfamily of those cubes $P\in\WW_R$ which intersect
$\cU_{5\ve}(\pom_R )$. Since the cubes $P\in\WW_{R,\ve}$ satisfy 
$\ell(P)\approx\ell(b_R(P))\lesssim\ve$, we get
\begin{align*}
\int_{\cU_{5\ve}(\pom_R )\cap\Omega_R} |\nabla v|^p\,dm & \leq \sum_{P\in\WW_{R,\ve}}
\int_P|\nabla v|^p\,dm \\ &\lesssim \sum_{Q\in\DD_R:\ell(Q)\leq C\ve} \inf_{x\in Q} \NN_R(\nabla v)(x)^p\,\ell(Q)^{n+1}.
\end{align*}
 Consequently, by the solvability of $(\wt R_p)$ in bounded starlike Lipschitz domains,
\begin{align*}
\int_{\cU_{5\ve}(\pom_R )\cap\Omega_R} |\nabla v|^p\,dm &
\lesssim  \sum_{Q\in\DD_R:\ell(Q)\leq C\ve} \ell(Q)\int_Q |\NN_R(\nabla v)|^p\,d\HH^n\\
& \lesssim \ve\,\int_{\pom_R} |\NN_R(\nabla v)|^p\,d\HH^n \lesssim \ve\,\|\nabla_{t_R} v\|_{L^p(\HH^n|_{\pom_R})}^p.
\end{align*}
Plugging this estimate into \eqref{eqal810} and using Carleson's embedding theorem, as in \eqref{eqI1*},
 we derive
$$I_{3,2}\lesssim \frac1\ve\!\sum_{R\in \ttt}\int_{\cU_{5\ve}(\pom_R )\cap\Omega_R}\!\!|\nabla v|^p\,dm\lesssim \!\!
\sum_{R\in \ttt}\!\!\|\nabla_{t_R} v\|_{L^p(\HH^n|_{\pom_R})}^p\!\lesssim \!\|\nabla_{H,p} f \|_{L^p(\sigma)}^p.$$
Combining the estimates obtained for $I_1$, $I_2$, and $I_3$,  we obtain \eqref{eqfj5}  and the proof of the
lemma is concluded.
\end{proof}

\vv

Given $u\in W^{1,2}(\Omega)$ which is harmonic in $\Omega$, we say that $g=\partial_\nu u$ {\it in the weak sense}  if
\begin{equation}\label{eqnormal}
\int_\Omega \nabla u\cdot \nabla \vphi\,dm = \int_\pom g\,\vphi\,d\sigma
\end{equation}
for all $\vphi\in {\rm Lip}(\overline \Omega).$
We are now ready  to prove the desired one-side Rellich type estimate required for the solvability of $(R_p)$.
\vv

\begin{lemma}\label{keylemma}
Let $\Omega\subset\R^{n+1}$ satisfy the assumptions in Theorem \ref{teomain}.  Given
 $f\in {\rm Lip}(\pom)$, denote by $u$ the solution of the Dirichlet problem in $\Omega$ with boundary data $f$.
 Then $\partial_\nu u$ exists in the weak sense,
 it belongs to $L^p(\sigma)$, and it satisfies
\begin{equation}\label{eqteo41}
\|\partial_\nu u\|_{L^p(\sigma)} \lesssim \|\nabla_{H,p} f\|_{L^p(\sigma)}.
\end{equation}
\end{lemma}

\begin{proof}
We have to show that there exists some function $g\in L^p(\sigma)$ satisfying \eqref{eqnormal} for all
$\vphi\in {\rm Lip}(\overline \Omega)$, with 
$\|g\|_{L^p(\sigma)} \lesssim \|\nabla_{H,p} f\|_{L^p(\sigma)}$. 
Let $w$ be the solution of the Dirichlet problem in $\Omega$ with boundary data $\vphi|_\pom$.
By Remark \ref{remnab}, both $\vphi$ and $w$ belong to $W^{1,2}(\Omega)\cap C(\overline\Omega)$ and they agree on $\pom$. Thus, $\vphi-w\in W_0^{1,2}(\Omega)$, and then by the harmonicity of $u$ (which also
belongs to $W^{1,2}(\Omega)$), we deduce that
$$
\int_\Omega \nabla u\cdot \nabla(\vphi - w)\,dm=0.
$$

We now consider  the almost harmonic extension $v$ of $f$ to $\Omega$. Recall that $v\in W^{1,2}(\Omega)\cap C(\overline\Omega)$ and $v|_\pom = f$ by Lemma \ref{lem190}. Then we have
$$\int_\Omega \nabla u\cdot \nabla\vphi \,dm = \int_\Omega \nabla u\cdot \nabla w\,dm =
\int_\Omega \nabla (u-v)\cdot \nabla w\,dm + \int_\Omega \nabla v\cdot \nabla w\,dm.$$
Since $u-v\in W_0^{1,2}(\Omega)$ and $w$ is harmonic in $\Omega$, we have
$$\int_\Omega \nabla (u-v)\cdot \nabla w\,dm 
=0.$$
Then, by Lemma \ref{lematec} we deduce
\begin{align}\label{eqqp1}
\left|\int_\Omega \nabla u\cdot \nabla\vphi \,dm\right| & \lesssim 
\|\nabla_{H,p} f\|_{L^p(\sigma)}\,\|w\|_{L^{p'}(\sigma)}+ \int_\Omega  |w| \,|\Delta v|\, d m.
\end{align}

Next we will estimate the last integral on the right hand side above. First we denote
$$F(x) := \CC(\Delta v)(x) = \sup_{r>0} \frac1{r^n}\int_{B(x,r)} |\Delta v| \,dm.$$
By Hölder's inequality we have
\begin{align}\label{eqal8r2}
&\int_\Omega |w| \,|\Delta v| \,dm  \leq \sum_{Q\in \DD_\sigma} \|w\|_{\infty,w(Q)}
\int_{w(Q)}|\Delta v|\,dm\\
& \leq
\bigg(\sum_{Q\in \DD_\sigma} \frac{\|w\|_{\infty,w(Q)}^{p'}}{\inf_{x\in Q} F(x)}
\int_{w(Q)}|\Delta v|\,dm\bigg)^{1/p'}
\bigg(\sum_{Q\in \DD_\sigma} \inf_{x\in Q} F(x)^{p-1}
\int_{w(Q)}|\Delta v|\,dm\bigg)^{1/p}.\nonumber
\end{align}

Observe now that the coefficients
$$a_Q := \frac1{\inf_{x\in Q} F(x)}
\int_{w(Q)}|\Delta v|\,dm$$
satisfy the following packing condition, for any given $S\in\DD_\sigma$:
$$\sum_{Q\in\DD_\sigma:Q\subset S} a_Q \leq \frac1{\inf_{x\in S} F(x)} \sum_{Q\in\DD_\sigma:Q\subset S}\int_{w(Q)}|\Delta v|\,dm \leq \frac1{\inf_{x\in S} F(x)}\,\int_{CB_S}|\Delta v|\,dm. 
$$
Recall that $B_S$ is a ball centered at $x_S\in\pom$ that contains $S$ with radius $\ell(S)$. Above we chose $C$ so that all the Whitney regions $w(Q)$ with $Q\subset S$ are contained 
$CB_S$.
By the definition of the Carleson functional $\CC$, it easily follows  that
$$\int_{C B_S}|\Delta v|\,dm \lesssim \sigma(S)\inf_{x\in S} F(x),$$
and thus
$$\sum_{Q\in\DD_\sigma:Q\subset S} a_Q \lesssim \sigma(S).$$
Then, by Carleson's embedding theorem, we get
$$\sum_{Q\in \DD_\sigma} \frac{\|w\|_{\infty,w(Q)}^{p'}}{\inf_{x\in Q} F(x)}
\int_{w(Q)}\!|\Delta v|\,dm \lesssim \int_\pom \sup_{Q\ni x} \|w\|_{\infty,w(Q)}^{p'}\,d\sigma(x)
\lesssim \!\int_\pom |\NN w|^{p'}\,d\sigma,$$
where $\NN$ denotes the usual non-tangential maximal operator for $\Omega$ with a big enough aperture.
Since the Dirichlet problem is solvable in $L^{p'}(\sigma)$, 
we get
$$\sum_{Q\in \DD_\sigma} \frac{\|w\|_{\infty,w(Q)}^{p'}}{\inf_{x\in Q} F(x)}
\int_{w(Q)}|\Delta v|\,dm \lesssim \|w\|_{L^{p'}(\sigma)}^{p'}.$$

To deal with the last term on the right hand side of \eqref{eqal8r2} we apply again Carleson's theorem:
\begin{align*}
\sum_{Q\in \DD_\sigma}& \inf_{x\in Q} F(x)^{p-1}
\int_{w(Q)}|\Delta v|\,dm \\ &\leq \sum_{Q\in \DD_\sigma} \left(\avint_Q |F|\,d\sigma\right)^p\,
\frac1{\inf_{x\in Q} F(x)}
\int_{w(Q)}|\Delta v|\,dm \lesssim \int_\pom |\cM_\sigma (F)|^p\,d\sigma,
\end{align*}
where $\cM_\sigma$ is the usual Hardy-Littlewood operator (with respect to $\sigma$).
Therefore, by \eqref{eqCC5},
$$\sum_{Q\in \DD_\sigma} \inf_{x\in Q} F(x)^{p-1}
\int_{w(Q)}|\Delta v|\,dm \lesssim \|F\|_{L^{p}(\sigma)}^{p} = \|\CC(\Delta v)\|_{L^{p}(\sigma)}^{p}
\lesssim \|\nabla_{H,p} f\|_{L^p(\sigma)}^p.
$$

Plugging the above estimates into \eqref{eqal8r2}, we obtain
$$\int_\Omega |w||\Delta v|\,dm \lesssim \|\nabla_{H,p} f\|_{L^p(\sigma)}\,\|w\|_{L^{p'}(\sigma)}.$$
Together with \eqref{eqqp1},  this yields
\begin{equation}\label{eqfkg49}
\left|\int_\Omega \nabla u\cdot \nabla\vphi \,dm\right| \lesssim \|\nabla_{H,p} f\|_{L^p(\sigma)}\,\|w\|_{L^{p'}(\sigma)}= \|\nabla_{H,p} f\|_{L^p(\sigma)}\,\|\vphi\|_{L^{p'}(\sigma)}.
\end{equation}
Consider now the map $T_u: {\textup{Lip}}(\pom)\to\R$
defined by
$$T_u(\vphi) = \int_\Omega \nabla u\cdot \nabla\vphi \,dm,\quad\quad \vphi\in {\textup{Lip}}(\pom),$$
where we assume that $\vphi$ has been extended in a Lipschitz way to the whole $\overline \Omega$ in
order to be able to compute the integral above. Notice that, since $u$ is harmonic in $\Omega$, the definition of $T_u$ does
not depend on the precise Lipschitz extension of $\vphi$, by the same argument as the one in the first paragraph of this proof. So $T_u$ is well defined and it is linear. Further, by \rf{eqfkg49}
 it
extends to a bounded functional $T_u:L^{p'}(\sigma)\to\R$, and thus, by the Riesz representation theorem, there
exists some function $g=\partial_\nu u\in L^{p}(\sigma)$ satisfying the properties claimed in the lemma. 
\end{proof}

\vv


\section{The end of the proof of Theorem \ref{teomain}} \label{secteomain}

Let us recall the definitions of the  single layer potential
\begin{equation}\label{def-single}
\cS f(x)= \int_{\partial \Omega} \EE(x-y) f(y) \,d\sigma(y), \,\,\quad x \in\Omega
\end{equation}
and  the  double layer potential
\begin{equation}\label{def-double}
\DD f(x)= \int_{\partial^*\om} \nu(y) \cdot \nabla_y \EE(x-y) f(y) \,d\sigma(y), \,\,\quad x \in\Omega,
\end{equation}
where $\nu$ stands for the measure theoretic outer unit normal of $\Omega$. 

\begin{lemma}\label{lem6.1}
Let $\Omega\subset\R^{n+1}$ satisfy the assumptions in Theorem \ref{teomain}.  Given
 $f\in {\rm Lip}(\pom)$, denote by $u$ the solution of the Dirichlet problem in $\Omega$ with boundary data $f$.
 Then we have
$$u(x)=\DD(u|_\pom)(x) - \cS(\partial_\nu u|_{\partial \om})(x) \quad \mbox{ for all $x\in\Omega$,}$$
where $\DD$ and $\cS$ denote the double and single layer potentials for $\Omega$, respectively.
\end{lemma}

Recall that in Lemma \ref{keylemma} we showed that $\partial_\nu u$ exists in the weak sense and
$\partial_\nu u\in L^p(\sigma)$.

\begin{proof}
Let $\vphi:\R^{n+1}\to\R$ be a smooth radial function such that $\chi_{B(0,2)^c} \leq \vphi\leq \chi_{B(0,1)^c}$, and let $\vphi_r(y) = \vphi(r^{-1}y)$.
Given $x\in\Omega$, $y\in\pom$, we denote
\begin{equation*}
f_{x,r}(y) = \vphi_r(y-x)\,\EE(y-x).
\end{equation*}
Then, by the definition  \rf{eqnormal} of $\partial_\nu u$ and Green's formula 
\rf{eq:Green-finiteperim}, we have
\begin{equation}\label{eqp99}
\int_\pom f_{x,r}\,\partial_\nu u\,d\sigma = \int_\Omega \nabla f_{x,r}\,\nabla u\,dm = \int_{\partial^* \Omega}
\partial_\nu f_{x,r}\, u\,d\sigma - \int_\Omega \Delta f_{x,r}\,u\,dm.
\end{equation}
Since  $u\, \nabla f_{x,r} \in C(\overline \Omega) \cap \dot W^{1,1}(\Omega)$ (as $u \in C(\overline \Omega) \cap \dot W^{1,2}(\Omega)$), and $\Omega$ is a bounded open set such that $\HH^n(\pom)<\infty$, the last equality follows from \eqref{eq:Green-Federer}.
For $r$ small enough {compared to $\dist(x,\pom)$}, the left hand side of \rf{eqp99} coincides with $\cS(\partial_\nu u)(x)$, while
the right hand side equals
\begin{multline*}
\int_{\partial^* \Omega} 
\partial_\nu  \vphi_r(y-x)\,\EE(y-x)\, u(y)\,d\sigma(y) + \int_{\partial^* \Omega}
 \vphi_r(y-x)\, \partial_\nu\EE(y-x)\,u(y)\,d\sigma(y)\\ - \int_\Omega \!\Delta f_{x,r}\,u\,dm
  = 0 + \DD(u|_\pom)(x) - \int_\Omega \Delta f_{x,r}\,u\,dm.
\end{multline*}
Observe now that $\supp (\Delta f_{x,r}) \subset \overline{B}(x,2r)$. So if we denote 
$$\wt u(y) = u(y) \big(1-\vphi_{d}(y-x)\big),$$
assuming $2r<d<2d<\dist(x,\pom)$, we have
$$\int_\Omega \Delta f_{x,r}\,u\,dm = \int \Delta f_{x,r}\,\wt u\,dm = 
\int f_{x,r}\,\Delta(\wt u)\,dm.
$$
Then, letting $r\to0$, we deduce that the integrals above tend to $\EE * \Delta\wt u(x) =\wt u(x)=u(x)$.
Therefore,
$$\cS(\partial_\nu u|_{\partial \om})(x) = \DD(u|_\pom) (x) - u(x),$$
as wished.
\end{proof}
\vv

Now we are ready to prove the first part of Theorem \ref{teomain}. Recall that the assumptions of this theorem 
imply the uniform rectifiability of $\pom$.

\begin{lemma}\label{lem6.2*}
Let $\Omega\subset\R^{n+1}$ be a bounded corkscrew domain with $n$-AD-regular boundary.  If there exists $p \in (1, 2+\ve_0)$ such that $(D_{p'})$  is solvable and $\ve_0$ is  defined in Remark \ref{rem:epsilon0}, then $(R_p)$ is solvable.
\end{lemma}

\begin{proof}
Given
 $f\in {\rm Lip}(\pom)$, denote by $u$ the solution of the Dirichlet problem in $\Omega$ with boundary data $f$.
 By the previous lemma, for any $x\in\Omega$,
$$\nabla u(x)=\nabla\DD(u|_\pom)(x) - \nabla\cS(\partial_\nu u|_{\pom})(x).$$
and so by \eqref{eq:ntboundssingle}, \eqref{eq:ntboundsgraddouble},  Lemma \ref{keylemma}, Lemma \ref{lemreal} below, and \rf{eqfac71}, we have
\begin{align*}
\|\NN(\nabla u)\|_{L^p(\sigma)} &\leq \|\NN(\nabla\DD(u|_\pom))\|_{L^p(\sigma)} + \|\NN(\nabla\cS(\partial_\nu u|_{\partial \om}))\|_{L^p(\sigma)}\\ &\lesssim  \left(\max_{1\leq k\leq n+1}\sum_{j=1}^{n+1} \|\partial_{t, j, k} f\|_{L^p(\HH^n|_{\partial^* \Omega})}^p\right)^{1/p}+\|\partial_\nu u|_{\pom}\|_{L^p(\sigma)}\\
&\lesssim \|\nabla_{t} f\|_{L^p(\HH^n|_{\partial^* \om})}+\|\nabla_{H,p} f\|_{L^p(\sigma)}\lesssim \|\nabla_{H,p} f\|_{L^p(\sigma)}.
\end{align*}
\end{proof}

\vv

To complete the proof of Lemma \ref{lem6.2*}, it remains to show that $$ \|\partial_{t,j,k} f\|_{L^p(\HH^n|_{\partial^* \om})} \lesssim \|\nabla_{t} f\|_{L^p(\HH^n|_{\partial^* \om})} \,\, \textup{ for every}\,\, j,k \in \{1, 2, \dots, n+1\}.$$ This follows from the next result.

\begin{lemma}\label{lemreal}
Let $\Omega\subset \R^{n+1}$ be a bounded domain with uniformly $n$-rectifiable boundary. 
Then there exists $C>0$ such that for any Lipschitz function $f:\pom\to\R$ and every $j,k \in \{1, 2, \dots, n+1\}$, we have
 \begin{equation}\label{eqnab711}
|\partial_{t,j,k} f(x)| \leq C |\nabla_{t} f(x)| \quad\mbox{ for $\HH^n|_{\partial^* \om}$-a.e.\ $x\in\partial^* \om $.}
\end{equation}
\end{lemma}

\begin{proof}
Let 
$$
E=\{ x \in \partial \om: \nabla_t f (x)\,\textup{exists}\}\quad  \textup{and}\quad E_0= \left\{x \in E: \frac{\HH^n(B(x,r) \cap E)}{\HH^n(B(x,r)}=1\right\},
$$
 and notice that $\HH^n(\partial \om \setminus E_0)=0$.
 Then, for all $\ve>0$ and $r>0$, let
\begin{multline*}
E_{\ve,r} = \bigg\{x\in E_0: \sup_{y\in B(x,r)\cap\pom}\frac{|f(y)- f(x)|}{|y-x|}\leq |\nabla_t f(x)| + \ve
\\ \text{ and } \avint_{B(x,s)\cap E_0} |\nabla_t f(y) - \nabla_t f(x)|\,d\HH^n\leq \ve \,\text{ for $0<s\leq r$} \,\bigg\},
\end{multline*}
and remark that $\HH^n(\pom\setminus E_{\ve,r})\to 0$ as $r\to0$.
We will show that for $\HH^n|_{\partial^* \Omega}$-almost every point  $x\in\partial^* \om$ which is a Lebesgue point for $\partial_{t,j,k} f$ with respect to $\HH^n|_{\partial^* \om}$,
a Lebesgue point for $\nabla_t f$ with respect to $\sigma$, and
also a density point of $E_{\ve,r}$ with respect to $\sigma$, we have
\begin{equation}\label{eqeer6}
|\partial_{t,j,k} f(x)| \lesssim |\nabla_{t} f(x)| + \ve.
\end{equation}
The lemma follows from this estimate, by letting  $x\in\partial^* \om$ be a density point for all the sets  $E_{\ve,r}$, with 
$\ve,r\in (0,1)\cap\Q$.

We fix now $x\in \partial^* \om$ as above, i.e.,  $x$ is a Lebesgue point for $\partial_{t,j,k} f$, a Lebesgue point for $\nabla_t f$, and
also a density point of $E_{\ve,r}$.
To prove \rf{eqeer6} for this point $x$, we first check that 
\begin{equation}\label{eqlipff}
{\rm Lip}(f|_{E_{\ve,r}\cap B(x,r/2)})\leq |\nabla_tf(x)| + C\,\ve.
\end{equation}
Indeed, for all $y,z \in E_{\ve,r}\cap B(x,r/2)$, 
 by the definition of $E_{\ve,r}$, we have
\begin{align*}
\frac{|f(y) - f(z)|}{|y-z|} \leq |\nabla_t f(y)| + \ve &\leq  \avint_{B(y,r/2)} |\nabla_t f - \nabla_t f(y)|\,d\sigma \\
& \quad + \avint_{B(y,r/2)} |\nabla_t f- \nabla_t f(x)|\,d\sigma + |\nabla_t f(x)|
+\ve.
\end{align*}
The first integral on the right hand side is bounded by $\ve$ and the second one by
$$C\avint_{B(x,r)} |\nabla_t f- \nabla_t f(x)|\,d\sigma \lesssim \ve,$$
taking into account that $B(y,r/2)\subset B(x,r)$.
Thus,
$$\frac{|f(y) - f(z)}{|y-z|} \leq |\nabla_t f(x)| + C\,\ve,$$
which proves \rf{eqlipff}.

Next we consider the Lipschitz extension $g:\R^{n+1}\to\R$ of $f|_{E_{\ve,r}\cap B(x,r/2)}$ given by the usual formula
\begin{equation}\label{eqext45}
g(y) = \inf_{z\in E_{\ve,r}\cap B(x,r/2)}\big(f(z) + {\rm Lip}(f|_{E_{\ve,r}\cap B(x,r/2)})\, |y-z|\big),
\end{equation}
so that ${\rm Lip}(g) = {\rm Lip}(f|_{E_{\ve,r}\cap B(x,r/2)})\leq |\nabla_tf(x)| + C\,\ve$.
We also define the function $h:\pom\to\R$ given by
$$h= f - g|_\pom.$$
Notice that 
$${\rm Lip}(h) \leq {\rm Lip}(f) + {\rm Lip}(g)\leq 2 \,{\rm Lip}(f) + C\,\ve\leq 3\,{\rm Lip}(f),$$
assuming $\ve$ small enough (and also that ${\rm Lip}(f)\neq 0$; otherwise the lemma is trivial).
Observe also that
\begin{equation}\label{eqclau000}
h(y)=0 \quad \mbox{ for all $y\in E_{\ve,r}\cap B(x,r/2)$,}
\end{equation}
 because $g=f$ in $E_{\ve,r}\cap B(x,r/2)$.

To estimate $\partial_{t,j,k} f(x)$, we consider a radial $C^\infty$ function $\vphi$ such that
$\chi_{B(0,1/2)}\leq \vphi\leq \chi_{B(0,1)}$, with $\|\nabla \vphi\|_\infty\lesssim1$,
 and we denote $\vphi_s(x) = s^{-n}\,\vphi(s^{-1}x)$ for any $s>0$.
Since $\HH^n|_{\partial^* \om}$ is a Radon measure, by the Lebesgue differentiation theorem,  for $\HH^n|_{\partial^* \om}$-a.e. $x \in \partial^* \om$, we infer that $\partial_{t,j,k} f(x)$ equals 
$$ \lim_{s\to 0} \bigg[\bigg(\int_{\partial^* \om} \vphi\Big(\frac{y-x}s\Big)\,d\HH^n(y)\bigg)^{-1}
\int_{\partial^* \om}  \partial_{t,j,k} f(y)\,\vphi\Big(\frac{y-x}s\Big)\,d\HH^n(y)\bigg].$$
As $\partial^* \Omega$ is $n$-rectifiable, we have that for $\HH^n|_{\partial^* \om}$-a.e. $x \in \partial^* \om$ the $n$-Hausdorff density exists and is equal to $1$ and tangent measures are flat. Thus, there exists a $n$-dimensional plane $V_x$ such that
$$
\lim_{s \to 0}  (2s)^{-n} \int_{\partial^* \om} \vphi\Big(\frac{y-x}s\Big)\,d\HH^n(y)=\int_{V_x} \vphi(z) \, d\HH^n(z)\approx_n 1,
$$
where the implicit constant is independent of $x$. Therefore, for such $x \in \partial^* \om$, it holds that 
\begin{align}\label{eqspl0000}
|\partial_{t,j,k} f(x)| & \lesssim \limsup_{s\to 0} \bigg|
\int_{\partial^* \om}  \partial_{t,j,k} f(y)\,\vphi_s(y-x)\,d\HH^n(y)\bigg|.
\end{align}

For simplicity, we define  $\sigma_*=\HH^n|_{\partial^* \om}$, which  clearly satisfies  $\sigma_*(B(x,r))\lesssim r^n$, for every $x \in \partial^* \Omega$ and every $0<r <\diam(\pom)$. Next we will estimate the integral on the right hand side above for $s\leq r/2$.
We split
\begin{align}\label{eqspl00}
 \bigg|
\int \partial_{t,j,k} f(y)\,\vphi_s(y-x)\,d\sigma_*(y)\bigg|& \leq 
\bigg|\int \partial_{t,j,k} g(y)\,\vphi_s(y-x)\,d\sigma_*(y)\bigg|\\
& \quad + \bigg|
\int \partial_{t,j,k} h(y)\,\vphi_s(y-x)\,d\sigma_*(y)\bigg|.\nonumber
\end{align}
To deal with the first integral on the right hand side we
consider a regularization of $g$ by means of an approximation of the identity
$\{\phi_\tau\}_{\tau>0}$, where the functions $\phi_\tau$ are radial, $C^\infty$, and supported in $B(0,\tau)$, and we denote $g_{\tau} = g*\phi_\tau$ for each $\tau>0$. 
Then $g_{\tau}\in C^\infty(\R^{n+1})$ and satisfies 
$$\|\nabla g_{\tau}\|_\infty \leq {\rm Lip}(g) \leq |\nabla_tf(x)| + C\,\ve.$$
The identity \rf{eqparts099} holds for $g_{\tau}$ because $g_\tau\in C^\infty(\R^{n+1})$, and thus
$$\|\partial_{t,j,k} g_\tau\|_\infty \leq 2\,|\nabla_tf(x)| + C\,\ve.$$
Thus, since $g_\tau$ converges uniformly to $g$ in $\pom$ as $\tau\to 0$, 
\begin{align}\label{eqghj4}
\bigg|\int \partial_{t,j,k} g(y)\,\vphi_s(y-x)\,d\sigma_*(y)\bigg| & =
\bigg|\int  g(y)\,\partial_{t,k,j} \vphi_s(y-x)\,d\sigma_*(y)\bigg|\\
& =
\lim_{\tau\to 0} \bigg|\int  g_\tau (y)\,\partial_{t,k,j}\vphi_s(y-x)\,d\sigma_*(y)\bigg|\notag
\\
& =
\lim_{\tau\to 0} \bigg|\int \partial_{t,j,k} g_\tau (y)\,\vphi_s(y-x)\,d\sigma_*(y)\bigg| \notag
\\
&\lesssim |\nabla_tf(x)| + \ve.\notag
\end{align}

Now we will estimate the last integral in \rf{eqspl00}. By \rf{eqclau000}, we have
\begin{align}\label{eqclar0h}
\bigg|
\int \partial_{t,j,k} h(y)\,\vphi_s(y-x)\,d\sigma_*(y)\bigg| &\leq
\int |h(y)\,\partial_{t,k,j}\vphi_s(y-x)|\,d\sigma_*(y)\\
& \lesssim \frac1{s^{n+1}} \int_{B(x,s)\setminus E_{\ve,r}
} |h(y)|\,d\sigma_*(y).\notag
\end{align}
Recall now that $x$ is a density point of $E_{\ve,r}$ with respect to $\sigma$. Hence, given any arbitrary $\delta>0$, we can choose $s_0>0$ small enough so that
\begin{equation}\label{eqclar01}
\sigma(B(x,2s)\setminus E_{\ve,r}) < \delta \,\sigma(B(x,2s)\quad \mbox{ for $0<s\leq s_0$}.
\end{equation}
For $0<s\leq s_0$ this implies {that for all $y\in B(x,s)\cap \pom$, there exists some $y'\in E_{\ve,r}$ such that 
$|y-y'|\leq 2C_0^{2/n}\,\delta^{1/n}s$,
where $C_0$ is the AD-regularity constant of $\pom$. Otherwise, assuming $2C_0^{2/n}\,\delta^{1/n} \leq 1$,
$$B(y,2C_0^{2/n}\,
\delta^{1/n}s)\subset B(x,2s) \setminus E_{\ve,r},$$
and thus 
$$\sigma\big(B(x,2s) \setminus E_{\ve,r}\big) \geq \sigma(B(y,2C_0^{2/n}\,
\delta^{1/n}s)) \geq C_0\, \delta\,(2s)^n \geq \delta\,\sigma(B(x,2s)),$$
which would contradict \rf{eqclar01}. Hence, using that ${\rm Lip}(h)\lesssim {\rm Lip}(f)$, for all points $y\in  B(x,s)\cap \pom$ and $y'$ as above,} we get
$$|h(y)| \leq |h(y'| + |h(y') - h(y)| \leq 0 + C\,{\rm Lip}(f)\,|y-y'|\lesssim 
{\rm Lip}(f)\,\delta^{1/n}\,s.$$
Plugging this estimate into \rf{eqclar0h}, we obtain
$$\bigg|
\int \partial_{t,j,k} h(y)\,\vphi_s(y-x)\,d\sigma_*(y)\bigg|\lesssim \frac1{s^{n+1}}
{\rm Lip}(f)\,\delta^{1/n}\,s\,\sigma_*(B(x,s))\lesssim {\rm Lip}(f)\,\delta^{1/n}.
$$
Since $\delta$ can be taken arbitrarily small, we infer that
$$\lim_{s\to0} 
\int \partial_{t,j,k} h(y)\,\vphi_s(y-x)\,d\sigma_*(y) =0.$$
Together with \rf{eqspl0000}, \rf{eqspl00}, and
\rf{eqghj4}, this gives
$$|\partial_{t,j,k} f(x)|\lesssim \limsup_{s\to0}\bigg|
\int \partial_{t,j,k} f(y)\,\vphi_s(y-x)\,d\sigma_*(y)\bigg| \lesssim |\nabla_tf(x)| + \ve,$$
which proves \rf{eqeer6} and concludes the proof of the lemma.
\end{proof}
\vv

The next lemma clarifies the
relationship between $\nabla_t f$ and $\partial_{t,j,k} f$ for Lipschitz functions.

\begin{lemma}\label{lemreal**}
Let $\Omega\subset \R^{n+1}$ be a bounded domain with uniformly $n$-rectifiable boundary. Then, for any 
Lipschtiz function $f:\pom\to\R$ and for every $j,k \in \{1, 2, \dots, n+1\}$, we have
\begin{equation}\label{eqnab720}
\partial_{t,j,k} f(x) =-\nu_j \,(\nabla_{t} f)_k(x) + \nu_k\,(\nabla_{t} f)_j(x)\quad\mbox{ for $\HH^n|_{\partial^* \om}$-a.e.\ $x\in\partial^* \om$.}
\end{equation}
\end{lemma}

\begin{proof}
Since $f$ is Lipschitz, we have that $|\nabla_t f(x)|\leq {\rm Lip}(f)$ for $\sigma$-a.e.\ $x\in\pom$, or equivalently, $\nabla_t f(x)\in \overline B(0, {\rm Lip}(f))$ for $\sigma$-a.e.\ $x\in\pom$.
Given any $\ve>0$, we cover $\overline B(0, {\rm Lip}(f))$ with a family of closed balls $B_{\ve,\ell}$, $\ell=1,\ldots,N(\ve)$, with radius $\ve\,{\rm Lip}(f)$ and bounded overlaps, applying Besicovitch covering theorem. 

We denote by $v_{\ve,\ell}$ the center of $B_{\ve,\ell}$ and we let
$$A_{\ve,\ell}= \big\{x\in \pom: \nabla_t f(x)\in B_{e,\ell}\big\},$$
(we understand that $A_{\ve,\ell}$ contains only points $x$ where $\nabla_t f(x)$ exists),
so that
\begin{equation}\label{eqllarg98}
|\nabla_t f(x) - \Pi_x(v_{\ve,\ell})| =
|\Pi_x\big(\nabla_t f(x) - v_{\ve,\ell}\big)| \leq |\nabla_t f(x) - v_{\ve,\ell}| \leq \ve{\rm Lip}(f)\end{equation}
for all $x\in A_{\ve,\ell}$,
where $\Pi_x$ is the orthogonal projection on the  tangent hyperplane to $\pom$ in $x$.
We consider the function $g_{\ve,\ell}:\pom\to\R$ defined by
$$g_{\ve,\ell}(x) = f(x) - v_{\ve,\ell}\cdot x.$$
Observe that $g_{\ve,\ell}$ is Lipschitz and, since the tangential gradient of the function 
$h_{\ve,\ell}(x) = v_{\ve,\ell}\cdot x$ equals
$$\nabla_th_{\ve,\ell}(x) = \Pi_x(\nabla h_{\ve,\ell}(x)) =  \Pi_x(v_{\ve,\ell}),$$
we have
$$|\nabla_t g_{\ve,\ell}(x)| = |\nabla_tf(x) - \nabla_th_{\ve,\ell}(x)| = |\nabla_tf(x) - \Pi_x(v_{\ve,\ell})|
\leq \ve\,{\rm Lip}(f)$$
for all $x\in A_{\ve,\ell}$,
where we used \rf{eqllarg98} in the last inequality.

By the preceding estimate and Lemma \ref{lemreal} applied to to  $g_{\ve,\ell}$, we obtain
\begin{equation}\label{eqllarg99}
|\partial_{t,j,k} g_{\ve,\ell}(x)| \lesssim |\nabla_t g_{\ve,\ell}(x)| \leq \ve\,{\rm Lip}(f)
\quad \mbox{ for $\HH^n|_{\partial^* \om}$-a.e.\ $x\in A_{\ve,\ell}$}.
\end{equation}
Observe now that, since $h_{\ve,\ell}$ is a $C^\infty$ function, by \rf{eqclau30}, for $\HH^n|_{\partial^* \om}$-a.e. \ $x \in \partial^* \om$,
\begin{align}\label{eqlla5}
\partial_{t,j,k} g_{\ve,\ell}(x) &= \partial_{t,j,k} f(x)  -\partial_{t,j,k} h_{\ve,\ell}(x) \\
& = \partial_{t,j,k}f(x) + \nu_j(x) (\nabla_{t} h_{\ve,\ell}(x))_k-\nu_k(x) (\nabla_{t} h_{\ve,\ell}(x))_j \notag\\
 &=\partial_{t,j,k} f(x) + \nu_j(x) (\Pi_x(v_{\ve,\ell}))_k-\nu_k(x) (\Pi_x(v_{\ve,\ell}))_j.\notag
\end{align}
Therefore, by the triangle inequality, for $\HH^n|_{\partial^* \om}$-a.e.\ $x\in A_{\ve,\ell}$,
\begin{align}\label{eqdjkq24}
|\nu_j(x)(&\nabla_{t} f(x))_k - \nu_k(x)(\nabla_t f(x))_j + \partial_{t,j,k} f(x)| \\
  & \leq |\nu_j(x)(\nabla_{t} f(x)-\Pi_x(v_{\ve,\ell}))_k - \nu_k(x)(\nabla_t f(x)-\Pi_x(v_{\ve,\ell}))_j|\notag\\
&+|\partial_{t,j,k} f(x) + \nu_j(x) (\Pi_x(v_{\ve,\ell}))_k-\nu_k(x) (\Pi_x(v_{\ve,\ell}))_j|\lesssim\ve\,{\rm Lip}(f),\notag
\end{align}
where we applied \rf{eqllarg98} to estimate the first term in the middle part of the chain of inequalities, and \rf{eqllarg99} and \rf{eqlla5}
for the second term. 
Since \rf{eqdjkq24} holds $\sigma$-a.e.\ for all $A_{\ve,\ell}$, we infer that 
$$|\nu_j(x)(\nabla_{t} f(x))_k - \nu_k(x)(\nabla_t f(x))_j + \partial_{t,j,k} f(x)|  \lesssim \ve\,{\rm Lip}(f)$$
for  $\HH^n|_{\partial^* \om}$-a.e.\ $x\in\partial^* \om$.
Finally, since $\ve$ can be taken arbitrarily small, \rf{eqnab720} follows.
\end{proof}

\vv

The second part of Theorem \ref{teomain}, which deals with the problem $(\wt R_p)$, is deduced from 
Lemma \ref{lem6.2*} and Lemma \ref{lemtec99}, which shows the connection between the Haj\l asz gradient
and the tangential gradient for Lipschitz functions under the assumptions in Theorem \ref{teomain}.
We proceed to prove it now.

\begin{proof}[Proof of Lemma \ref{lemtec99}]
Remember that we have to show that, or any Lipschitz function $f:\pom\to\R$, 
\begin{equation}\label{eqfi99}
\|\nabla_{H,p} f\|_{L^p(\sigma)} \approx \|\nabla_t f\|_{L^p(\sigma)},
\end{equation}
assuming either that  $\pom$ admits a weak $(1,p)$-Poincaré inequality  or $\Omega$ satisfies the two sided local John condition.

In \rf{eqfac71} we showed that 
$$|\nabla_t f(x)|\lesssim \cM_\sigma(\nabla_{H,p} f)(x)\quad \mbox{ for $\sigma$-a.e.\ $x\in \pom$.}$$
It is immediate to check that the arguments in \rf{eqfac71} work just under the assumption that 
the boundary of $\pom$ is uniformly $n$-rectifiable. So we deduce that
$$\|\nabla_t f\|_{L^p(\sigma)}\lesssim \|\nabla_{H,p} f\|_{L^p(\sigma)}.$$

In the converse direction, consider first the case when $\Omega$ satisfies the two sided local John condition,  which implies that $\pom$ is uniformly rectifiable and satisfies $\HH^n(\pom \setminus \partial^*\om)=0$ (see e.g. \cite[Corollary 3.14]{HMT} for the proof). By \cite[Theorem 4.27]{HMT}, we know that the Haj\l asz space
$W^{1,p}(\sigma)$ coincides with the Hofmann-Mitrea-Taylor Sobolev space $L^p_1(\pom)$ defined in 
\rf{eqlp1HMT} and, by  \cite[display  (4.3.20)]{HMT}, we have
$$\|\nabla_{H,p} f\|_{L^p(\sigma)} \lesssim \|\nabla_{t,HMT} f\|_{L^p(\sigma)},$$
 where 
 $$
\nabla_{t,HMT} f(x) := \Big(\sum_k \nu_k\,\partial_{t,j,k} f(x)\Big)_{1\leq j\leq n+1} \quad \textup{for all} \,\, x \in \partial^* \Omega.
$$
By   \cite[Lemma 3.40]{HMT} and Lemma \ref{lemreal**},  $\nabla_{t,HMT} f= -\nabla_{t} f$ $\sigma$-a.e., and so
we get $\|\nabla_{H,p} f\|_{L^p(\sigma)} \lesssim \|\nabla_t f\|_{L^p(\sigma)}$. Remark that if
$f$ is $C^1$ in a neighborhood of $\pom$, we can rely on the simpler Lemma 
\ref{lemHMT-tan}.
\vv

 Suppose now that $\pom$ admits a weak $(1,p)$-Poincaré inequality. In this case, probably \rf{eqfi99} is already known to hold. However, for completeness we show the details here.
By the Keith-Zhong theorem
\cite{KZ}, it turns out that
 $\pom$ admits a weak $(1,q)$-Poincaré inequality for some $q<p$. 
Remark now that 
$$\rho(x):=\limsup_{\pom\ni y\to x} \frac{|f(y) - f(x)|}{|y-x|}$$
is an upper gradient for $f$. Indeed, it is easy to check that
for any rectifiable curve $\gamma
\subset\pom$ with endpoints $x,y\in\pom$,
$$|f(x) - f(y)|\leq \int_\gamma \rho\,d\HH^1.$$
(see  \cite[Exercise 7.25]{Heinonen}).
Consequently, since $\pom$ supports a weak $(1,q)$ Poincaré inequality, there exists some $\Lambda>1$ such that for any ball $B$ centered at $\pom$ it holds
\begin{equation}\label{eqpoin89}
\avint_B |f- m_{B, \sigma}|\,d\sigma \leq C\,r(B)\left(\avint_{\Lambda B}\rho^q\,d\sigma\right)^{1/q}
=  C\,r(B)\left(\avint_{\Lambda B}|\nabla_t f|^q\,d\sigma\right)^{1/q},
\end{equation}
because $\rho = |\nabla_t f|$ $\sigma$-a.e., by Lemma \ref{lem2.2}.

We denote $B_{x,k}=B(x,2^{-k}|x-y|)$, $B_{y,k}=B(y,2^{-k}|x-y|)$ for $k\in\Z$, and we write
\begin{align*}
|f(x) - f(y)| &\leq |f(x) - m_{B_{x,-1}, \sigma}f| + |m_{B_{x,-1}, \sigma}f - m_{B_{y,0}, \sigma}f|\\ &\quad + |m_{B_{y,0},\sigma}f - f(y)|
\end{align*}
Using the Poincaré inequality \rf{eqpoin89} and the fact that $B_{y,0}\subset B_{x,-1}$, with $r(B_{y,0})\approx r(B_{x,-1})$, it follows that
\begin{align*}
|m_{B_{x,-1}, \sigma}f - &m_{B_{y,0},\sigma}f| \leq \avint_{B_{y,0}}|f- m_{B_{x,-1}, \sigma}f|\,d\sigma
\lesssim \avint_{B_{x,-1}}|f- m_{B_{x,-1}, \sigma}f|\,d\sigma\\
& \lesssim
|x-y|\left(\avint_{\Lambda B_{x,-1}}|\nabla_t f|^q\,d\sigma\right)^{1/q} 
\leq |x-y|\,\cM_{\sigma,q} (\nabla_t f)(x),
\end{align*}
where $\cM_{\sigma,q}$ is the maximal operator defined by
$$\cM_{\sigma,q}g(x) = \sup_{r>0}\left(\avint_{B(x,r)} |g|^q\,d\sigma\right)^{1/q}.$$
We also have
\begin{align*}
&|f(x) - m_{B_{x,-1}, \sigma}f|  \leq \sum_{k\geq-1} |m_{B_{x,k}, \sigma}f - m_{B_{x,k+1},  \sigma}f|\\
&\lesssim \sum_{k\geq-1} 2^{-k}|x-y|\left(\avint_{\Lambda B_{x,k}}|\nabla_t f|^q\,d\sigma\right)^{1/q} 
\lesssim |x-y|\,\cM_{\sigma,q} (\nabla_t f)(x).
\end{align*}
By analogous estimates, we get
$$|f(y) - m_{B_{y,0}, \sigma}f|\lesssim |x-y|\,\cM_{\sigma,q} (\nabla_t f)(y).$$
Therefore,
$$|f(x) - f(y)| \lesssim |x-y|\,\big(\cM_{\sigma,q} (\nabla_t f)(x) +\cM_{\sigma,q} (\nabla_t f)(y)\big).$$
Consequently, $C\cM_{\sigma,q} (\nabla_t f)$ is a Haj\l asz upper gradient for $f$, and since
 $\nabla_{H,p}f$ has minimal $L^p(\sigma)$ norm over all Haj\l asz upper gradients and
$\cM_{\sigma,q}$ is bounded in $L^p(\sigma)$, we obtain
$$\|\nabla_{H,p}f\|_{L^p(\sigma)} \lesssim \|\cM_{\sigma,q} (\nabla_t f)\|_{L^p(\sigma)} \lesssim
\|\nabla_t f\|_{L^p(\sigma)}.$$
\end{proof}

\vv


\section{Solvability of the regularity problem with boundary data in $W^{1,p}(\pom)$}
\label{sec7*}

This section is mainly devoted to the proof of Theorem \ref{thm:1.3}.
First we show that we can easily extend the solvability of $(R_p)$ to boundary data in $C(\pom) \cap W^{1,p}(\pom)$:

\begin{theorem}\label{thrm:continuous-regularity}
Let $\Omega\subset\R^{n+1}$ be a bounded corkscrew domain with $n$-AD-regular boundary. 
If there exists $p \in (1, 2+\ve_0)$ such that $(D_{p'})$  is solvable and $\ve_0$ is  defined in Theorem \ref{teomain}, then $(R_p)$ is solvable for continuous functions in the sense that, for any continuous function $f\in W^{1,p}(\pom)\cap C(\pom)$, the solution $u$ of the Dirichlet problem satisfies
$$\|\NN(\nabla u)\|_{L^p(\sigma)} \lesssim \| f\|_{\dot W^{1,p}(\pom)}.$$
\end{theorem}

\begin{proof}
Let $1<p<\infty$ and $f\in W^{1,p}(\Omega)\cap C(\pom)$.
In view of   \cite[Theorem 5]{Hajlasz},  there exists a sequence of Lipschitz functions 
$f_k:\pom\to\R$ which converge to $f$ in $W^{1,p}(\pom)$. Further, the functions $f_k$ constructed
in that theorem converge   pointwise to $f$ $\sigma$-a.e. as well and they are uniformly bounded
(with $L^\infty$ norm depending on $\|f\|_\infty$ and $\diam(\pom)$.

By Theorem \ref{teomain}, we know that the solution $u_k$ of the Dirichlet problem with
boundary data $u_k$ satisfies
\begin{equation}\label{equkk99}
\|\NN(\nabla u_k)\|_{L^p(\sigma)} \lesssim \|\nabla_{H,p}f_k\|_{L^p(\sigma)}.
\end{equation}
Since the functions $f_k$ are uniformly bounded,  then the sequence of harmonic extensions $u_k$  is uniformly bounded, by the maximum principle.
Then, passing to a subsequence, if necessary, we can assume that both the functions $u_k$ and their
gradients $\nabla u_k$ converge uniformly on compact subsets of $\Omega$. 
Then, denoting 
$\wt u = \lim_{k\to\infty} u_k,$
by \rf{equkk99} it easily follows  that 
\begin{align}\label{equkk999}
\|\NN(\nabla \wt u)\|_{L^p(\sigma)}& \leq \limsup_{k\to\infty} \|\NN(\nabla u_k)\|_{L^p(\sigma)}\\ &\nonumber  \lesssim \limsup_{k\to\infty}\|\nabla_{H,p}f_k\|_{L^p(\sigma)}
= \|\nabla_{H,p}f\|_{L^p(\sigma)}.
\end{align}

On the other hand, by the $L^{p'}$-solvability of the Dirichlet problem, it follows that the harmonic measure in $\Omega$, denoted by $\omega_\Omega$, and the surface measure on $\pom$ are mutually absolutely continuous. So, the functions $f_k$ converge to $f$ pointwise $\omega_\Omega$-a.e.
Consequently, by the dominated convergence theorem, for any fixed $x\in\Omega$,
$$\lim_{k\to\infty} u_k(x) = \lim_{k\to\infty} \int f_k\,d\omega^x = \int f\,d\omega^x = u(x),$$
where $u$ is the solution the Dirichlet problem with boundary data $f$.
So we deduce that $u=\wt u$ and $\nabla u=\nabla\wt u$ and then by \rf{equkk999} the lemma follows.
\end{proof}

\vv

In the next theorem we extend the definition of solvability of $(R_p)$ to general boundary data in $W^{1,p}(\pom)$ and we study the pointwise convergence at the boundary.

\begin{theorem}[\bf Solvability]\label{thm:regularity-noncont}
Let $\Omega\subset\R^{n+1}$ be a bounded corkscrew domain with $n$-AD-regular boundary. 
 If $(R_p)$ is solvable, then for any $f \in  W^{1,p}(\pom)$ there exists a harmonic function $u$ in $\Omega$ such that $\|\NN(\nabla u)\|_{L^p(\sigma)} \leq C \|f\|_{\dot W^{1,p}(\pom)}$ and $u \to f$ non-tangentially $\sigma$-a.e. on $\pom$.
\end{theorem}

\begin{proof}
In order to define a harmonic solution $u$ for $f \in  W^{1,p}(\pom)$ as required in the theorem, 
first we consider $f \in  W^{1,p}(\pom) \cap \Lip(\pom)$ and we explain how to extend $f$ to the whole $\R^{n+1}$ in a controlled way in $\dot W^{1,p}(\pom)$.
Suppose that $0 \in \pom$, and define $\wt f= f-  f(0)$ on $\pom$ and $\wt f=0$ on  $ \partial B(0, 4 \diam(\pom))$. Note that if $x \in \pom$ and $y \in \partial B(0, 4 \diam(\pom))$, it holds that 
\begin{align*}
|\wt f(x) - \wt f(y)|&=| f(x) -  f(0)| \leq |x| (\nabla_{H,p} f(x)+\nabla_{H,p} f(0))\\
& \leq \diam(\pom) (\nabla_{H,p} f(x)+\nabla_{H,p} f(0))\\& \leq |x-y| (\nabla_{H,p} f(x)+\nabla_{H,p} f(0)).
\end{align*}
Set $\widetilde \Omega:=B(0,4 \diam(\pom)) \setminus \pom$ and 
\begin{equation*}
\wt \sigma=
\begin{dcases}
\sigma, &\textup{on}\,\, \pom\\
\HH^n|_{\partial B(0,4 \diam(\pom))},&\textup{on} \,\,\partial B(0,4 \diam(\pom)).
\end{dcases}
\end{equation*} 
Set also
\begin{equation*}
\wt g(x)=
\begin{dcases}
\nabla_{H,p} f(x), &\textup{if}\,\,x\in \pom\\
\nabla_{H,p} f(0),&\textup{if}\,\,x\in \partial B(0, 4 \diam(\pom)).
\end{dcases}
\end{equation*} 
 Since $\partial\widetilde \Omega$ is compact, $\nabla_{H,p} f(0) \in L^p(\wt \sigma)$  and thus, $\wt g$ is a Haj\l asz upper gradient of $\wt f$ on $ \partial \widetilde \Omega$. It is clear from the argument above that $\wt f \in \Lip(\partial \widetilde \Omega)$
and so we can construct $\wt w$ in $\widetilde \Omega$ as in Lemma \ref{lem:Lip-ext} such that 
$\wt w$ is Lipschitz continuous in $B(0,4 \diam(\pom))$ with $\Lip(\wt w) \lesssim \Lip (\wt f)$ and $\wt w \to \wt f$ continuously on $\partial \widetilde \Omega$. Finally we extend $\wt w$ by $0$ in $\R^{n+1} \setminus B(0, 4 \diam(\pom))$. 

Additionally, it is not hard to show that
\begin{equation}\label{eq:6.6}
\|\NN_{\widetilde \Omega}(\nabla \wt w)\|_{L^p(\wt \sigma)} \leq C \|\wt f\|_{\dot W^{1,p}(\partial \widetilde \Omega)}.
\end{equation}
Indeed, fix $\xi \in \partial \widetilde \Omega$ and let $Q \in \DD(\partial \widetilde \Omega)$ so that $\xi \in Q$. If $P_0  \in w(Q)$ and  $x\in 1.1 P_0$, by Lemma \ref{lem46},
\begin{align*}
|\nabla \wt w(x)| &\lesssim \frac1{\ell(Q)}\,m_{CB_{Q}, \sigma}(\nabla_{H,p} \wt f)\,\ell(Q) + m_{CB_{Q}, \sigma}(\nabla_{H,p} \wt f)\\
& \lesssim m_{CB_{Q}, \sigma}(\nabla_{H,p} \wt f) \lesssim \mathcal{M}_{\wt \sigma}(\nabla_{H,p} \wt f)(\xi),
\end{align*}
 which readily implies that 
$$
\NN_{\widetilde \Omega}(\nabla \wt w)(\xi) \lesssim \mathcal{M}_{\wt \sigma}(\nabla_{H,p} \wt f)(\xi),
$$
and thus \eqref{eq:6.6}. 
Thus,
\begin{align*}
\int_{\wt \Omega} |\nabla \wt w|^p\,dm & \approx \!\sum_{Q\in \DD(\partial \wt\Omega)} \int_{w(Q)} |\nabla \wt w|^p\,dm 
\leq \!\sum_{Q\in \DD(\partial \wt\Omega)} \inf_{x\in Q} |\NN_{\wt\Omega}(\nabla \wt w)(x)|^p\,\ell(Q)^{n+1}\\
& \lesssim \sum_{Q\in \DD(\partial \wt\Omega)} \!\!\ell(Q) \int_Q |\NN_{\wt\Omega}(\nabla \wt w)|^pd\sigma\lesssim \diam(\pom)
\int_{\partial\wt\Omega} |\NN_{\wt\Omega}(\nabla \wt w)|^p\,d\sigma\\
& \lesssim \diam(\pom) \|\wt f\|_{\dot W^{1,p}(\partial \widetilde \Omega)}^p \lesssim_{\diam(\pom)} \left(\nabla_{H,p} f(0)^p+\| f\|_{\dot W^{1,p}(\partial  \Omega)}^p \right).
\end{align*}

If $w$ is the solution of the Dirichlet problem with data $\wt f$ in $\Omega$, we set
\begin{equation*}
w'(x)=
\begin{dcases}
w(x), &\textup{if}\,\,x\in \Omega\\
\wt f(x), &\textup{if}\,\,x\in \pom\\
\wt w(x), &\textup{if}\,\,x\in \R^{n+1} \setminus \overline \Omega.
\end{dcases}
\end{equation*} 
Remark that $w \in  W^{1,p}(\Omega) \cap C^\alpha(\overline \Omega)$ for some $\alpha>0$,  where $C^\alpha(\overline \Omega)$ is the class of $\alpha$-H\"older continuous functions in $\overline \Omega$ (see e.g. \cite[Theorem 6.44]{HKM}). Therefore, since $\wt w $ is compactly supported and $\Omega$ is bounded, we readily  infer that $w' \in C^\alpha(\R^{n+1})$ and is compactly supported (thus bounded in $\R^{n+1}$). As we have proved that $w' \in \dot W^{1,p}(\R^{n+1}\setminus \pom)$ and $w'$ is bounded in $\R^{n+1}$, it trivially holds that $w'  \in  W^{1,p}(\R^{n+1}\setminus \pom)$. Applying \cite[Lemma 11]{HaM}, we deduce that $w' \in  W^{1,p}(\R^{n+1})$, and by Sobolev's inequality we obtain
\begin{align*}
 \|w\|_{L^{p^*}(\Omega)} &\leq  \|w'\|_{L^{p^*}(\R^{n+1})} \lesssim  \|\nabla w'\|_{L^{p}(\R^{n+1})} \\ &\lesssim_{\diam(\pom)} \left(\nabla_{H,p} f(0)+\| f\|_{\dot W^{1,p}(\partial  \Omega)} \right).
\end{align*}

Consider now an arbitrary function $f \in  W^{1,p}(\pom)$.  By density, there exists a sequence $f_j \in \Lip(\pom)$ such that $f_j \to f$ in $ W^{1,p}(\pom)$. Passing to a subsequence, we also have that $f_j \to f$ and  $\nabla_{H,p} f_j \to \nabla_{H,p} f$ pointwisely $\sigma$-a.e. on $\pom$, and without loss of generality, we can  assume that $f_j(0) \to f(0)$ and  $\nabla_{H,p} f_j(0) \to \nabla_{H,p} f(0)$. If $u_j(x)=\int \wt f_j\,d\omega^x$, where $\wt f_j= f_j-f_j(0)$, since $(R_p)$ is solvable,  it holds that 
\begin{equation}
\|\NN(\nabla (u_j - u_k))\|_{L^p(\sigma)}  \leq C \|f_j-f_k\|_{\dot W^{1,p}(\pom)},\label{align1}
\end{equation}
\begin{equation}
\|u_j-u_k\|_{L^{p^*}(\Omega)}  \lesssim_{\diam(\pom)} |\nabla_{H,p} f_j(0)- \nabla_{H,p} f_k(0)| + \|f_j-f_k\|_{\dot W^{1,p}(\pom)},\label{align2}
\end{equation}
 with bounds independent of $j$ and $k$, where we used that $\| \wt f_j- \wt f_k\|_{\dot W^{1,p}(\pom)}=\|  f_j-  f_k\|_{\dot W^{1,p}(\pom)}$.  Therefore, by  \rf{align2}, after passing to a subsequence, we can find a harmonic function $u$ in $\Omega$ such that $u_j$ and $\nabla u_j$ respectively converge to $u$ and $\nabla u$ uniformly on compact subsets of $\Omega$ (this follows from Montel's theorem for harmonic functions, taking into account that $\sup_j\|u_j\|_{L^{p^*}(\Omega)}<\infty$).  Moreover, if we define 
\begin{equation}\label{eq:truncated-nt}
 \NN^\ve (F)(\xi):= \sup_{\gamma(\xi) \cap \{x \in\Omega: \delta_\Omega(x)>\ve\}} |F(x)|,
 \end{equation}
  we have that 
 $$
\|\NN^\ve(\nabla u)\|_{L^p(\sigma)} \lesssim \liminf_{j \to \infty}\|\NN^\ve(\nabla u_j)\|_{L^p(\sigma)} \lesssim \liminf_{j \to \infty} \|f_j\|_{\dot W^{1,p}(\pom)}=\|f\|_{\dot W^{1,p}(\pom)}.
 $$
Since  $\NN^\ve(\nabla u)$ increases to $\NN(\nabla u)$ as $\ve \to 0$, it readily holds 
  $$
\|\NN(\nabla u)\|_{L^p(\sigma)} \lesssim \|f\|_{\dot W^{1,p}(\pom)}.
$$

Now it remains to show that, for any $\alpha>0$, 
$$\lim_{\gamma_\alpha(\xi)\ni x\to \xi} u(x) = \wt f(\xi) = f(\xi) - f(0)
\quad \mbox{ for $\sigma$-a.e.\ $\xi\in\pom$.}$$
To this end, first we take $\beta>\alpha$  and we note that, by Lemma \ref{lemsurfacetot}, 
 for any fixed aperture $\beta>0$, for $\sigma$-a.e.\ $\xi\in\pom$ there exists 
$r_{\xi,\beta}>0$ such that 
$$
\gamma_\beta(\xi)\cap B(\xi,r_{\xi,\beta}) \subset \bigcup_{R\in\ttt}\Omega_R.
$$
and moreover, for $\sigma$-a.e.\ $\xi\in\pom$ belongs to $\bigcup_{R\in\ttt}\partial\Omega_R^{\pm}$. 
 Then, since, up to a set of zero surface measure, all points in 
$\partial\Omega_R^\pm$ are tangent points and $R\in\ttt$ is countable at most, we deduce that $\sigma$-a.e.\ $\xi\in\pom$ belongs to at most two Lipschitz domains $\Omega_R^+$, $\Omega_R^-$ (it may happen that $\Omega_R^-=\varnothing$).
Further, reducing $r_{\xi,\beta}$ if necessary, we may assume that $\gamma_\beta(\xi)\cap B(\xi,r_{\xi,\beta})\subset \Omega_R=
\Omega_R^+\cup\Omega_R^-$. From this condition for $\alpha<\beta$, it easily follows that $\gamma_\alpha(\xi)\cap B(\xi,r_{\xi,\alpha})$ is contained in a non-tangential cone for $\Omega_R$ for a suitable $r_{\xi,\alpha}\in (0,r_{\xi,\beta})$, and moreover we can write
$$\gamma_\alpha(\xi)\cap B(\xi,r_{\xi,\alpha}) := \big(\gamma_\alpha^+(\xi)\cap B(\xi,r_{\xi,\alpha})\big) \cup 
\big(\gamma_\alpha^-(\xi)\cap B(\xi,r_{\xi,\alpha})\big),$$
with $\gamma_\alpha^\pm(\xi)\cap B(\xi,r_{\xi,\alpha})\subset\Omega_R^\pm$. By abusing notation, we assume that
$\gamma_\alpha^\pm(\xi)$ are also non tangential cones for $\Omega_R^\pm$.

Next we apply  a well-known argument of Kenig and Pipher (see \cite[pp. 461-462]{KP}) which shows that 
if $\NN_{\Omega_R^\pm}(\nabla u)(\xi)<\infty$, then 
$u$ has non-tangential trace $\sigma$-a.e.\ on $\pom_R^\pm$ in the sense that both limits
\begin{equation}\label{eqlimopw}
\lim_{\gamma_\alpha^+(\xi)\ni x\to \xi} u(x), \quad \lim_{\gamma_\alpha^-(\xi)\ni x\to \xi}u(x)
\end{equation}
exist. Indeed, for arbitrary $x,y\in \gamma_\alpha^+(\xi)\cap B(\xi,r_{\xi,\alpha})$ there exists a non-tangential path $\ell_{x,y}$ (for $\Omega_R^+$) joining $x$ and $y$ contained in $\gamma_\alpha^+(\xi)$ with
length comparable to $|x-y|$. So
$$|u(x)-u(y)|\leq \int_{\ell_{x,y}}|\nabla u|\,d\HH^1 \lesssim \NN_{\Omega_R^+}u(\xi)\,|x-y|\leq \NN u(\xi)\,|x-y|
.$$
Here we assumed
that $\NN=\NN_\alpha$ as $\| \NN_\alpha(\nabla u)\|_{L^p(\pom)} \approx_{\alpha, \beta} \| \NN_\beta(\nabla u)\|_{L^p(\pom)}$ for any $\alpha \neq \beta$,  by \cite[Proposition 2.2]{HMT}.
By the Cauchy criterion, it follows that 
$$ u^+(\xi):=\lim_{\gamma_\alpha^+(\xi)\ni x\to \xi} u(x)$$
exists. The same argument shows the existence of $u^-(\xi):=\lim_{\gamma_\alpha^+(\xi)\ni x\to \xi} u(x)$.
Then we deduce
\begin{equation}\label{eq:trace-ineq}
\sup_{x \in \gamma^\pm_\alpha(\xi)\cap B(\xi, \ve)} |u(x)- u^\pm(\xi)| \lesssim \ve \,
\NN(\nabla u)(\xi)\quad\textup{for} \,\sigma\textup{-a.e.}\,\, \xi \in \pom,\; 0<\ve< r_{\xi,\alpha}.
\end{equation}

Next we should show that $u^+(\xi)=u^-(\xi)=\wt f(\xi)$ $\sigma$-a.e.
Arguing as above, by \eqref{align1},
$$ \|\NN(\nabla (u_j - u) )\|_{L^p(\sigma)} \lesssim  \|f_j-f\|_{\dot W^{1,p}(\pom)} \to 0 \,\,\textup{as}\,\, j \to \infty.
 $$
 Thus, passing to a subsequence, we infer that $\NN(\nabla (u_j - u)) \to 0$ $\sigma$-a.e. on $\pom$.  
For a given $\xi\in\pom$ such that $\NN(\nabla u)(\xi)<\infty$ and $0<\ve<h(\xi)$, we fix $x^\pm\in \gamma^\pm_\beta(\xi)$ with $|x^\pm-\xi|\leq \ve$.
 By \rf{eq:trace-ineq} we have
$$ |u(x^\pm)- u^\pm(\xi)| \lesssim \ve \,\NN(\nabla u)(\xi).$$ 
Then, by the triangle inequality,
\begin{align*}
|&u^\pm(\xi) - \wt f(\xi)| \\&\leq |u^\pm(\xi) - u(x^\pm)| + |u(x^\pm) - u_j(x^\pm)| + |
u_j(x^\pm) - \wt f_j(\xi)| + |\wt f_j(\xi)-\wt f(\xi)|
\\
& \lesssim \ve \,\NN(\nabla u)(\xi)+ |u(x^\pm)- u_j(x^\pm)| + \ve \,\NN(\nabla u_j)(\xi)+ |\wt f_j(\xi)-\wt f(\xi)|.
\end{align*}
Here we used that $u_j \to f_j$ continuously and so $u_j^+(\xi)=u_j^-(\xi)=\wt f_j(\xi)$.
Letting $j\to\infty$, as $u_j(x^\pm)\to u(x^\pm)$, $\NN(\nabla (u_j-u) )(\xi) \to 0$, and  $\wt f_j(\xi)\to \wt f(\xi)$, we obtain
$$|u^\pm(\xi) - \wt f(\xi)| \lesssim \ve \,\NN(\nabla u)(\xi).$$ 
Since $\ve$ is arbitrarily small, we infer that $u^\pm(\xi) = \wt f(\xi)$ and our theorem is concluded as  $u+f(0)$ is the desired solution.
\end{proof}

\vv

Theorem \ref{thm:1.3} is a direct consequence of Theorem \ref{thm:regularity-noncont} and the next lemma.

\begin{lemma}[\bf Uniqueness]\label{lemuniqueness}
Let $\Omega\subset\R^{n+1}$ be an open set with bounded $n$-AD-regular boundary satisfying the weak local John condition.
Let $u:\Omega\to\R$ be a harmonic function, vanishing at $\infty$ when $\Omega$ is unbounded, which has a vanishing non-tangential limit for $\sigma$-a.e.\ $x\in\pom$ and such that $\|\NN(\nabla u)\|_{L^p(\sigma)}<\infty$. Then $u$ vanishes identically in $\Omega$.
\end{lemma}

Recall that the  weak local John condition is defined in Section \ref{subsecwhitney}.

\begin{proof}
Consider the family of Whitney cubes $\WW(\Omega)$ and let $\{\vphi_Q\}_{Q\in\WW(\Omega)}$ be a partition of unity of $\Omega$ so that each $\vphi_Q$ is supported in $1.1Q$ and $\|\nabla\vphi_Q\|_\infty\lesssim \ell(Q)^{-1}$.
For each $\delta\in(0,\diam(\Omega))$, let 
\begin{equation}\label{eqwhitneydelta}
\WW_\delta(\Omega)= \{Q\in \WW(\Omega):\ell(Q)\geq \delta\}
\end{equation}
and
$$\vphi_\delta = \sum_{Q\in \WW_\delta(\Omega)} \vphi_Q.$$
Also, we define
$$u_\delta = \vphi_\delta\,u.$$
Notice that $u_\delta\in C^\infty(\R^{n+1})$ and $\supp u_\delta\subset\Omega$. From the properties
of the Whitney cubes, it easily follows  that there exists some constant $C_9>0$ (depending on the parameters of the construction of the Whitney cubes) such that
$\vphi_\delta(x) =0$ if $\dist(x,\pom)\leq C_9^{-1}\delta$, and 
$\vphi_\delta(x) =1$ if $\dist(x,\pom)\geq C_9\delta$.
Consequently, if we let
$$\WW_0(\Omega)= \big\{Q\in\WW(\Omega):C_{10}^{-1}\delta\leq\ell(Q)\leq C_{10}\delta\big\}$$
for a suitable constant $C_{10}$ depending on $C_9$, we infer that
$$\supp\nabla\vphi_\delta \cup \supp\Delta\vphi_\delta \subset 
\big\{x\in\Omega: C_9^{-1}\delta\leq \dist(x\,\pom)\leq C_9\,\delta \big\}\subset 
\bigcup_{Q\in\WW_0(\Omega)} Q.$$

Fix $x\in\Omega$ and let $\delta\ll\dist(x,\pom)$.
Then, since $u$ is harmonic (and vanishes at $\infty$ when $\Omega$ is unbounded), we have
\begin{align*}
u(x) & = u_\delta(x) = \int G(x,y)\,\Delta u_\delta(y)\,dy\\
& =
2\int G(x,y)\,\nabla u(y)\cdot \nabla \vphi_\delta(y)\,dy + \int G(x,y)\,u(y)\,\Delta \vphi_\delta(y)\,dy =: I_1 + I_2.
\end{align*}
Since $G(x,\cdot)$ is   H\"older continuous at $\pom$, there exists some $a>0$ such that
$$G(x,y)\lesssim \dist(y,\pom)^{a}\quad \mbox{ for $y\in\Omega$ such that $\dist(y,\pom)\leq \frac12\,\dist(x,\pom)$},$$
with the implicit constant depending on $\dist(x,\pom)$.
Then, concerning $I_1$, we have
$$I_1 \lesssim \sum_{Q\in\WW_0} \frac1\delta \int_{Q} G(x,y)\,|\nabla u(y)|\,dy
\lesssim \delta^{a-1} \sum_{Q\in\WW_0}\int_{Q} |\nabla u(y)|\,dy.
$$
Notice now that, for any $Q\in\WW_0$,
$$|\nabla u(y)|\leq \NN(\nabla u)(z)\quad \mbox{ for all $y\in Q$, $z\in b(Q)$,}$$
assuming the operator $\NN$ to be associated to non-tangential regions with big enough aperture.
So we get
\begin{align*}
I_1 & \lesssim  \delta^{a-1} \sum_{Q\in\WW_0}\ell(Q)\,\int_{b(Q)} \NN(\nabla u)(z)\,d\sigma(z)
\\
& \lesssim \delta^{a}\|\NN(\nabla u)\|_{L^1(\sigma)} \leq \delta^{a}\|\NN(\nabla u)\|_{L^p(\sigma)}\,\sigma(\pom)^{1/p'}.
\end{align*}

Next we turn to $I_2$. Using again the H\"older continuity of $G(x,\cdot)$ at $\pom$ we get
$$I_2 \lesssim \sum_{Q\in\WW_0} \frac1{\delta^2} \int_{1.1Q} G(x,y)\,|u(y)|\,dy
\lesssim \delta^{a-2} \sum_{Q\in\WW_0}\int_{Q} |u(y)|\,dy.
$$ 
To bound $u(y)$ for $y\in Q$, observe that, since $\Omega$ 
satisfies the weak local John condition, then there exists
a subset $G_Q\subset C_3Q\cap \pom$ (for some some $C_3>1$ big enough) with 
$\sigma(G_Q)\approx \ell(Q)^n$ such that for each $z\in G_Q$ 
there is a non-tangential path $\gamma_{x_Q,z}$ that joins the center $x_Q$ of $Q$ and $z$ with $\HH^1(\gamma_{x_Q,z})\lesssim |x_Q-z|\approx\delta$. For each $y\in Q$ and $z\in G_Q$ we let $\gamma_{y,z}$ be the union of the path
$\gamma_{x_Q,z}$ and the segment $[x_Q,y]$, so that $\gamma_{y,z}$ is a non-tangential path joining $y$ and $z$ with 
$\HH^1(\gamma_{y,z})\lesssim |y-z|\approx\delta$.

Assuming the operator $\NN$ to be associated to non-tangential regions $\gamma_\alpha$ with big enough aperture, we deduce that $\gamma_{y,z}\setminus\{z\}$ is contained in $\gamma_\alpha(z)$. Further, since the non-tangential limit of $u$ vanishes $\sigma$-a.e., 
we infer that for $\sigma$-a.e.\ $z\in G_Q$ there is some point $z'\in \gamma_{y,z}\setminus \{z\}$ such that $u(z')\leq \delta$.
Integrating along the path $\gamma_{y,z'}$ (this is the subpath consisting of the points from $\gamma_{y,z}$ lying between $y$ and $z'$), we deduce that
\begin{align*}
|u(y)| &\leq |u(z')| + \int_{\xi\in \gamma_{y,z'}}|\nabla u(\xi)|\,d\HH^1(\xi) \lesssim \delta + 
\sup_{\xi\in \gamma_{y,z'}}|\nabla u(\xi)|\,\,\HH^1(\gamma_{y,z})\\ & \lesssim \delta \,(1+ \NN(\nabla u)(z)).
\end{align*}
Hence,
$$u(y) 
 \lesssim \delta + \delta \,\essinf_{z\in G_Q}\NN(\nabla u)(z).$$
Therefore, taking into account that the sets $G_Q$, with $Q\in\WW_0$, have bounded overlaps, 
we obtain
\begin{align*}
I_2 & \lesssim  
\delta^{a-2} \sum_{Q\in\WW_0}\ell(Q)\,\int_{G_Q} (\delta + \delta\,\NN(\nabla u)(z))\,d\sigma(z)
\\
& \lesssim \delta^{a}\,\sigma(\pom) + \delta^{a}\|\NN(\nabla u)\|_{L^1(\sigma)} \leq \delta^{a}\,\sigma(\pom) +\delta^{a}\|\NN(\nabla u)\|_{L^p(\sigma)}\,\sigma(\pom)^{1/p'}.
\end{align*}

Combining all the estimates above, we obtain
$$|u(x)|\lesssim \delta^{a}\,\sigma(\pom) +\delta^{a}\|\NN(\nabla u)\|_{L^p(\sigma)}\,\sigma(\pom)^{1/p'}.$$
Since $\delta$ can be taken arbitrarily small, we infer that $u(x)=0$.
\end{proof}

\vv


\section{Invertibility of the single layer potential operator}\label{sec8*}

We will need the following technical result.

\begin{lemma}\label{lemlimfac}
Let $\Omega\subset\R^{n+1}$ be an open set with bounded $n$-AD-regular boundary satisfying the corkscrew condition, and let $f\in
L^p(\sigma)$, for $p\in [1,\infty]$. Then, for $\sigma$-a.e.\ $x\in\pom$, the non-tangential limit of $\cS f$ at $x$ equals $\cS f(x)$. That is,
\begin{equation}\label{eqlimnt7}
\lim_{\gamma_\alpha(x)\ni y\to x} \cS f(y) = \cS f(x).
\end{equation}
Moreover, for $p \in (1,\infty]$, it holds that 
\begin{equation}\label{eq:Lpboundsingle}
\|\cS f\|_{L^p(\pom)} \lesssim \diam(\pom) \|f\|_{L^{p}(\pom)}.
\end{equation}
\end{lemma}

Although this result is already known, we prove it here for completeness.

\begin{proof}
 If we set $B_k:=B(x,\rho_k):=B(x,2^{-k}\diam(\pom))$, by the estimate
\begin{align*}
|\cS f(x)| &\leq \sum_{k \geq 0} \int_{ B_k\setminus B_{k+1} } \frac{|f(y)|}{|x-y|^{n-1}} \,d\sigma(y)\\
& \approx \sum_{k \geq 0} \frac{1}{{\rho_k}^{n-1}} \int_{B_k} |f(y)|\,d\sigma(y) \lesssim \diam(\pom)\, \mathcal{M}_\sigma(f)(x),
\end{align*}
and the $L^p(\sigma)$ boundedness of $\mathcal{M}_\sigma$, we get  \rf{eq:Lpboundsingle}. Thus,
$$\cS (|f|)(x) <\infty \quad\mbox{ for $\sigma$-a.e.\ $x\in\pom$.}$$

We will show  that \rf{eqlimnt7} holds for all $x\in \pom$ such that $\cS (|f|)(x) <\infty$. Indeed,
given $\ve>0$, let $\delta_1>0$ be such that
$$\int_{B(x,\delta_1)} \frac1{|x-z|^{n-1}}\,d\sigma(z) \leq \ve.$$
Denote $B= B(x,\delta_1)$ and notice now that the function $\cS(f\chi_{B^c})$ is continuous in $\frac12B$. Thus there exists some $\delta_2\in (0,\delta_1/2)$ such that
$$\big|\cS(f\chi_{B^c})(x) - \cS(f\chi_{B^c})(y)\big|\leq \ve\quad \mbox{ if $|x-y|\leq\delta_2.$}$$
Then, for all $y\in\gamma_\alpha(x)$ such that $|x-y|\leq\delta_2$, we have
\begin{align*}
\big|\cS f(x) - \cS f(y)\big| & \leq \cS(|f|\chi_{B})(x) + \cS(|f|\chi_{B})(y)
+
\big|\cS(f\chi_{B^c})(x) - \cS(f\chi_{B^c})(y)\big|\\ &\lesssim \ve + \cS(|f|\chi_{B})(y).
\end{align*}
To estimate $\cS(|f|\chi_{B})(y)$ we use the fact that $|y-z|\gtrsim |x-z|$ for all $z\in\pom$ (because $y\in\gamma_\alpha(x)$), and thus
\begin{align*}
\cS(|f|\chi_{B})(y) &\approx \int_{B} \frac1{|y-z|^{n-1}}\,|f(z)|\,d\sigma(z)
\lesssim \int_{B} \frac1{|x-z|^{n-1}}\,|f(z)|\,d\sigma(z)\\ &\nonumber \approx\cS(|f|\chi_{B})(x)\leq \ve.
\end{align*}
Hence, 
$\big|\cS f(x) - \cS f(y)\big|\lesssim \ve,$
which proves \rf{eqlimnt7} and concludes the proof.
\end{proof}

\vv

Next we need to introduce the  notion of solvability of the regularity problem in unbounded domains with bounded
boundary and extend Theorem \ref{teomain} to this type of domains.
Given an unbounded domain $\Omega \subset \R^{n+1}$ with bounded $n$-AD-regular boundary, we say that the regularity problem $( R_{p})$ is solvable for the Laplacian if  there exists some constant $C_{ R_p}>0$  such that, for any  Lipschitz function $f:\pom\to\R$, the solution $u:\Omega\to\R$ (vanishing at $\infty$) of the continuous
Dirichlet problem for the Laplacian in $\Omega$ with boundary data $f$ satisfies
\begin{equation}\label{eq:main-est-reg**}
\|\NN(\nabla u)\|_{L^p(\sigma)} \leq C_{ R_p}\|f\|_{W^{1,p}(\sigma)}.
\end{equation}
 Notice that, unlike in \rf{eq:main-est-reg} in the case of bounded domains, the estimate \rf{eq:main-est-reg**} involves the inhomogeneous norm
$\|\cdot\|_{W^{1,p}(\sigma)}$. In fact, for nice unbounded domains with bounded boundary the estimate \rf{eq:main-est-reg} fails in general. For example,
in the case that $\Omega=\R^{n+1}\setminus\overline B(0,1)$ and $u(x)=|x|^{1-n}$, we have that $u|_{\partial B(0,1)}=1$ and thus $\| u|_{\partial B(0,1)} \|_{\dot W^{1,p}(\sigma)}=0$ but  $\|\NN(\nabla u)\|_{L^p(\sigma)} =c_n > 0$.

\vv

\begin{theorem}[\bf Solvability in unbounded domains with compact boundary]\label{lemunbounded}
Let $p \in (1,\infty)$ and let $\Omega\subset\R^{n+1}$,  with $n\geq2$, be an unbounded corkscrew domain with bounded $n$-AD-regular boundary 
such that there exists $x_0 \in \R^{n+1} \setminus \overline \Omega$ with $\dist(x_0, \pom) \approx \diam(\pom)$. If $(D_{p'})$ is solvable,
then $(R_p)$ is solvable  with constants independent of $ \diam(\pom)$.  Moreover, for any $f \in  W^{1,p}(\pom)$, there exists a harmonic function $u$ in $\Omega$ such that $\|\NN(\nabla u)\|_{L^p(\sigma)} \leq C \|f\|_{ W^{1,p}(\pom)}$ and $u \to f$ non-tangentially $\sigma$-a.e. on $\pom$. If  $\Omega$ satisfies the weak local John condition, then this solution is unique.
\end{theorem}

\begin{proof}
We will deduce this lemma from Theorem \ref{teomain}, by using the Kelvin transform.
Indeed, by hypothesis, there exists a point $x_0\in \R^{n+1} \setminus \overline \Omega$ such that  $\dist(x_0,\pom) \approx \diam(\pom)$. Without loss of generality we assume that $x_0=0$.

Given $x\in\R^{n+1}\setminus\{0\}$, we let 
$x^*= \frac{\diam(\pom)^2}{|x|^2}\,x.$ 
Notice that the map defined by $I(x)=x^*$ is an involution of $\R^{n+1}\cup\{\infty\}$, understanding that
$I(0)=\infty$.
For $r>0$, we denote $r^*= \frac{\diam(\pom)^2}{r^2}\,r = \frac{\diam(\pom)^2}{r}$.
We also set
$$\Omega^* =
\big\{x^*:x\in\Omega\big\}\cup \{0\}$$ 
(so identifying $\Omega$ with $\Omega\cup\{\infty\}$, we have $\Omega^*=I(\Omega)$).
Given a function $f:\R^{n+1}\supset E\to \R$, its Kelvin transform is defined
by
$$ f^*(x^*)= \frac{\diam(\pom)^{n-1}}{|x^*|^{n-1}}\,f(x),$$
understanding that $f(\infty)=0$.
We let the reader check that, if $u:\Omega\to \R$  vanishes at $\infty$, then $\Delta u=0$ in $\Omega$ if and only if $\Delta (u^*)=0$ in $\Omega^*$. Further, $(u^*)^*=u$.  Now it is not hard to see that if $\| u\|_{L^\infty(\Omega \cap 2B \setminus B)} < \infty$, where $B$ is a ball centered on $\xi_0 \in \pom$ with radius $2 \diam(\pom)$, then  $|u(x)|=O(|x-\xi_0|^{-n+1})$ as $x \to \infty$. Indeed, since $\EE(x-\xi_0)= c_n |x-\xi_0|^{1-n}$ is the fundamental solution for the Laplacian with pole at $\xi_0$, then  we may choose a constant $\kappa_0 \approx \| u\|_{L^\infty(\Omega \cap (2B \setminus B) )} \diam(\pom)^{n-1}$, so that $|u(x)| \leq v(x):= \kappa_0 \, \EE(x-\xi_0)$ for any $ x \in \Omega \cap (2B \setminus B)$. Then, since both $u$ and $v$ vanish at infinity, we may apply the maximum principle and deduce that $|u(x)| \lesssim |x-\xi_0|^{1-n}$ for any $x \in \Omega \setminus 2B$. Therefore,  $|u(x)| \lesssim |x|^{1-n}$ for any $x \in \Omega \setminus B(0,  3 \diam(\pom))$ and so $|u^*(0)| \lesssim_n\| u\|_{L^\infty(\Omega \cap (2B \setminus B) )} $.

Let $\rho_0= M \diam(\pom)$ for some $M \geq 4 $ be such that
$$  \pom \subset B(0,\rho_0) \subset B(0,2\rho_0) \quad \textup{and}\quad B(0,(2\rho_0)^*) \subset B(0,\rho_0^*) \subset \R^{n+1} \setminus \overline \Omega.$$
To shorten notation, we write $r_0:= (2\rho_0)^*$ and $R_0:=2\rho_0$.
It is immediate to check that the involution $I$ is bilipschitz
in the annulus $A(0,r_0,R_0)$  with uniform constants, and in particular on $\pom$. From this fact it easily follows that
$\pom^*$ is $n$-AD-regular  with uniform bounds since $I(\pom)=\pom^*$ and $\HH^n(E) \approx \HH^n(I^{-1}(E))$ for any $E \subset \pom^*$.  Moreover,
for any function $f:\pom\to\R$,
$$\|f\|_{L^p(\HH^n|_\pom)} \approx \|f^*\|_{L^p(\HH^n|_\pom*)},$$
with constants independent of $\diam(\pom)$.
Further, the involution $I$ transforms a non-tangential (unbounded) region for $\Omega$ into a non-tangential region for $\Omega^*$ containing some ball centered at the origin.

Next we will check that the solvability of $(D_{p'})$ for $\Omega$ implies the solvability of $(D_{p'})$ for $\Omega^*$. To this end, consider $g\in C(\Omega^*)$ and the solution $v:{\Omega^*}\to\R$ of the
Dirichlet problem for $ \Omega^*$ with boundary data $g$. We assume the non-tangential regions
$\gamma_{\Omega^*}(\cdot)$ to have a big enough aperture so that $\overline{B(0,r_0)}\subset \gamma_{\Omega^*}(x^*)$ for every $x^*\in\pom^*$. In this way, by the maximum principle,
\begin{align*}
\NN_{\Omega^*}v(x^*) = \!\!\sup_{y\in\gamma_{\Omega^*}(x^*)} \!\!|v(y)| = \!\!\sup_{y\in\gamma_{\Omega^*}(x^*)\setminus B(0,r_0)}\!\!|v(y)|
= \!\!\sup_{y\in I(\gamma_{\Omega^*}(x^*)\setminus B(0,r_0))}|v(y^*)|
\end{align*}
By the discussion above, assuming the aperture of $\gamma_\Omega(\cdot)$ big enough, we have
$$I(\gamma_{\Omega^*}(x^*)\setminus B(0,r_0))\subset\gamma_\Omega(x)\cap B(0,R_0)\subset A(0,r_0,R_0).$$
Hence,
\begin{align*}
\NN_{\Omega^*}v(x^*) &\leq \sup_{y\in \gamma_\Omega(x) \cap  B(0,R_0)}|v(y^*)|\approx 
\sup_{y\in \gamma_\Omega(x)\cap B(0,R_0)}\frac{ \diam(\pom)^{n-1}}{|y|^{n-1}}\,|v(y^*)| \\ &\leq \NN_{\Omega}(v^*)(x).
\end{align*}
It is clear that $u:=v^*$ is the solution of the Dirichlet problem in $\Omega$ with boundary data $f:=g^*$, and 
then
\begin{align}
\label{eqak590}
\|\NN_{\Omega^*}v\|_{L^p(\HH^n|_{\pom^*})}^p & \lesssim \int_{\pom^*}\NN_{\Omega}u(x)^p\,
d\HH^n(x^*)\\
&\approx \int_{\pom}\NN_{\Omega}u(x)^p\,d\sigma(x) \lesssim \|f\|_{L^p(\sigma)}^p \approx
\|g\|_{L^p(\HH^n|_{\pom^*})}^p,\nonumber
\end{align}
where we took into account that the image measure $I_\#(\HH^n|_{\pom^*})$ is comparable to
$\sigma \equiv\HH^n|_{\pom}$, because $I$ is bilipschitz from $\pom^*$ to $\pom$. 
So the solvability of $(D_{p'})$ for $\Omega^*$ holds, as wished.

Since $\Omega^*$ is bounded, from Theorem \ref{teomain} we deduce that $(R_p)$ is solvable for $\Omega^*$. We will transfer this solvability to $\Omega$ by the Kelvin transform again.
So let $f:\pom\to\R$ be a Lipschitz function and let $u:\Omega\to\R$ be the solution (vanishing at $\infty$) of the continuous
Dirichlet problem in $\Omega$ with boundary data $f$.
We assume the non-tangential regions
$\gamma_{\Omega}(\cdot)$ to have a big enough aperture so that $\overline{B(0,R_0)}\subset \gamma_{\Omega}(x)$ for every $x\in\pom$. In this way, by the maximum principle, arguing as above, we obtain
\begin{align}\label{eqak591}
\NN_{\Omega}(\nabla u)(x) & = \sup_{y\in\gamma_{\Omega}(x)}|\nabla u(y)| = \sup_{y\in\gamma_{\Omega}(x)\setminus B(0,R_0)}|\nabla u(y)|\\
& = \sup_{z\in I(\gamma_{\Omega}(x)\setminus B(0,R_0))}|\nabla u(z^*)|
 \leq \sup_{z\in \gamma_{\Omega^*}(x^*)\setminus B(0,r_0)}|\nabla u(z^*)|.\nonumber
\end{align}
Observe now that $v:=u^*$ is the solution of the Dirichlet problem for $\Omega^*$ with boundary data 
$g:=f^*$, and we have that for $y \in \Omega^*$,
 \begin{equation}\label{eqak86}
 \nabla v(y) = \nabla \Big(\frac{\diam(\pom)^{n-1}}{|y|^{n-1}}\Big)\, u(y^*) + \frac{ \diam(\pom)^{n-1}}{|y|^{n-1}}\,\nabla u(y^*)\,DI^{-1}(y),
 \end{equation}
where $DI^{-1}$ is the Jacobian matrix of the map $I^{-1}$. From this identity, the fact that $I(\gamma_{\Omega^*}(x^*)\setminus B(0,r_0))\subset A(0,r_0,R_0)$,  that $I$ is bilipschitz in $A(0,r_0,R_0)$, and $DI^{-1}(z)=(DI(I^{-1}(z)))^{-1}$, we deduce that for all $z\in \gamma_{\Omega^*}(x^*)\setminus B(0,r_0)$,
 \begin{equation}\label{eqak592}
|\nabla u(z^*)|\lesssim |\nabla v(z)| + \diam(\pom)^{-1}|u(z^*)|\lesssim |\nabla v(z)| + \diam(\pom)^{-1} |v(z)|.
\end{equation}

To estimate $v(z)$, we take into account that the $(D_{p'})$ solvability in $\Omega^*$ implies that
$\Omega^*$ satisfies the weak local John condition, by \cite{AHMMT} (recall the  weak local John condition was defined in \rf{eqwlj11}). This condition ensures that there exists $C>1$ and
 $G_z\subset\pom^*\cap B(z,C\dist(z,\pom^*))$ with $\HH^n(G_z)\approx\dist(z,\pom^*)^n$
such that all $w \in G_z$ can be joined to $z$ by 
 a non-tangential path $\gamma_{z,w}$ satisfying $\HH^1(\gamma_{z,w})\lesssim |z-w|$.
Integrating along $\gamma_{z,w}$ and assuming the aperture of the cone $\gamma_{\Omega^*}(w)$ big enough, we obtain
\begin{equation*} 
|v(z)|\lesssim \diam(\pom) \!\sup_{y\in \gamma_{\Omega^*}(w)} \!|\nabla v(y)| + |g(w)| = \diam(\pom) \,\NN_{\Omega^*}(\nabla v)(w) +  |g(w)|
\end{equation*}
for all $w\in G_z$.
Averaging over $w\in G_z$, and taking into account that  
$$G_z\subset\pom^*\cap B(z,C\dist(z,\pom^*))\subset\pom^*\cap B(x^*,C'|x^*-z|)$$
and that $\HH^n(G_z)\approx |x^*-z|^n$,
we infer that
$$ \diam(\pom)^{-1}|v(z)| \lesssim  \cM_{\HH^n|_{\pom^*}} (\NN_{\Omega^*}(\nabla v))(x^*) +\diam(\pom)^{-1}  \cM_{\HH^n|_{\pom^*}}  g(x^*).$$
Combining this estimate with \rf{eqak591} and \rf{eqak592},  for $\HH^n$-a.e.\ $x^*\in\pom^*$ we get
\begin{align*}
\NN_{\Omega}(\nabla u)(x) & \lesssim \NN_{\Omega^*}(\nabla v)(x^*) + 
\cM_{\HH^n|_{\pom^*}} (\NN_{\Omega^*}(\nabla v))(x^*) \\ &\quad+ \diam(\pom)^{-1}\cM_{\HH^n|_{\pom^*}}  g(x^*)\\
& \lesssim 
 \cM_{\HH^n|_{\pom^*}} (\NN_{\Omega^*}(\nabla v))(x^*) + \diam(\pom)^{-1}\cM_{\HH^n|_{\pom^*}}  g(x^*).
\end{align*}

From the preceding inequality, arguing as in \rf{eqak590} and using the $L^p(\HH^n|_{\pom^*})$ boundedness of $\cM_{\HH^n|_{\pom^*}}$ and the 
solvability of $(R_p)$ for $\Omega^*$, we deduce that
\begin{align*}
\|\NN_{\Omega}(\nabla u) \|_{L^p(\sigma)}  &\lesssim    \|\NN_{\Omega^*}(\nabla v) \|_{L^p(\HH^n|_{\pom^*})} +  \diam(\pom)^{-1} \|g \|_{L^p(\HH^n|_{\pom^*})}\\
& \lesssim   \|\nabla_{H,p} g \|_{L^p(\HH^n|_{\pom^*})} + \diam(\pom)^{-1} \|f \|_{L^p(\sigma)}.
\end{align*}
To conclude the proof of the lemma (with constants independent of $\diam(\pom)$) it just remains to check
that
\begin{equation}\label{claim7245}
 \|\nabla_{H,p} g \|_{L^p(\HH^n|_{\pom^*})}\lesssim  \frac1{\diam(\Omega)} \|f\|_{L^p(\sigma)} + \|\nabla_{H,p} f \|_{L^p(\sigma)}.
\end{equation}
To this end, we consider arbitrary points $x^*,y^*\in\pom^*$, and we set
\begin{align*}
|g(x^*)& - g(y^*)|  = |f^*(x^*) - f^*(y^*)| = \bigg| \frac{\diam(\Omega)^{n-1}}{|x^*|^{n-1}} \,f(x) -  \frac{\diam(\Omega)^{n-1}}{|y^*|^{n-1}} \,f(y)\bigg|\\
& \leq 
 \bigg| \frac{\diam(\Omega)^{n-1}}{|x^*|^{n-1}} -  \frac{\diam(\Omega)^{n-1}}{|y^*|^{n-1}} \bigg|\,|f(x)| + 
 \frac{\diam(\Omega)^{n-1}}{|y^*|^{n-1}} \,|f(x) - f(y)| \\
 & \lesssim \frac{|x^*-y^*|}{\diam(\Omega)} \,|f(x)| + |f(x) - f(y)|.
\end{align*}
By the definition of $\nabla_{H,p} f$ and the bilipschitzness of $I$, we have
\begin{align*}
|f(x) - f(y)| &\leq \big(|\nabla_{H,p}f(x)| + |\nabla_{H,p}f(y)|\big) \,|x-y|&\\ &\approx   \big(|\nabla_{H,p}f(x)| + |\nabla_{H,p}f(y)|\big) \,|x^*-y^*|.
\end{align*}
Hence, the function
$$G(x^*) := C\bigg( \frac{|f(x)|}{\diam(\Omega)}  + |\nabla_{H,p}f(x)|\bigg)$$
is a Haj\l asz upper gradient of $g$ and so
$$
 \|\nabla_{H,p} g \|_{L^p(\HH^n|_{\pom^*})}\lesssim \frac1{\diam(\Omega)} \|f\|_{L^p(\sigma)} + \|\nabla_{H,p} f \|_{L^p(\sigma)},
$$
which proves the claim \rf{claim7245}.

It only remains to prove the last part of the lemma. To this end, fix $f \in W^{1,p}(\pom)$  and let $g=f^*$. Then, since $(R_p)$ is solvable in $\Omega^*$, for the  function $g \in W^{1,p}(\pom^*)$ we may apply Theorem \ref{thm:1.4} and find a harmonic function $v:\Omega^* \to \R$ such that $ \|\NN_{\Omega^*}(\nabla v) \|_{L^p(\HH^n|_{\pom^*})}  \lesssim  \|\nabla_{H,p} g \|_{L^p(\HH^n|_{\pom^*})}$ and $ v \to g$ for $\HH^n|_{\pom^*}$-a.e. $x \in \pom^*$. If we set $u:=v^*$ then it is clear that $u$ is harmonic in $\Omega$ and by similar considerations as above, we can show that $u \to f$ n.t. in $\Omega$ and $\|\NN_{\Omega}(\nabla u) \|_{L^p(\sigma)}  \lesssim \|\nabla_{H,p} f \|_{L^p(\sigma)}$. Uniqueness follows from Lemma \ref{lemuniqueness} (which is still true in this case). We leave the details for the interested reader.
\end{proof}


\vv

\begin{remark}[\bf One-sided Rellich inequality]\label{rem:8.3}
Under the assumptions of Theorem \ref{lemunbounded}, given any function $f\in
{\rm Lip}(\pom)$  and the solution $u$ of the Dirichlet problem with boundary data $f$, 
we deduce that $\partial_\nu u$ exists in the weak sense (see \rf{eqnormal}),
 it belongs to $L^p(\sigma)$, and it satisfies
\begin{equation}\label{eqnutang}
\|\partial_\nu u\|_{L^p(\sigma)} \lesssim  \|f\|_{ W^{1,p}(\pom)}.
\end{equation}
To prove this, notice first that  as $u^*$ is the solution to the Dirichlet problem with data $f^*$ in $\Omega^*$, it holds that $u^* \in W^{1,2}(\Omega^*)$, and so by the bound $|\nabla u(z) |\lesssim |z|^{-n}$ for $z \in \Omega \setminus B(0, 3 \diam(\pom))$ (since $ u$ is harmonic in $\Omega$ and vanishes at $\infty$, we have shown in the proof of Theorem \ref{lemunbounded} that $|u(z)| \lesssim |z|^{-n+1}$ for such $z$), the equation \rf{eqak86}, and arguing as in the proof of \rf{eqak592},  it also holds that $u\in  \dot W^{1,2}(\Omega)$. 
Then,
applying the Riesz representation theorem and arguing as in the proof of Lemma \ref{keylemma},
it suffices to check that, for any $\vphi\in C_c^\infty(\R^{n+1})$,
\begin{equation}\label{eqnab**}
\left|\int_\Omega \nabla u\cdot \nabla\vphi \,dm\right| \lesssim  \|f\|_{ W^{1,p}(\pom)}\,\|\vphi\|_{L^{p'}(\sigma)}.
\end{equation}
To this end, for any $\delta\in(0,\diam(\pom)/2)$, consider the family $\WW_\delta(\Omega)$ of Whitney cubes for $\Omega$ with side length
at least $\delta$, as in \rf{eqwhitneydelta}, and let
$\Omega_\delta$ be the interior of $\bigcup_{Q\in \WW_\delta(\Omega)} Q$. Since the boundary of
$\Omega_\delta$ is made up of finitely many faces of cubes with side length comparable to $\delta$  and $u$ is smooth in a neighborhood of $\overline \Omega_\delta$,
we can apply Green's formula in $\Omega_\delta$ (which is a set of finite perimeter since it is clear that $\HH^n(\partial^* \om_\delta)<\infty$)  to deduce that  
$$\int_{\Omega_\delta} \nabla u\cdot \nabla\vphi \,dm 
= \int_{\partial\Omega_\delta} \partial_{\nu_\delta} u\,\vphi\,d\HH^n,
$$
where $\d_{\nu_\delta} u$ stands for the non-tangential trace  of $\nu_\delta \cdot \nabla u$ on $\pom_\delta$.

Let $I_\delta$ be the subfamily of the cubes from $\WW_\delta(\Omega)$ whose closure intersects $\partial\Omega_\delta$ and denote by $\omega_\vphi(\cdot)$ the modulus of continuity of $\vphi$. Then we have
\begin{align*}
\left|\int_{\Omega_\delta} \nabla u\cdot \nabla\vphi \,dm \right|
& \leq \sum_{Q\in I_\delta} \int_{\overline Q\cap \partial\Omega_\delta} |\nabla u|\,|\vphi|\,d\HH^n\\
& \lesssim \sum_{Q\in I_\delta} \inf_{y\in b(Q)} \NN(\nabla u)(y)\,
\inf_{y\in b(Q)}\big( |\vphi(y)| + \omega_\vphi(C\delta)\big)
\,\ell(Q)^n.
\end{align*}
By H\"older's inequality, we get
\begin{align*}
\biggl|\int_{\Omega_\delta}  \nabla u\cdot \nabla\vphi \,dm \biggr| &\lesssim \Bigl(\sum_{Q\in I_\delta} \inf_{y\in b(Q)} \NN(\nabla u)(y)^p\,\,\ell(Q)^n\Bigr)^{1/p}\\ &\quad
\Bigl(\sum_{Q\in I_\delta}\inf_{y\in b(Q)}\big( |\vphi(y)| + \omega_\vphi(C\delta)\big)^{p'}
\,\ell(Q)^n\Bigr)^{1/p'}\\
&\lesssim \left(\int_\pom \NN(\nabla u)^p\,d\sigma\right)^{1/p}
\left(\int\big(|\vphi| + \omega_\vphi(C\delta)\big)^{p'}\,d\sigma\right)^{1/p'}.
\end{align*}
By the $(R_p)$ solvability, we deduce
\begin{align*}
\biggl|\int_{\Omega_\delta} & \nabla u\cdot \nabla\vphi \,dm \biggr| \lesssim  \|f\|_{ W^{1,p}(\pom)}
\,\big(\|\vphi\|_{L^{p'}(\sigma)}+ \omega_\vphi(C\delta)\big).
\end{align*}
Letting $\delta\to 0$ and applying the the dominated convergence theorem (recall that
$\nabla u\in L^2(\Omega)$), the estimate \rf{eqnab**} follows.
\end{remark}

\vv

\begin{lemma}\label{lem6.1**}
Let $p \in (1,\infty)$ and let $\Omega\subset\R^{n+1}$ be an unbounded corkscrew domain with bounded $n$-AD-regular boundary such that 
 there exists $x_0 \in \R^{n+1} \setminus \overline \Omega$ with $\dist(x_0, \pom) \approx \diam(\pom)$. If $(D_{p'})$ is solvable, given any function $f\in \Lip(\pom)$, let $u$ denote the solution of the Dirichlet problem with boundary data $f$.
Then $\partial_\nu u$ exists in the weak sense (see \rf{eqnormal}), it belongs to $L^p(\sigma)$, and we have
$$u(x) = \DD f(x) + \cS (\partial_\nu u)(x) \quad \mbox{ for all $x\in\Omega$.}$$
\end{lemma}

\begin{proof}
Let $r>0$ be such that $\pom\subset B(0,r/2)$, and denote $\Omega_r= \Omega \cap B(0,r)$, so that $\pom_r=\pom\cap\partial B(0,r)$.
It is clear that $\Omega_r$ is bounded corkscrew domain with $n$-AD-regular boundary (with constant depending on $r$). 
Further, from the solvability of $(D_{p'})$ in $\Omega$, one easily deduces the
solvability of $(D_{p'})$ in $\Omega_r$. This can be proved by using Theorem \ref{teo9.2} below and the maximum principle.

From Lemma \ref{keylemma} applied to the boundary function $f_r$ equal to $f$ in $\pom$ and to $u$ in $\partial B(0,r)$, it follows that $\partial_\nu u|_{\pom_r}$ exists in the weak sense, and it belongs to $L^{p}(\HH^n|_{\pom_r})$. Also, 
by Lemma \ref{lem6.1} we have
$$u(x)=\DD_{\Omega_r}(u|_{\pom_r})(x) - \cS_{\Omega_r}(\partial_\nu u|_{\pom_r})(x) \quad \mbox{ for all $x\in\Omega_r$,}$$
where $\DD_{\Omega_r}$ and $\cS_{\Omega_r}$ denote the double and single layer potentials for $\Omega_r$, respectively.

It is clear that the restriction of $\partial_\nu u|_{\pom_r}$ to $\pom$ coincides $\partial_\nu u|_{\pom}$ and it is independent of
$r$. So $\partial_\nu u|_{\pom}\in L^{p}(\sigma)$.
We claim that, for any $x\in\Omega$,
\begin{equation}\label{eqclaim287}
\lim_{r\to\infty} \DD_{\Omega_r}(u|_{\pom_r})(x) = \DD(u|_{\pom})(x), \,\,\quad \lim_{r\to\infty} \cS_{\Omega_r}(\partial_\nu u|_{\pom_r})(x) = \cS(\partial_\nu u|_{\pom})(x),
\end{equation}
which would prove the assertion (a). To show the first identity, we have to check that
$$\lim_{r\to 0} \int_{\partial B(0,r)} \nu(y) \cdot \nabla_y \EE(x-y)\, u(y) \,d\HH^n(y)=0.$$
Since $u$ vanishes at $\infty$, we have $|u(y)|\lesssim |y|^{1-n}$ for all $y$ far away from $\pom$, with the implicit constant depending on $u$. Thus,
\begin{align*}
\left|\int_{\partial B(0,r)} \!\nu(y) \cdot \nabla_y \EE(x-y)\, u(y) \,d\HH^n(y)\right| & \lesssim_u
\int_{\partial B(0,r)} \frac1{|x-y|^n\,r^{n-1}} \,d\HH^n(y) \\ &\lesssim \frac{r^{n}}{\big||x|-r\big|^n\,r^{n-1}},
\end{align*}
which tends to $0$ as $r\to\infty$. An analogous estimate which we leave for the reader proves the second identity in \rf{eqclaim287}.
\end{proof}

\vv
Next we are ready to prove Theorem \ref{teo-invert}. For the reader's convenience we announce it here again as a lemma. 

\vv

\begin{lemma}[\bf Invertibility of layer potentials] \label{lem-invert}
Let $p \in  (1,\infty)$ and let $\Omega\subset\R^{n+1}$ be a bounded two-sided corkscrew domain with $n$-AD-regular boundary such that $\R^{n+1} \setminus \overline{\Omega}$ is connected.
Suppose either that $\Omega$
satisfies the two-sided 
local John condition or that $\pom$ supports a weak $(1,p)$-Poincar\'e inequality.
If $(D_{p'})$ is solvable both for 
 $\Omega$ and for $\R^{n+1}\setminus \overline\Omega$, then
$\cS: L^p(\pom)\to  W^{1,p}(\pom)$ is bounded and invertible.
\end{lemma}

\begin{proof}
The fact that $\cS$ is bounded from $L^p(\pom)$ to $ W^{1,p}(\pom)$ follows from \rf{eq:Lpboundsingle} and the fact that
$\pom$ is uniformly rectifiable, which in turn implies that the $n$-dimensional Riesz transform
is bounded in $L^p(\sigma)$. See also  \cite[Corollary 3.28]{HMT} for more details.

To prove the invertibility of $\cS$
notice that, by Theorems \ref{teomain} and 
\ref{lemunbounded}, it holds that $(\wt R_p)$ is solvable both for 
 $\Omega$ and for $(\overline\Omega)^c$. Further, by \cite{AHMMT}, both $\Omega$ and $(\overline\Omega)^c$  satisfy the weak local John condition.
 To simplify notation, we will also write $\Omega^+:=\Omega$ and $\Omega^- = (\overline\Omega)^c$.

First we will show that
\begin{equation}\label{eqinv22}
\|f\|_{L^p(\sigma)}\leq C\,\|\cS f\|_{ W^{1,p}(\pom)}\quad \mbox{ for all $f\in L^p(\sigma)$.}
\end{equation}
From the $n$-rectifiability of $\pom$ and the jump relations for the Riesz transform (see \cite{Tolsa-jumps}, for example) we deduce
that the non-tangential limits
$$
\partial_{\nu,\pm} \cS f(x) := \lim_{\gamma_\alpha^{\pm}\ni y\to x} \nabla \cS f(y)\cdot \nu(x)
$$
exist for $\sigma$-a.e.\ $x\in\pom$ and moreover
\begin{equation}\label{eqinv23}
f(x) = \partial_{\nu,+} \cS f(x) - \partial_{\nu,-} \cS f(x).
\end{equation}
Remark that in this identity $\partial_{\nu,+} \cS f(x)$ and $\partial_{\nu,-} \cS f(x)$ should be understood as non-tangential limits. We claim that
$$\|\partial_{\nu,\pm} \cS f(x)\|_{L^p(\sigma)} \lesssim \|\cS f\|_{ W^{1,p}(\pom)}.$$
Observe that \rf{eqinv22} follows from this claim and \rf{eqinv23}.
To prove the claim, recall that by Theorems \ref{thm:regularity-noncont} and  \ref{lemunbounded}, there are functions $u^+$, $u^-$ harmonic in $\Omega^+$, $\Omega^-$, respectively, such that 
$$\|\NN_{\Omega^+}(\nabla u^+ )\|_{L^p(\sigma)} + \|\NN_{\Omega^-}(\nabla u^- )\|_{L^p(\sigma)}\leq C \|\cS f\|_{ W^{1,p}(\pom)}<\infty$$
and
$$\lim_{\gamma_\alpha^{\pm}\ni y\to x} u^\pm (y) = \cS f(x)  \quad \mbox{ for $\sigma$-a.e.\ $x\in\pom$.}
$$

From the $L^p(\sigma)$ boundedness of the maximal Riesz transform and standard Calder\'on-Zygmund estimates, we also have
$$\|\NN_{\Omega^+}(\nabla \cS f)\|_{L^p(\sigma)} + \|\NN_{\Omega^-}(\nabla \cS f)\|_{L^p(\sigma)}\leq C 
\|f\|_{L^p(\sigma)}<\infty$$
and, by Lemma \ref{lemlimfac},
$$\lim_{\gamma_\alpha^{\pm}\ni y\to x} \cS f (y) = \cS f(x)\quad \mbox{ for $\sigma$-a.e.\ $x\in\pom$.}
$$
Consequently, since the harmonic functions $w^\pm:=\cS f -u^\pm$ vanishes n.t.\! at $\sigma$-a.e. point on $\pom$ and satisfies $\|\NN_{\Omega^\pm}(\nabla w^\pm)\|_{L^p(\sigma)}<\infty$, we may apply Lemma \ref{lemuniqueness} and  Theorem \ref{lemunbounded} and  infer that $\cS f= u^\pm $ in $\Omega^\pm$.
Therefore, 
$$\NN_{\Omega^\pm}(\nabla \cS f) = \NN_{\Omega^\pm}(\nabla u^\pm)\quad\text{ in }\, L^p(\sigma) ,$$
and so
\begin{align*}
\|\partial_{\nu,\pm} \cS f(x)\|_{L^p(\sigma)}&\leq \|\NN_{\Omega^\pm}(\nabla \cS f)\|_{L^p(\sigma)}\\
&  =\|\NN_{\Omega^\pm}(\nabla u^\pm)\|_{L^p(\sigma)} \leq C \|\cS f\|_{ W^{1,p}(\pom)}<\infty,
\end{align*}
which proves our claim and thus \rf{eqinv22}, which implies  that $\cS$ is injective and has closed range. 

To complete the proof of the invertibility of $\cS$ it suffices to show that its range is dense in 
$ W^{1,p}(\pom)$.
To this end, we will show that ${\rm Lip}(\pom)\subset \cS(L^p(\sigma))$ (recall that ${\rm Lip}(\pom)
$ is dense in the space $ W^{1,p}(\pom)$ \cite[Theorem 5]{Hajlasz}). Given $f\in {\rm Lip}(\pom)$, we set $v^\pm$ to be the solution of the
(continuous) Dirichlet problem in $\Omega^\pm$ with boundary data $f$.
By the solvability of $(\wt R_p)$ in $\Omega^\pm$ and 
Lemmas \ref{lem6.1} and \ref{lem6.1**}, we have
$$
v^\pm(x)= \pm \DD (f)(x) \mp \cS(\partial_{\nu,\pm} v^\pm|_{\pom})(x) \quad \mbox{ for all $x\in\Omega^\pm$,}$$
where $\partial_{\nu,\pm} v^\pm|_\pom\in L^p(\sigma)$ should be understood in the weak sense. Taking non-tangential limits in $\Omega^\pm$, we obtain
$$f(x) = \pm\DD_\pm (f)(x) \mp \cS(\partial_{\nu,\pm} v^\pm|_{\pom})(x) \quad \mbox{ for $\sigma$-a.e.\ $x\in\pom$.}$$
Since $\DD_+ (f)(x) - \DD_- (f)(x) = f(x)$ for $\sigma$-a.e.\ $x$, summing both identities we get
$$f(x) = \cS(\partial_{\nu,-} v^-|_{\pom} - \partial_{\nu,+} v^+|_{\pom})(x) \quad \mbox{ for $\sigma$-a.e.\ $x\in\pom$,}$$
which implies that $f$ belongs to $\cS(L^p(\sigma))$. Therefore, since $S$ has closed range
and ${\rm Lip}(\pom)$ is dense in $W^{1,p}(\pom)$, it holds that  $W^{1,p}(\pom) \subset\cS(L^p(\sigma))$, which  finishes the proof of the lemma.
\end{proof}
\vv


\section{From the regularity problem to the Dirichlet problem}\label{secconverse}

In this section we will prove Theorem \ref{propoconverse}. That is, given a bounded domain 
 $\Omega\subset\R^{n+1}$ with $n$-AD-regular boundary, we suppose the regularity problem
for the Laplacian is solvable in $L^{p}$ for some and $p\in (1,\infty)$, and then we have to show that the Dirichlet problem is solvable in $L^{p'}$.

We need the following result.

\begin{lemma}\label{lemrev}
Let $\Omega\subset\R^{n+1}$ be a  domain with bounded $n$-AD-regular boundary. Given $x \in\Omega$, denote by $\omega^x$ the
harmonic measure for $\Omega$ with pole at $x$. Suppose that $\omega^x$ is 
absolutely continuous with respect to surface measure for every $x$.
Let $p\in (1,\infty)$ and $\Lambda>1$ and suppose that, for every ball $B$ centered  at $\pom$ with $\diam(B)\leq 2\diam(\Omega)$ and all $x \in \Lambda B$ such that $\dist(x,\pom)\geq \Lambda^{-1}r(B)$, it holds
\begin{equation}\label{eqrever}
\left(\avint_{\Lambda B} \left(\frac{d\omega^{x}}{d\sigma}\right)^p\,d\sigma\right)^{1/p} \leq \kappa\,\sigma(B)^{-1},
\end{equation}
for some $\kappa>0$.
Then, if $\Lambda$ is big enough, the Dirichlet problem is solvable in $L^{s}$, for $s>p'$. 
Further, for all $f\in L^{p'}(\sigma)\cap C(\pom)$, its harmonic extension $u$ to $\Omega$ satisfies
\begin{equation}\label{eqrever2*}
\|\NN(u)\|_{L^{p',\infty}(\sigma)}\lesssim \kappa\,\|f\|_{L^{p'}(\sigma)}.
\end{equation}
\end{lemma}



\begin{proof}
Let $f\in C(\pom)$ and let $u$ the solution of the Dirichlet problem in $\Omega$ with boundary data $f$.
Suppose that $f\geq0$.
Consider a point $\xi\in\pom$ and a non-tangential cone $\gamma(\xi)\subset\Omega$, with vertex $\xi$ and  with a fixed aperture.  Fix a point $x\in\gamma(\xi)$ such denote $d_x = \dist(x,\pom)$.
We intend to estimate $u(x)$, first assuming $d_x\leq 2\diam(\pom)$.

To this end,  we pick a smooth function $\vphi$ which equals $1$ in $B(0,1)$ and vanishes in $\R^{n+1}\setminus B(0,2)$. For some $M>4$ to be chosen later, we denote
$$\vphi_M(y) =\vphi\Big(\frac{y}{M d_x}\Big).$$
We set 
$$f_0 (y) = f(y)\, \vphi_M(y-\xi),\qquad\quad f_1(y) = f(y)-f_0(y),$$
and we denote by $u_0$ and $u_1$ the corresponding solutions of the associated Dirichlet problems so that $u=u_0 + u_1$.

To estimate $u_0(x)$ we use \rf{eqrever} to show that
\begin{align*}
u_0(x) & = \int f_0\,d\omega^x \leq 
\int_{B(\xi,2Md_x)} f\,\frac{d\omega^x}{d\sigma}\,d\sigma\\
& \leq 
\left(\int_{B(\xi,2Md_x)} |f|^{p'}\,d\sigma\right)^{1/p'}
\left(\int_{B(\xi,2Md_x)}
\left(\frac{d\omega^x}{d\sigma}\right)^p\,d\sigma\right)^{1/p}\\
& \leq \kappa\,C(M)\,\cM_{\sigma,p'}f(\xi)\,\frac{\sigma(B(\xi,2Md_x))^{1/p'}}{\sigma(B(\xi,d_x))^{1/p'}}\lesssim \kappa\,C(M)\,\cM_{\sigma,p'}f(\xi),
\end{align*}
assuming $\Lambda\geq 2M$.

To deal with $u_1(x)$, we first estimate  $\avint_{B(\xi,M d_x)}u_1\,dm$. To do so, we consider the 
splitting of $\Omega$ into the usual family of Whitney cubes and we denote by $I_B$ the family of those
cubes that intersect $B:=B(\xi,Md_x)$. By the properties of $\WW(\Omega)$, the cubes $P\in I_B$ are contained
in $CB:=B(\xi,CMd_x)$, for some $C$ depending just on $n$ and the parameters in the construction of $\WW(\Omega)$. Then, taking into account that
$u_1\leq u$, we have
\begin{align}\label{eqNN99}
\int_{B(\xi,Md_x)} u_1\,dm & \leq \sum_{P\in I_B} \int_P  u\,dm \leq 
\sum_{P\in I_B} \inf_{y\in b(P)}\NN u(y)\,\ell(P)^{n+1} \\
& \lesssim
\sum_{Q\in\DD_\sigma:Q\subset C'B} \ell(Q)\,\int_Q \NN u\,d\sigma\lesssim Md_x
\int_{C'B} \NN u\,d\sigma,\nonumber
\end{align}
where in the second inequality we took into account that $d_x\leq2\diam(\pom)$.
So we deduce
$$\avint_{B(\xi,Md_x)} u_1\,dm\lesssim \avint_{C'B} \NN u\,d\sigma\lesssim \cM_\sigma(\NN u)(\xi).$$
Now, taking into account that $f_1$ vanishes in $B(\xi,Md_x)$, from the H\"older continuity 
of $u_1$ in $\pom\cap B(\xi.Md_x/2)$, we infer that
$$u_1(x) \lesssim \frac1{M^\alpha}\,\avint_{B(\xi,Md_x)} u_1\,dm\lesssim \frac1{M^\alpha}\,\cM_\sigma(\NN u)(\xi),$$
for some $\alpha>0$ depending just on the AD-regularity constant of $\pom$.

Altogether, we have
\begin{equation}\label{eqnab84}
u(x) \leq \kappa\,C(M)\,\cM_{\sigma,p'}f(\xi) + \frac C{M^\alpha}\,\cM_\sigma(\NN u)(\xi)
\end{equation}
for all $x\in\gamma(\xi)$
with $d_x\leq 2\diam(\pom)$.
 In case that $\Omega$ is unbounded, it turns out that the closure of $A:=\{x\in\Omega: d_x > 2\diam(\pom)\}$ is contained in the cone $\gamma(\xi)$ if the aperture of
$\gamma(\xi)$ is assumed to be big enough. Thus, by the maximum principle, since \rf{eqnab84} holds for $x\in\partial A$ and $u$ vanishes at $\infty$,
it follows that the same estimate is also valid for $x\in \gamma(\xi)\cap A.$
Hence \rf{eqnab84} holds for all $x\in\gamma(\xi)$ in any case.
So we obtain
\begin{equation}\label{eqboot1}
\NN u(\xi) \leq \kappa\,C(M)\,\cM_{\sigma,p'}f(\xi) + \frac C{M^\alpha}\,\cM_\sigma(\NN u)(\xi)\quad \mbox{ for
all $\xi\in\pom$.}
\end{equation}
Thus, for $s>p'$,
\begin{align*}
\|\NN u\|_{L^s(\sigma)} &\leq \kappa\,C(M)\,\|\cM_{\sigma,p'}f\|_{L^s(\sigma)} + \frac C{M^\alpha}\,
\|\cM_\sigma(\NN u)\|_{L^s(\sigma)} \\ & \leq \kappa\,C'(M)\,\|f\|_{L^s(\sigma)} + 
\frac{C'}{M^\alpha}\,
\|\NN u\|_{L^s(\sigma)} 
.
\end{align*}
Since $f$ is continuous $\pom$ is bounded, $\|\NN u\|_{L^s(\sigma)} <\infty$, and hence, choosing
$M$ (and thus $\Lambda$) big enough, we get
$$\|\NN u\|_{L^s(\sigma)} \leq \kappa\,C'(M)\,\|f\|_{L^s(\sigma)}.$$

\vv

Regarding the last statement of the lemma, recall that $\cM_{\sigma,p'}$ is bounded from 
$L^{p'}(\sigma)$ to $L^{p',\infty}(\sigma)$ and that 
$\cM_\sigma$ is bounded in $L^{p',\infty}(\sigma)$. Then, from \rf{eqboot1} we infer that
\begin{align*}
\|\NN u\|_{L^{p',\infty}(\sigma)} &\leq \kappa\,C(M)\,\|\cM_{\sigma,p'}f\|_{L^{p',\infty}(\sigma)} + \frac C{M^\alpha}\,
\|\cM_\sigma(\NN u)\|_{L^{p',\infty}(\sigma)} \\
&\lesssim \kappa\,C(M)\,\|f\|_{L^{p'}(\sigma)} + \frac C{M^\alpha}\,
\|\NN u\|_{L^{p',\infty}(\sigma)}.
\end{align*}
Since $\|\NN u\|_{L^{p',\infty}(\sigma)}<\infty$, the latter gives \rf{eqrever2*} for $M$ and $\Lambda$ big enough. 
\end{proof}

\vv

\begin{theorem}\label{teo9.2}
Let $\Omega\subset\R^{n+1}$ be a domain with bounded $n$-AD-regular boundary. Given $x \in\Omega$, denote by $\omega^x$ the
harmonic measure for $\Omega$ with pole at $x$.
For $p\in (1,\infty)$, the following are equivalent:
\begin{itemize}
\item[(a)] $(D_{p'})$ is solvable for $\Omega$.

\item[(b)] The harmonic measure $\omega$ is absolutely continuous with respect to $\sigma$ and
for every ball $B$ centered in $\pom$ and for all $x\in \Omega\cap 3B\setminus 2B$ with $\diam(B)\leq 2\diam(\pom)$, it holds
$$\left(\avint_{B} \left(\frac{d\omega^{x}}{d\sigma}\right)^p\,d\sigma\right)^{1/p} \lesssim \sigma(B)^{-1}.$$

\item[(c)] The harmonic measure $\omega$ is absolutely continuous with respect to $\sigma$ and there is some $\Lambda>1$ big enough such that, for every ball $B$ centered in $\pom$ with $\diam(B)\leq 2\diam(\pom)$
and all $x \in \Lambda B$ such that $\dist(x,\pom)\geq \Lambda^{-1}r(B)$, it holds
$$
\left(\avint_{\Lambda B} \left(\frac{d\omega^{x}}{d\sigma}\right)^p\,d\sigma\right)^{1/p} \lesssim_\Lambda \sigma(B)^{-1}.
$$
\end{itemize}
\end{theorem}

 Remark that the
implication (a) $\Leftrightarrow$ (b) may  already be known,
although we have not found any reference where this is stated. See \cite{Hofmann} for some related results. On the other hand, the equivalence (a) $\Leftrightarrow$ (c) is new, as far as we know.

\begin{proof}
{\bf (a) $\boldsymbol\Rightarrow$ (b)}. By duality, it is enough to show that 
for every ball $B$ centered in $\pom$, for all $x\in \Omega\cap 3B\setminus 2B$, and all 
$f\in C_c(\pom\cap B)$,
$$\left|\int_{B} f\,d\omega^x\right| \lesssim \|f\|_{L^{p'}(\sigma)}\sigma(B)^{-1/p'}.$$
Denoting by $u$ the harmonic extension
of $f$ to $\Omega$, the preceding inequality can be rewritten as
$$
|u(x)| \lesssim \|f\|_{L^{p'}(\sigma)}\sigma(B)^{-1/p'}.
$$

To prove the latter inequality, by standard arguments (as in \rf{eqNN99}, say) and the $L^{p'}$ solvability of the Dirichlet problem, it follows that
\begin{align*}
\avint_{4B} |u|\,dm & \lesssim \avint_{CB\cap\pom} |\NN(u)|\,d\sigma
\leq\left(\avint_{CB\cap\pom} |\NN(u)|^{p'}\,d\sigma\right)^{1/p'}\\ &\lesssim \|f\|_{L^{p'}(\sigma)}\sigma(B)^{-1/p'}.
\end{align*}
By the subharmonicity of $|u|$ (extended by $0$ in $\Omega^c$) in $4B\setminus B$, we
have
$$|u(x)|\lesssim\avint_{4B} |u|\,dm\quad\mbox{for all $x\in \Omega\cap 3B\setminus 2B$.}$$
Together with the previous estimate, this implies (b).

\vv
{\bf (a) $\boldsymbol\Rightarrow$ (c)}. The arguments are almost the same as the ones in the proof of 
(a) $\Rightarrow$ (b), just replacing the condition $x\in \Omega\cap 3B\setminus 2B$  by
$x\in \Omega\cap \Lambda B$, $\dist(x,\pom)\geq \Lambda^{-1}\,r(B)$. We leave the details for the reader.
\vv

{\bf (b) $\boldsymbol\Rightarrow$ (a)}.
First we will show that there exists some $\ve>0$ such that
for any
 ball $B$ centered in $\pom$ with $\diam(B)\leq 2\diam(\pom)$ and for all $x\in \Omega\setminus 6B$,
\begin{equation}\label{eqrever10}
\left(\avint_{B} \left(\frac{d\omega^{x}}{d\sigma}\right)^{p+\ve}\,d\sigma\right)^{1/{(p+\ve)}} \lesssim \sigma(B)^{-1},
\end{equation}
To this end, notice first that,  for all $x\in \Omega\cap\partial (2B)$, by the so-called Bourgain's estimate,
$\omega^x(8B)\gtrsim 1.$
Then, for any function $f\in C_c(\pom)$, the assumption in (b) and the preceding estimate give
$$|u(x)| \leq C \, \|f\|_{L^{p'}(\sigma)}\sigma(B)^{-1/p'}\leq C \, \|f\|_{L^{p'}(\sigma)}
\frac{\omega^x(8B)}{
\sigma(B)^{1/p'}}
\quad \mbox{ for all $x\in \Omega\cap  \partial(2B)$,}$$
where, as above, $u$ is the harmonic extension of $f$ to $\Omega$.
By the maximum principle we infer that the above inequality also holds for all $y\in\Omega\setminus 2B$. By duality it follows that
$$\left(\,\avint_{B} \left(\frac{d\omega^{y}}{d\sigma}\right)^p\,d\sigma\right)^{1/p} \lesssim \frac{\omega^y(8B)}{\sigma(B)}
\quad\mbox{ for all $y\in\Omega\setminus 2B$}
.$$
So for any given ball $B_0$ centered in $\pom$ with $\diam(B_0)\leq 2\diam(\pom)$ and $y\in \Omega\setminus 6B_0$ and any ball $B'$ centered
at $1.1B_0\cap \pom$ with $r(B')\leq 2r(B_0)$, we have
$$\left(\,\avint_{B'} \left(\frac{d\omega^{y}}{d\sigma}\right)^p\,d\sigma\right)^{1/p} \lesssim \frac{\omega^y(8B')}{\sigma(B')}
.$$
By Gehring's lemma (see \cite[Theorem 6.38]{Giaquinta-Martinazzi}, for example) adapted to $n$-AD-regular sets, there exists some $\ve>0$ such that
$$\bigg(\,\avint_{B_0} \left(\frac{d\omega^{y}}{d\sigma}\right)^{p+\ve}\,d\sigma\bigg)^{1/(p+\ve)} \lesssim \frac{\omega^y(8B_0)}{\sigma(B_0)}
,$$
which yields \rf{eqrever10}.

Next we intend to apply Lemma \ref{lemrev} with $p+\ve$ in place of $p$. To this end, given $\Lambda >1$, a ball $B$ centered in $\pom$
with $\diam(B)\leq 2\diam(\pom)$, and $x\in \Lambda B$ with $\dist(x,\pom)\geq \Lambda^{-1}r(B)$,
we cover $B\cap\pom$ with a family of balls $B_i$, $i\in I_B$, with $r(B_i)=(100\Lambda)^{-1}r(B)$, so that
the balls $B_i$ are centered at $B\cap\pom$, $x\not\in 6B_i$ for any $i\in I_B$, and $\# I_B\leq C(\Lambda)$.
Applying \rf{eqrever10} to each of the balls $B_i$ and summing over $i\in I_B$, we infer
that
$$
\left(\avint_{\Lambda B} \left(\frac{d\omega^{x}}{d\sigma}\right)^{p+\ve}\,d\sigma\right)^{1/(p+\ve)}\leq C(\Lambda)\,\sigma(B)^{-1}.
$$
From Lemma \ref{lemrev} we deduce that $(D_s)$ is solvable for $s>(p+\ve)'$, and thus in particular for $s=p'$.
\vv

{\bf (c) $\boldsymbol\Rightarrow$ (b)}.
We will argue in the same way as in the proof of (a) $\boldsymbol\Rightarrow$ (b), using 
the estimate \rf{eqrever2*} instead of the solvability of $(D_{p'})$.
Again by duality, it suffices to show that 
for every ball $B$ centered in $\pom$ with $\diam(B)\leq 2\diam(\pom)$, for all $x\in \Omega\cap 3B\setminus 2B$ and all 
$f\in C_c(\pom\cap B)$, the harmonic extension $u$ of $f$ to $\Omega$ satisfies
\begin{equation}\label{eqsuf834}
|u(x)| \lesssim \|f\|_{L^{p'}(\sigma)}\sigma(B)^{-1/p'}.
\end{equation}

By standard arguments, the Kolmogorov inequality, and \rf{eqrever2*}, we have
$$\avint_{4B}|u|\,dm \lesssim \avint_{CB} \NN(u)\,d\sigma
\lesssim \|\NN(u)\|_{L^{p',\infty}(\sigma)}\,\sigma(B)^{-1/p'}
\lesssim \|f\|_{L^{p'}(\sigma)}\,\sigma(B)^{-1/p'}
.$$
Since $f$ vanishes in $\pom\setminus B$, by the subharmonicity of $|u|$ (extended by $0$ to $\Omega^c$) 
in $4B\setminus B$ we have
$$|u(x)| \lesssim\avint_{4B} |u|\,dm\quad\mbox{for all $x\in \Omega\cap 3B\setminus 2B$,}$$
which, together with the previous estimate, implies \rf{eqsuf834}.
\end{proof}
\vv

\begin{remark}\label{rem9.3}
The arguments in the above proof of (b) $\Rightarrow$ (a) show that  solvability of $(D_{p'})$ for some
$p'\in (1,\infty)$ implies  solvability of $(D_{p'-\ve})$ for some $\ve>0$.
\end{remark}

\vv

\begin{remark}\label{rem9.4}
The above theorem also holds if $\pom$ is unbounded. Indeed, the only place where the boundedness of $\pom$ is used is in Lemma
\ref{lemrev}, to ensure that $\|\NN u\|_{L^s(\mu)} <\infty$ and $\|\NN u\|_{L^{p',\infty}(\sigma)}<\infty$.
A way of circumventing this technical problem is the following. For $r>0$, consider the open set $\Omega_r:=\Omega\cap B(0,r)$. It is easy to check that $\pom_r$ is $n$-AD-regular and that an estimate such as \rf{eqrever} also holds for the harmonic measure $\omega_{\Omega_r}$,
with bounds uniform on $r$, so that $(D_s)$ is solvable for $\Omega_r$, with $s>p'$, and \rf{eqrever2*} also holds.
Given $f\in C(\pom)$ with compact support, let $r>0$ be big enough so that $\supp f\subset B(0,r)$, and let $f_r:\pom_r\to\R$
be such that $f_r=f$ in $\pom\cap B(0,r)$ and $f_r= 0$ in $\pom_r\cap\Omega$. The we apply Lemma \ref{lemrev} 
to the solution $u_r$ of the Dirichlet problem with data $f_r$ in $\Omega_r$.
 Letting $r\to\infty$, then one easily deduces that $\|\NN u\|_{L^s(\sigma)} \lesssim \kappa\|f\|_{L^s(\sigma)}$, as well as the related estimate 
 \rf{eqrever2*}. Since the case when $\pom$ is unbounded is not the main  focus of our work, we leave the details for the reader.
\end{remark}

\vv

\begin{proof}[\bf Proof of Theorem \ref{propoconverse}]
We will show that the assumptions in Theorem \ref{teo9.2} (c) hold. To this end, we will prove that for $B$, $\Lambda$, and $x \in \Omega$ as in Theorem \ref{teo9.2} (c),
\begin{equation}\label{eqclau883}
\left(\avint_{\Lambda B} \big(\cM_{\sigma,0}\, \omega^x\big)^p \,d\sigma\right)^{1/p} \lesssim_\Lambda \sigma(B)^{-1},
\end{equation}
where $\cM_{\sigma,0}$ is the truncated maximal operator defined by 
$$\cM_{\sigma,0} \,\tau(\xi) = \sup_{0<r\leq \dist(x,\pom)/4} \frac{|\tau|(B(\xi,r))}{\sigma(B(\xi,r))},$$
for any signed Radon measure $\tau$.

Given a ball $B_\xi=B(\xi,r)$, with $\xi\in\pom\cap\Lambda B$, $0<r\leq \dist(x,\pom)/4$, (so that $x\not \in 4B_\xi$), consider a smooth non-negative function $\vphi_{B_\xi}$ which equals $1$ on $B_\xi$ and vanishes away from $2B_\xi$. Then we have
\begin{align*}
\omega^x(B_\xi) & \leq \int \vphi_{B_\xi}\,d\omega^x \\ &= - \int \nabla_y G(y,x) \cdot \nabla \vphi_{B_\xi}(y)\,dm(y) \lesssim \frac1r\int_{2B_\xi} |\nabla_y G(y,x)|\,dm(y),
\end{align*}
where $G$ is the Green function for the Laplacian in $\Omega$.
By estimates analogous to the ones in \rf{eqNN99}, we derive
$$\int_{2B_\xi} |\nabla_y G(y,x)|\,dm(y) \lesssim r\,\int_{CB_\xi} \NN_0(\nabla_1 G(\cdot,x))\,d\sigma,$$
where $\nabla_1$ denotes the gradient with respect to the first variable
and $$\NN_0 u(\xi) = \sup_{y\in\gamma(\xi),|y-\xi|\leq \dist(x,\pom)/2}|u(y)|.$$
Therefore,
$$\frac{\omega^x(B(\xi,r))}{\sigma(B(\xi,r))}\lesssim \avint_{B(\xi,Cr)} \NN_0(\nabla_1 G(\cdot,x))\,d\sigma.$$
Taking the supremum over $0<r\leq \dist(x,\pom)/4$, we derive
$$\cM_{\sigma,0}\,\omega^x(\xi) \lesssim \cM_\sigma( \NN_0(\nabla_1 G(\cdot,x))(\xi).$$
Thus,
\begin{equation}\label{eqfkf4}
\|\cM_{\sigma,0}\, \omega^x\|_{L^p(\sigma)} \lesssim 
\|\cM_\sigma( \NN_0(\nabla_1 G(\cdot,x)))\|_{L^p(\sigma)} \lesssim
\|\NN_0(\nabla_1 G(\cdot,x))\|_{L^p(\sigma)}.
\end{equation}

Recall now that the Green function can be written as
\begin{equation}\label{eqgreen77}
G(y,x) = \EE(y-x) - \int_\pom \EE(y-z)\,d\omega^x(z)\quad 
\mbox{ for all $y,x\in\Omega$, $y\neq x$,}
\end{equation}
where $\EE$ is the fundamental solution of the Laplacian. See for example \cite{AHM3TV}. Then,
$$\nabla_y G(y,x) = K(y-x) - \int_\pom \nabla_y \EE(y-z)\,d\omega^x(z) = K(y-x) - \RR\omega^x(y),$$
where $\RR$ is the $n$-dimensional Riesz transform and $K$ is its kernel, i.e., $K=\nabla \EE$.
Therefore, for all $\xi\in\pom$,
$$\NN_0(\nabla_1 G(\cdot,x))(\xi) \lesssim \NN_0(K(\cdot-x))(\xi) + \NN_0(\RR\omega^x)(\xi).$$
Since $|K(y-x)|\lesssim 1/|x-\xi|^{-n}$ for all $y$ such that $|y-\xi|\leq \dist(x,\pom)/2$, we have
\begin{align}\label{eqnabv4}
\|\NN_0(\nabla_1 G(\cdot,x))\|_{L^p(\sigma)}^p&\lesssim \int \frac1{|x-\xi|^{np}}\,d\sigma(\xi) + 
\|\NN(\RR\omega^x )\|_{L^p(\sigma)}^p\\
& \lesssim \frac1{\dist(x,\pom)^{n(p-1)}} +
\|\NN(\RR\omega^x )\|_{L^p(\sigma)}^p,\nonumber
\end{align}
where we used the upper AD-regularity of $\sigma$ to estimate $\int \frac1{|x-\xi|^{np}}\,d\sigma(\xi)$ by
standard arguments (for example, splitting the domain of integration into annuli).

To estimate $\|\NN(\RR\omega^x )\|_{L^p(\sigma)}$, observe that $\RR\omega^x = \nabla (\EE*\omega^x)$.
Clearly, $u:=\EE*\omega^x$ is a harmonic function  in $\Omega$ and, moreover, from \rf{eqgreen77} it follows that $\EE*\omega^x$ extends continuously to $\pom$, since $G(\cdot,x)$ vanishes continuously on $\pom$,
as $\Omega$ is Wiener regular. 
Then, from the solvability of $(R_p)$,  we get
$$\|\NN(\RR\omega^x )\|_{L^p(\sigma)} \lesssim \|\nabla_{H,p}(\EE*\omega^x)\|_{L^p(\sigma)}.$$
Since $G(\cdot,x)$ is constantly $0$ on $\pom$, using again \rf{eqgreen77}, we have
$$\nabla_{H,p}(\EE*\omega^x) = \nabla_{H,p}(\EE (x-\cdot)).$$
Thus, 
$$\|\NN(\RR\omega^x )\|_{L^p(\sigma)} \lesssim \|\nabla_{H,p}(\EE (x-\cdot)\|_{L^p(\sigma)}.$$
We claim now that the function 
$$g_x(\xi):= \frac{C}{|x-\xi|^n},\quad \quad\xi\in\pom,$$
is a Haj\l asz upper gradient for $\EE (x-\cdot)$ for a suitable $C>0$. This is easy to check: for $\xi,\xi'\in\pom$ such that
$|\xi-\xi'|\leq \frac12|x-\xi|$, we have
\begin{align*}
|\EE(x-\xi) - \EE(x-\xi')| & = \bigg|\frac{C}{|x-\xi|^{n-1}} - \frac{C}{|x-\xi'|^{n-1}}\bigg|
\leq C \frac{|\xi-\xi'|}{|x-\xi|^n} \\ &\leq |\xi-\xi'|\,(g_x(\xi) + g_x(\xi')).
\end{align*}
In the case when $|\xi-\xi'|> \frac12\max(|x-\xi|, |x-\xi'|)$, we have
\begin{align*}
|\EE(x-\xi) - \EE(x-\xi')| & \leq 
\frac{C}{|x-\xi|^{n-1}} + \frac{C}{|x-\xi'|^{n-1}} \\
& \leq \frac{2C\,|\xi-\xi'|}{|x-\xi|^{n}} + \frac{2C\,|\xi-\xi'|}{|x-\xi'|^{n}}
\leq |\xi-\xi'|\,(g_x(\xi) + g_x(\xi')),
\end{align*}
which proves our claim and implies that $\|\nabla_{H,p}(\EE (x-\cdot))\|_{L^p(\sigma)}\leq \|g_x\|_{L^p(\sigma)}$.
Thus, 
$$\|\NN(\RR\omega^x )\|^p_{L^p(\sigma)} \lesssim \|g_x\|^p_{L^p(\sigma)}
 \lesssim 
\int \frac1{|x-\xi|^{np}}\,d\sigma(\xi) \lesssim \frac1{\dist(x,\pom)^{n(p-1)}}.$$
Together with \rf{eqfkf4} and \rf{eqnabv4}, this gives
$$\|\cM_{\sigma,0}\, \omega^x\|_{L^p(\sigma)} \lesssim  \frac1{\dist(x,\pom)^{n(p-1)/p}} \approx_\Lambda
\frac1{\sigma(B)^{1/p'}},$$
which implies \rf{eqclau883}.

It is easy to check that, in light of \rf{eqclau883}, $\omega^x$ is absolutely continuous with respect to
$\sigma$ and that \rf{eqrever} holds. Indeed, consider the functions
$$h_k = \sum_{Q\in\DD_{\sigma,k}} \frac{\omega^x(Q)}{\sigma(Q)}\,\chi_Q,$$ 
where $\DD_{\sigma,k}$ is the subfamily of the cubes from $\DD_\sigma$ with side length $2^{-k}$.
For $k$ big enough, it is immediate to check that $h_k\lesssim \cM_{\sigma,0}\,\omega^x$. Thus, the
functions $h_k$ are uniformly in $L^p(\sigma)$, and so we can extract a weakly convergent subsequence 
so that  $h_k \rightharpoonup h\in L^p(\sigma)$. It is also immediate that the sequence of measures $h_k\,\sigma$ converges weakly to
$\omega^x$. So we infer that $\omega^x = h\,\sigma$, with $h$ satisfying
$$\left(\avint_{\Lambda B} |h|^p \,d\sigma\right)^{1/p} \lesssim_\Lambda \sigma(B)^{-1},
$$ 
which proves the theorem.
\end{proof}

\vv


\section{A counterexample}\label{sec:counterexample}

In this section we show that, for $n\geq3$, there exists an AD-regular NTA (i.e, chord-arc) domain $\Omega\subset\R^{n+1}$ with connected boundary
such that, for all $p\in [1,\infty)$ and all $M>0$, one can find a Lipschitz function $f:\pom\to\R$ such that the solution of the Dirichlet
problem with boundary data $f$ satisfies
\begin{equation}
\|\NN(\nabla u)\|_{L^p(\sigma)} \geq M  \|\nabla_t f\|_{L^p(\sigma)}.
\label{eqhjk23}
\end{equation}
That is, $(\wt R_p)$ is not solvable in $L^p$ for any $p\in [1,\infty)$. 
This is in strong contrast with what happens for  domains satisfying the two-sided local John condition (like two-sided chord-arc domains, for example) for which  $(\wt R_p)$ is  solvable in $L^p$ for any $p>1$ small enough, by Theorem \ref{teomain}. It is also worth comparing this example with the chord-arc domains considered by Jerison in \cite{Je}, where he showed that for every $p>1$ there exists a chord-arc domain such that $(D_{p'})$ is not solvable and thus, neither is $(R_{p})$ nor $(\wt R_{p})$.

\vv

First we deal with the case $p\in [1,n)$.
We denote $x=(x_1,\ldots,x_{n+1})=(x',x_{n+1})$. To define the aforementioned NTA domain $\Omega\subset\R^{n+1}$,  we consider the solid truncated closed cone
$$K= \big\{x\in\R^{n+1}: |x'|^2 \leq |x_{n+1}-1|^2,\, 0\leq x_{n+1}\leq1 \big\}.$$
Observe that the vertex of $K$ is $(0,\ldots,0,1)$, and that its basis is a circle of radius $1$ contained in the horizontal hyperplane $x_{n+1}=0$ and centered at the origin.
Next we consider the open cylinder
$$C= \big\{x\in\R^{n+1}: |x'|^2 < 4,\, |x_{n+1}|< 1\big\},$$
and we set $\Omega = C\setminus K$. Observe that $K\subset \overline C$ and that
$$\pom = \partial C \cup\partial K,\qquad \{(0,\ldots,0,1)\} = \partial C \cap\partial K,$$
so that $\pom$ is connected. It is also easy to check that $\Omega$ is an AD-regular NTA domain\footnote{It is easy to see that this domain does not satisfy the two-sided local John condition since for any small ball around $(0,0,\cdots ,1)$, the vertex of the cone, one cannot find a corkscrew point in $(\overline\Omega)^c$ so that every boundary point in that ball can be connected with a ``good" (in fact, with any) curve.}.

For all $s\in (0,1)$ we consider the subset $K_s\subset \partial K$ given by
$$K_s = \{x\in \partial K: 0<|1-x_{n+1}|<s\},$$
and then we define the function $f_s:\pom\to \R$ by
$$f_s(x) = \left\{
\begin{array}{ll} 
0 & \quad \mbox{ if $x\in\partial C$,}\\ 
\dfrac{1- x_{n+1}}s  & \quad \mbox{ if $x\in K_s$,}\\
1  & \quad \mbox{ if $x\in\partial K$ and $|1-x_{n+1}|\geq s$.}
\end{array}
\right.
$$
Observe that $f_s$ is Lipschitz, $\nabla_t f_s=0$ $\sigma$-a.e.\ in $\pom\setminus K_s$, and 
$|\nabla_t f_s(x)|\leq \dfrac1s$ for $\sigma$-a.e.\  $x\in K_s$.
Thus,
$$\int |\nabla_t f_s|^p\,d\sigma \leq \frac1{s^p} \,\sigma(K_s) \lesssim s^{n-p}.$$
Therefore, for $1\leq p<n$,
$$\|\nabla_t f_s\|_{L^p(\sigma)} \to 0\quad \mbox{ as $s\to 0$.}$$
On the other hand, it is easy to check that the solution $u_s$ of the Dirichlet problem with boundary data $f_s$
converges uniformly in compact subsets of $\Omega$
to the function $u_0(x)= \omega_\Omega^x(\partial K)$ as $s\to0$, where $\omega_\Omega^x$ is the harmonic measure for $\Omega$ with pole in $x$. Analogously, $\nabla u_s$ converges  uniformly in compact subsets of $\Omega$ to $\nabla u_0$, which is a non-zero harmonic function.
So it follows that 
$\|\NN(\nabla u_s)\|_{L^p(\sigma)}$ is bounded away from $0$ uniformly on $s$.
Thus,
$$\frac{\|\NN(\nabla u_s)\|_{L^p(\sigma)}}{\|\nabla_t f_s\|_{L^p(\sigma)}} \to \infty \quad \mbox{ as $s\to 0$,}$$
which proves the existence of the functions $f,u$ satisfying \rf{eqhjk23}, by taking $s$ small enough.
\vv

Next we show how, in the case $n\geq 3$, one can modify the preceding domain so that $(\wt R_p)$
is not solvable either for $p\geq n$. {Notice first that, by 
Theorem \ref{propoconverse}, $(\wt R_p)$ implies $(D_{p'})$.
So it suffices to modify $\Omega$ so that Dirichlet problem is not solvable in $L^{s'}$ for some $s'>p'$ to ensure that
$(\wt R_p)$ does not hold.} In the case $p\geq n\geq3$, we have {$p'\leq \frac{n}{n-1} \leq \frac{3}{2}$}, and by the {extrapolation} of  solvability of the Dirichlet problem, it is enough to show that this not solvable {in $L^{s'}$, for some $s'>3/2$.}
Now the idea is to replace the bottom face of $\partial \Omega$ (which equals the bottom face of the cylinder $C$) by a suitable Lipschitz graph $\Gamma$, so that the harmonic measure does not satisfy a reverse H\"older inequality with exponent $s$. To this end, consider a graph constructed as follows.
For any $\ve>0$, let $\Gamma_0\subset \R^2$ be a curve made up by joining 
successively the points 
$(-2,-2)$, $(-\ve,-2)$, $(0,-1)$, $(\ve,-2)$, $(2,-2)$ by segments. Notice that at the point $(0,-1)$,
$\Gamma_0$ has a vertex of an angle tending to $0$ as $\ve\to0$. Next we let 
$$\Gamma = \{x\in\R^{n+1}: (x_1,x_{n+1})\in\Gamma_0\}.$$
We let $\Omega'$ be the part of the domain $\Omega$ that lies above $\Gamma$. That is, if $\Gamma$ is defined by the function $\wt \gamma:[-2,2]^n\to \R$,
we set
$\Omega' = \{x\in\Omega: x_{n+1}>\wt \gamma(x')\}.$
Now we claim that
\begin{equation}\label{eqclaim37}
\frac{d\omega_{\Omega'}}{d\sigma}\not \in   { L^{5/2}} (\sigma|_{B((0,-1),1/2)})\quad \mbox{ for $\ve>0$ small enough.}
\end{equation}
 {Assuming this for the moment, we deduce that $\omega_{\Omega'}$ does not satisfy a reverse H\"older inequality with exponent $5/2$. However, such reverse inequality is a necessary condition for $(D_{5/3})$ (see for example Proposition 2 from \cite{Hofmann}). Hence $(D_{5/3})$ is not solvable and, as explained above (choosing $s=5/2$, $s'=5/3$), this implies that $(\wt R_p)$ is not solvable for $p\geq3$.}
Further, essentially the same calculations that we did to show \rf{eqhjk23} for $\Omega$ are valid for $\Omega'$, and so the problem is not solvable either for $p\in[1,n)$.

\vv
It remains to prove \rf{eqclaim37}. To this end, we will relate 
 the harmonic measure in $\Omega'$ to the one of the planar domain
$$U_0 = \{x\in B((0,-1),1/2): x_2>\wt \gamma_0(x_1)\},$$
where $\wt \gamma_0$ is the function that defines the graph $\Gamma_0$ in $\R^2$, and using a conformal mapping to study the harmonic measure for $U_0$, we will derive \rf{eqclaim37}. See also Remark 2.1.17 in \cite{Ke} for a related argument.                      

Let us see the detailed. arguments. By translating, dilating, and rotating $U_0$, we can transform $U_0$ into
the planar domain 
$$U_1 = \{z\in B(0,1): 0<{\rm Arg}(z)<\alpha\},\quad \mbox{ with $\alpha\to 2\pi$ as $\ve\to0$.}$$
Next, by the conformal map $g_1(z)=z^{\pi/\alpha}$ we transform $U_1$ into the half disk
$\{z\in B(0,1):{\rm Im}(z)>0\}$, and then by standard arguments we find a conformal map $g_2$ of that half disk into the unit disk (for example we can transform this into the first quadrant of the complex plane by a Mobius transformation, then transform the first quadrant into the upper half space by $z\mapsto z^2$, and then apply another Mobius transformation of the upper half space into the unit disk). Anyway, all that matters about $g_2$ is that it is
smooth (in fact, analytic) in a neighborhood of $g_1(0)=0$.
From this fact, it follows that, for $0<r<1/2$ and $v:=(0,-1)$ (the middle vertex of $\wt \gamma_0$),
\begin{align*}
\omega_{U_0}^{z_0}(\partial U_0\cap B(v,r)) &= 
\omega_{U_1}^{z_1}(\partial U_1\cap B(0,2r)) = \frac1{2\pi}\,\HH^1(g_2\circ g_1(\partial U_1\cap B(0,2r))) \\
& = \frac1{2\pi}\,\HH^1(g_2([-(2r)^{\pi/\alpha},(2r)^{\pi/\alpha}])) = c r^{\pi/\alpha}
+  O(r^{2\pi/\alpha}),
\end{align*}
where $\omega_{U_0}$ and $\omega_{U_1}$ are the harmonic measures for $U_0$ and $U_1$, respectively, with poles
$z_1=g^{-1}(0)$ and $z_0$ being the corresponding point in $U_0$. 
Then, for $\{w_1,w_2\}= \partial B(v,r) \cap\partial U_0$, we have
\begin{align*}
\frac{d\omega_{U_0}^{z_0}}{d\HH^1|_{\partial U_0}}(w_1) + \frac{d\omega_{U_0}^{z_0}}{d\HH^1|_{\partial U_0}}(w_2) & = 
\frac{d(\omega_{U_0}^{z_0}(\partial U_0\cap B(v,r)))}{dr}\\ & = c \,r^{\frac\pi\alpha -1}  +  O(r^{(\frac{2\pi}\alpha)-1})
\approx r^{\frac\pi\alpha -1}.
\end{align*}
By choosing appropriately the conformal map $g_2$, we can assume that
$z_0 = (0,-3/4)$, and then by symmetry we have $\frac{d\omega_{U_0}^{z_0}}{d\HH^1|_{\partial U_0}}(w_1) = \frac{d\omega_{U_0}^{z_0}}{d\HH^1|_{\partial U_0}}(w_2)$. Thus, by the connection between the Green function $G_{U_0}(z,\xi)$ and the harmonic measure $\omega_{U_0}^{z_0}$, we have
\begin{equation}\label{eqgreen**23}
\partial_{\nu_{U_0}} G_{U_0}(z,z_0) = \frac{d\omega_{U_0}^{z_0}}{d\HH^1|_{\partial U_0}}(z) 
\approx |z-v|^{\frac\pi\alpha-1} 
\end{equation}
for $z\in B(v,1/2)\cap \partial U_0$.

To deal with the harmonic measure in $\Omega'$ we consider the function
$$h(x) = G_{U_0}((x_1,x_{n+1}),z_0),$$
where $x=(x_1,\ldots,,x_{n+1})$, $z_0$ is fixed, and we understand that $(x_1,x_{n+1}) = x_1 + i\,x_{n+1}$.
Denote $v_{\Omega'} = (0,\ldots,0, -1)$.
Notice that $h$ is harmonic in $B(v_{\Omega'},1/8)\cap \Omega'$, continuous in
$\overline{B(v_{\Omega'},1/8)\cap \Omega'}$, and it vanishes identically in 
$\overline B(v_{\Omega'},1/8)\cap \partial\Omega'$.
Hence, by the boundary Harnack principle and \rf{eqgreen**23}, for any $x \in B(v_{\Omega'},1/10)\cap \partial\Omega'$ and $y_0\in\Omega'\setminus B(v_{\Omega'},1/2)$, we have
$$\frac{\omega_{\Omega'}^{y_0}}{d\sigma}(x) = \partial_\nu G_{\Omega'}(x,y_0) \approx
\partial_\nu h(x) = \partial_{\nu_{U_0}} G_{U_0}((x_1,x_{n+1}),z_0)\approx
|x_{n+1} +1|^{\frac\pi\alpha-1}.
$$
Thus, if we let $Q_{v_{\Omega'}}\subset\R^{n+1}$ be an (Euclidean) cube centered at $v_{\Omega'}$ contained in
$B(v_{\Omega'},1/10)$ with $\ell(Q_{v_{\Omega'}})\approx 1$, we get
\begin{align*}
\int_{B(v_{\Omega'},1/10)\cap \partial\Omega'} \bigg|\frac{\omega_{\Omega'}^{y_0}}{d\sigma}\bigg|^{5/2}\,d\sigma & \gtrsim 
\int_{Q_{v_{\Omega'}}\cap \partial\Omega'} \big|x_{n+1} + 1\big|^{\frac52(\frac\pi\alpha-1)}\,d\sigma\\ 
& \approx \ell(Q_{v_{\Omega'}})^{n-1} \int_{|y +1|\leq \ell(Q_{v_{\Omega'}})} \big|y + 1\big|^{\frac52(\frac\pi\alpha-1)}\,dy=\infty
\end{align*}
for $\alpha\in (5\pi/3,2\pi)$. This concludes the proof of \rf{eqclaim37}.

\vvv

\appendix

\section{Theorem \ref{propoconverse}  for second order elliptic operators with $L^\infty$ coefficients}\label{appendix}

In this section we prove  a version of Theorem \ref{propoconverse} for second order elliptic operators with $L^\infty$ coefficients. Remark that the previous proof of this theorem in Section \ref{secconverse} does not extend to these operators.
The proof below is inspired by the suggestion we got from an anonymous referee, to whom we are very grateful. 

We say that $A=(a_{ij})_{1 \leq i,j \leq n+1}$ is an {\it elliptic} $(n+1)\times (n+1)$ matrix with measurable coefficients $a_{ij}: \om \to \R$ if    there exists $\lambda \in (0,1]$ such that  $\|A\|_{L^\infty(\om)} \leq \lambda^{-1}$ and
\begin{align*} 
 A(x)\xi \cdot \xi \geq \lambda \, |\xi|^2, \,\, \textup{for a.e.}  \, x \in \om\,\, \textup{and every} \,\,  \xi \in \R^{n+1}.
\end{align*}
We define the second order elliptic operator associated with $A$ by $L=-\div A \nabla$ and its adjoint by $L^*=-\dv A^T \nabla$, where $A^T$ is  the transpose  matrix of $A$.  We will prove the following theorem.

\vv


\vv

We say that the {\it  variational Dirichlet problem}  for $L$ is solvable in $\om$  if,  for every $f \in C_c(\pom)$ and  $\Phi \in  \dot W^{1,2}(\om)\cap C(\oom)$ satisfying $\Phi|_\pom=f$,  there exists $w \in \dot W^{1,2}(\om)$  such that
\begin{equation}\label{eq: Dirichlet problem}
(D^L_{\textup{v}})=
\begin{dcases}
Lw=-\div A\nabla w=0 &\textup{weakly in} \,\, \om \\
w-\Phi \in  Y_0^{1,2}(\om).\quad &\textup{on} \,\, \pom.
\end{dcases}
\end{equation}
This problem is uniquely solvable.  Indeed, by Lax-Milgram,  there exists $u \in Y^{1,2}_0(\om)$ such that $\int_\om A \nabla u \nabla \Psi= -\int A \nabla \Phi \nabla \Psi$  for every $\Psi \in Y^{1,2}_0(\om)$ (uniqueness follows from the maximum principle). Then,   it is easy to see that $w=u+\Phi$ solves  \eqref{eq: Dirichlet problem}.  Moreover, if $\pom$ is $n$-AD-regular, it follows that $u \in C(\oom)$ (see \cite[Theorem 6.27]{HKM} and \cite[Theorem 3]{Mar78}). In the particular case where $f \in \Lip_c(\pom)$, we have that $u \in C^\alpha_{\loc}(\oom)$ for some $\alpha \in (0,1)$.  Let us highlight that, clearly,  $w \in Y^{1.2}(\om)$, which will be useful in the arguments below. Moreover,  in view of the maximum principle and the Riesz representation theorem,  for any fixed point $x \in\om$, there exists a unique probability measure $\hm_L^x$ such that 
\begin{equation}\label{eq:elliptic measure}
w(x) = \int_{\d \om} f(\xi) \, d\omega_L^x(\xi).
\end{equation}

\vv

Following \cite{KP}, if $B_x:= B(x, c \delta_\om(x))$ for some $c\in (0, 1/2)$,  we introduce the {\it modified non-tangential maximal operator} of  function $u\in L^2_{\loc}(\Omega )$ by
\begin{equation}\label{eqmod-Nalpha}
\wt \NN_\alpha(u)(x):=\sup_{y \in \gamma_\alpha (x)} \left(\avint_{B_y} |u(z)|^2\,dz \right)^{1/2} , \,\,x\in \pom.
\end{equation}
When $L =-\Delta$,  and $u$ is harmonic in $\om$ then so is $\nabla u$ and thus it holds that   $\wt \NN(\nabla u) \approx \NN(\nabla u)$.

Adapting the definitions in Section \ref{sec:intro},  we say that {\it the Dirichlet problem is solvable in $L^p$} for $L$ (write $(D^L_{p})$ is solvable)  if there exists some constant $C_{D_p}>0$  such that, for any $f\in C_c(\pom)$,  the solution $u:\Omega\to\R$ of $(D^L_{\textup{v}})$ given by \eqref{eq: Dirichlet problem} in $\Omega$ with boundary data $f$ satisfies \eqref{eq:ntDirichlet}.

We also say that {\it the regularity problem is solvable in $L^p$} for $L$  (write $(R^L_{p})$ is solvable) if  there exists some constant $C_{R_p}>0$  such that, for any compactly supported  Lipschitz function 
$f:\pom\to\R$, the solution $u:\Omega\to\R$ of $(D^L_{\textup{v}})$ given by \eqref{eq: Dirichlet problem} in $\Omega$ with boundary data $f$ satisfies
\begin{equation}\label{eq:main-est-regL}
\|\wt \NN(\nabla u)\|_{L^p(\sigma)}   \leq 
\begin{dcases}
C_{R_p} \|f\|_{\dot W^{1,p}(\sigma)}, &\textup{if}\,\, \om \,\,\textup{is bounded or}\,\, \pom\,\,\textup{is unbounded}\\
C_{R_p} \|f\|_{ W^{1,p}(\sigma)}, &\textup{if}\,\, \om\,\,\textup{is unbounded and}\,\,\pom \,\,\textup{is bounded}.
\end{dcases}
\end{equation}

\vv

\begin{lemma}
Let $\om \subset \rrn$, $n \geq 1$,  be an open set satisfying the corkscrew condition with $n$-AD-regular boundary $\pom$,  and let $B:=B(\xi,r)$ be a ball centered at  $\xi \in\pom$ of radius $r$.  Then,  for any $u \in W^{1,2}(B) \cap C^\alpha(\overline{B})$ such that $u=0$ on $\overline{B} \cap \pom$,  it holds that
\begin{equation}\label{eq:goodPoincare}
\avint_B |u| \,dx \lec r\, \avint_B |\nabla u| \,dx.
\end{equation}
\end{lemma}

\begin{proof}
If we set $u=0$ in $\rrn \setminus \om$, then,  by (the proof of) \cite[Lemma 11]{HaM}, we  have that $u \in W^{1,2}(B) \cap C^\alpha(\overline{B})$ and so we can apply  \cite[Lemma 9.1]{Maz}  to infer that 
\begin{align}\label{eq:mazya}
\int_B |u(x)| \,dx \lec \frac{r^{n+1}}{\Cap_1(\overline{B} \cap \om^c,  2B)} \, \int_B |\nabla u| \,dx,
\end{align}
where, for any compact set $K \subset D$,  where $D \subset \rrn$ is open, we define
$$
\Cap_1(K,D):=\inf_{\substack{u \in C^\infty_c(D)\\ u \geq 1 \,\,\textup{on}\,\, K}} \int_D |\nabla u(x)|\,dx.
$$
Using \cite[Lemma 3.5]{KKST}, we have that $$\Cap_1(\overline{B} \cap \om^c, 2B) \geq \Cap_1(\overline{B} \cap \pom, 2B)\approx \HH^n_\infty(\overline{B} \cap \pom)\approx r^n,$$
 where the last comparability follows from the fact that $\pom$ is $n$-AD-regular.  Therefore,   \eqref{eq:mazya} becomes
$$
\int_B |u(x)| \,dx  \lec r\, \int_B |\nabla u| \,dx.
$$
concluding the proof of the lemma.
\end{proof}

\vv

We will prove the following theorem.
\begin{theorem}\label{th:reg-dir-elliptic}
Let $\om \subset \R^{n+1}$, $n \geq 2$, be an open set satisfying the corkscrew condition with $n$-AD regular boundary.  If $(R^L_p)$ is solvable in $\om$ then $(D^{L^*}_{p'})$  is solvable in $\om$.
\end{theorem}

\begin{proof}
We will show that the assumptions in Theorem \ref{teo9.2} (c) hold. To this end, we will prove that for $B$, $\Lambda$, and $x \in \pom$ as in Theorem \ref{teo9.2} (c),
\begin{equation}\label{eq:average-harm.measure}
\left(\avint_{\Lambda B} \big(\cM_{\sigma,0}\, \omega^x\big)^p \,d\sigma\right)^{1/p} \lesssim_\Lambda \sigma(B)^{-1},
\end{equation}
where $\cM_{\sigma,0}\,\tau$ is the truncated maximal operator of a signed Radon measure $\tau$ defined by 
$$\cM_{\sigma,0} \,\tau(\xi) = \sup_{0<r\leq \dist(x,\pom)/4K} \frac{|\tau|(B(\xi,r))}{\sigma(B(\xi,r))}.$$
 We set  $$r_x:= \dist(x,\pom)/4$$ 
and we consider a maximal collections of points $\{\xi_i\}_{i=1}^N \subset\pom\cap \Lambda B$ which are $r_x$-separated.
Then, for every $\xi \in \Lambda B \cap \pom$  and $r \in (0,r_x]$, we can find $i \in \{1, \dots, N\}$ for which
$$
B(\xi,r) \subset B(\xi_{i}, 2r_x).
$$
We fix $\xi \in \Lambda B \cap \pom$, $r \in (0,r_x]$, and $i \in \{1, \dots, N\}$.  We take $K>11$ big enough (depending only on $C_0$), so that $[B(\xi_{i},  (K-1) r_x) \setminus  B(\xi_{i}, 10 r_x)]\cap \pom \neq \varnothing$. 
Then we consider 
a compactly supported Lipschitz function $\vphi_i:\R^{n+1}\to\R$ satisfying:
\begin{itemize}
\item $\vphi_i = 1$ in $A(\xi_{i}, 9 r_x, K r_x)$.
\item $\vphi_i = 0$ in $\rrn\setminus A(\xi_{i}, 8r_x, (K+1) r_x)$.
\item $0\leq\vphi_i\leq 1$ and Lip$(\vphi_i) \leq (r_x)^{-1}$.  
\end{itemize}
From these properties and the fact that $g_i:=\frac{1}{r_x} 1_{B(\xi_{i},  (K+1) r_x)}$ is a Haj\l asz upper gradient  for $\vphi_i$, it follows easily that
\begin{equation}\label{eq:homogeneous phi}
 \|\vphi_i\|_{\dot W^{1,p}(\sigma)} \lesssim_{K} r_x^{{n}/{p}-1}.
\end{equation}
If $\om$ is unbounded and $\pom$ is bounded,  then we have that  $r_x \lec \diam(\pom)$ and thus, 
\begin{equation}\label{eq:inhomogeneous phi}
\|\vphi_i\|_{ W^{1,p}(\sigma)} :=\diam(\pom)^{-1}\| \vphi_i\|_{L^p(\pom)} + \|\vphi_i\|_{\dot W^{1,p}(\sigma)} \lesssim_{K} r_x^{{n}/{p}-1}.
\end{equation}

Let $u_i$ be the solution of the variational  Dirichlet problem \eqref{eq: Dirichlet problem}  for $L$ in $\om$ with data $\vphi_i$, which by \eqref{eq:elliptic measure}, is  given by 
$$
u_i(y)=\int_\pom \vphi_i\, d\hm_i^y.
$$

By the choice of $K$, there exists 
$$
\xi_0 \in [B(\xi_{i},  (K-1)  r_x) \setminus  B(\xi_{i}, 10 r_x)]\cap \pom,
$$
 and we define $B_0:=B(\xi_0,  r_x/4)$.  If we set $v_i:=1-u_i$, it is easy to see that $Lv_i=0$ in $4B_0$,  $v_i=0$ on $4B_0 \cap \pom$, and $0\leq v_i \leq 1$. Therefore,  by boundary H\"older continuity of $v_i$ (see \cite[Lemma 2.10]{AGMT}),  it holds that 
 $$
1-u_i(y)= v_i(y) \leq C (\delta_\om(y)/ r_x)^\alpha \sup_{2B_0 \cap \om} v_i \leq 1/2,
 $$
 for all $y \in B_0$ such that $\delta_\Omega(y) \leq (2C)^{-\alpha}  r_x$. Therefore,  we have that $u_i(y) \geq \frac{1}{2}$ in $y \in B_0$ such that $\delta_\Omega(y) \leq (2C)^{-\alpha}  r_x=:r_0$.  We define  $r_i:=|\xi_i -\xi_0| \approx  r_x$ and set $B_i:=B(\xi_i, r_i)$. Then, we can cover $\d B_i \cap \pom$ by a uniformly bounded number of balls $\wt B_k$ centered at $\d B_i \cap \pom$ with radius $r_0$ and by the same argument as before we can show that $u_i(y) \geq \frac{1}{2}$ for any $y \in \wt B_k$ such that $\delta_\Omega(y) \leq r_0$. This implies that 
 $$
 u_i(y) \geq 1/2,  \quad \textup{for any}\,\, y \in   \{y \in \d B_i \cap \om : \delta_\Omega(y) \leq r_0\}.
 $$
Since every $y \in   \{y \in \d B_i \cap \om : \delta_\Omega(y) >  r_0\}$ can be connected by a Harnack chain of balls centered at $\d B_i \cap \om$ with radii $\approx r_0$ to a point  $z \in \{y \in \d B_i \cap \om : \delta_\Omega(y) \leq r_0\}$,  then,  by Harnack's inequality,  we obtain that there exists a uniform constant $c \in (0,1/2)$ such that $u_i(y) \geq c$ for every $y \in \d B_i \cap \om$.  Hence,  since $G(x,y) \lec r_x^{1-n}$  for every $y \in \d B_i \cap \om$,  by the maximum principle, 
\begin{equation}\label{eq:max-princ}
G(y,  x) \lec \delta_\om(x)^{1-n} u_i(y) \quad \textup{for every }\,\, y \in B_i.
\end{equation}

For every $B(\xi,r) \subset B(\xi_i, 2 r_x)$,  since $u_i \in Y^{1,2}(\om) \cap C^\alpha_{\loc}(\oom)$ and vanishes in  $B(\xi,  8r)\cap\pom$,  we apply \eqref{eq:goodPoincare} and get that
\begin{align}\label{eq:ui<M(N(nablaui))}
 \avint_{B(\xi,4 r) \cap \om} u_i(y) dy & \lec r\,  \avint_{B(\xi,  4r)\cap \om} |\nabla u_i(y)| \,dy \lec r \avint_{B(\xi,  C r)} \wt{\mathcal{N}}_2(\nabla u_i)\,d\sigma  \\
& \lec r \mathcal{M}(\wt{\mathcal{N}}_2(\nabla u_i))(\xi). \notag
\end{align}
 Therefore,  by \cite[Lemma 2.6]{AGMT},  Cauchy-Schwarz,  Caccioppoli's inequality,  Moser's estimates at the boundary, \eqref{eq:max-princ}, and \eqref{eq:ui<M(N(nablaui))}, we obtain  
\begin{align*}
\frac{\hm^x(B(\xi,r))}{\sigma(B(\xi,r))} &\lec \avint_{B(\xi,  2 r)\cap \om} |\nabla G(y,x)| \,dy \lec r^{-1} \avint_{B(\xi,4r)\cap \om}  G(y,x)\\
& \lec  \delta_\om(x)^{1-n} r^{-1} \avint_{B(\xi,4r)\cap \om}   u_i \lec \delta_\om(x)^{1-n}  \mathcal{M}(\wt{\mathcal{N}}_2(\nabla u_i))(\xi).
\end{align*}
Consequently,  for any fixed $\xi \in \Lambda B$,  if we take  supremum over all $r\leq r_x$, we get  
$$
\cM_{\sigma,0} \,\hm^x(\xi) \lec \delta_\om(x)^{1-n}  \mathcal{M}(\wt{\mathcal{N}}_2(\nabla u_i))(\xi),
$$
and, by the solvability of $(R^L_p)$ in $\om$ (if $\om$ is bounded or $\pom$ unbounded) along with \eqref{eq:homogeneous phi},
\begin{align*}
\| \cM_{\sigma,0}& \,\hm^x \|_{L^p(\sigma, \Lambda B)} \lec \delta_\om(x)^{1-n} \sum_{j=1}^N  \| \mathcal{M}(\wt{\mathcal{N}}_2(\nabla u_j)) \|_{L^p(\sigma, B(\xi_j, 2r_x)) }\\
&\lec \delta_\om(x)^{1-n}  \sum_{j=1}^N \|\wt{\mathcal{N}}_2(\nabla u_j)\|_{L^p(\sigma)} \lec \delta_\om(x)^{1-n}  \sum_{j=1}^N  \|\vphi_j\|_{ \dot W^{1,p}(\sigma)}\\ 
&\lec N \delta_\om(x)^{1-n}   r_x^{{n}/{p}-1} \lec \delta_\om(x)^{-{n}/{p'}} \lec \frac{1}{\sigma(B)^{1/p'}}.
\end{align*} 
If $\om$ is unbounded with bounded  boundary, then we use \eqref{eq:inhomogeneous phi}. This readily  implies \rf{eq:average-harm.measure}.  Since Lemma \ref{lemrev} and Theorem \ref{teo9.2} hold for general elliptic operators as well with the same proof,  we may argue as in the end of the proof of Theorem \ref{propoconverse} to conclude the proof of the theorem.
\end{proof}

\vv


\end{document}